\documentclass[12 pt]{article}
\usepackage{geometry}                		
\usepackage{graphicx}				
\usepackage{amssymb}
\usepackage{comment}
\usepackage{caption}
\usepackage{amsfonts}
\usepackage{amsmath}
\usepackage{pgfplots}
\usepackage{verbatim}
\usepackage{amsthm}
\usepackage[labelformat=empty]{caption}
\usepackage{braket,amsfonts}
\usepackage{array}
\usepackage{mathabx}
\usepackage{tikz}
\usetikzlibrary{positioning}
\usetikzlibrary{decorations.markings}
\usepackage{float}
\usepackage{subcaption}
\usepackage{multirow}
\usepackage{changepage}
\usepackage{mathtools}
\usepackage{comment}
\usepackage{mathptmx}
\usepackage{placeins}
\usetikzlibrary{calc}

\newtheorem{theorem}{Theorem}[subsection]

\newtheorem{lemma}[theorem]{Lemma}
\newtheorem{definition}{Definition}[subsection]

\usepackage{lipsum}
\usepackage{adjustbox}

\title{Port-Hamiltonian Discontinuous Galerkin Finite Element Methods}
\author{N. Kumar\footnotemark[1],\quad J.J.W. van der Vegt\footnotemark[1],\quad H.J. Zwart\footnotemark[1]\ \footnotemark[2] \\ \footnotemark[1] University of Twente, The Netherlands  \\ \footnotemark[2] Eindhoven University of Technology, The Netherlands}

\begin{document}
\maketitle
\begin{abstract}
A port-Hamiltonian (pH) system formulation is a geometrical notion used to formulate conservation laws for various physical systems. The distributed parameter port-Hamiltonian formulation models infinite dimensional Hamiltonian dynamical systems that have a non-zero energy flow through the boundaries. In this paper we propose a novel framework for discontinuous Galerkin (DG) discretizations of pH-systems. Linking DG methods with pH-systems gives rise to compatible structure preserving finite element discretizations along with flexibility in terms of geometry and function spaces of the variables involved. Moreover, the port-Hamiltonian formulation makes boundary ports explicit, which makes the choice of structure and power preserving numerical fluxes easier. We state the Discontinuous Finite Element  Stokes-Dirac structure with a power preserving coupling between elements, which provides the mathematical framework for a large class of pH discontinuous Galerkin discretizations. We also provide an a priori error analysis for the port-Hamiltonian discontinuous Galerkin Finite Element Method (pH-DGFEM). The port-Hamiltonian discontinuous Galerkin finite element method is demonstrated for the scalar wave equation showing optimal rates of convergence.
\end{abstract}
\section{Introduction}
This paper defines a novel elementwise discontinuous finite element Stokes-Dirac structure, which provides the mathematical framework for elementwise discontinuous Galerkin (DG) discretizations of port-Hamiltonian (pH) systems. The pH-system formulation,  \cite{ jacob2012linear, 10.1093/imamci/dnaa018,schoberl2011first, van2002hamiltonian}, is an approach for  modeling and control of various kinds of physical systems. Infinite dimensional pH-systems with non-zero energy flow through it (spatial) boundaries are called distributed parameter pH-systems, \cite{jacob2012linear, van2002hamiltonian, van2014port}. These distributed parameter pH-systems comprise of an interconnection structure called Stokes-Dirac structure and an energy functional called Hamiltonian. The Stokes-Dirac structure gives us an effective way to express the behaviour of a system through pairs of input-output variables (whose product gives the power) at the boundary. This structural property of pH-systems makes them very suitable for the modeling and analysis of interconnected multi-physics systems. Due to the varied complexities of these multi-physics systems (inter-domain coupling, non-linearities, etc.), the interest in numerical discretization techniques for pH-systems is increasing. In particular, numerical techniques that can preserve the mathematical structure (power balance) of the system upon discretization are important.\\
 Various approaches have been followed to obtain structure preserving discretizations of port-Hamiltonian systems. Finite difference and model order reduction methods have been proposed by Lopezlena et al., \cite{lopezlena2003energy}, and Trenchant et al., \cite{trenchant2018finite}. Golo et al. \cite{golo2004hamiltonian, talasila2002wave}, proposed a method which is effective in preserving the structure of a pH-system on discretization, but this method has restrictive compatibility conditions and becomes complicated for higher dimensions.\\
Based on Discrete Exterior Calculus, \cite{hirani2003discrete}, Seslija et al. \cite{seslija2012discrete}, proposed a discrete formulation of pH conservation laws, along with a consistent approximation of the closure equations, which results in a compatible discretization. In \cite{cardoso2019partitioned}, Cardoso-Ribeiro et al., have proposed the partitioned Finite Element Method (pFEM) for pH-systems with 2 conservation laws, e.g., shallow water equations, acoustic wave equations. The pFEM for structure preserving discretizations of pH-systems solves one conservation law in weak form and the other one in strong form. Serhani et al., \cite{serhani2019partitioned}, have extended this pFEM to damped pH-systems. Kotyczka et al., \cite{ kotyczka2019numerical, kotyczka2018weak} proposed a mixed finite element discretization based on a general pH weak formulation. In this approach, the weak form of the Dirac structure on each element is defined as the basis for a mixed-Galerkin pH formulation. Using a power-preserving interconnection, a structure preserving discretization of the entire system is obtained.\\
 The consistent approximation of the closure equation in the Discrete Exterior Calculus approach and the weak form in the partitioned-FEM and mixed-FEM methods provide more flexibility in terms of consistency and stability. However, these methods are restricted in terms of mesh geometry and the flexibility in the choice of function spaces chosen for the variables in the conservation laws. In this paper, using the concept of an elementwise Dirac structure and power preserving interconnection, we move a step forward by proposing a discontinuous Galerkin (DG) method for structure preserving discretizations of pH-systems.\\ 
The flexibility in terms of mesh geometry (e.g., hanging nodes) and function spaces allowing hp-adaptive methods, which DG methods \cite{di2011mathematical,  ern2004theory, hesthaven2007nodal, riviere2008discontinuous} provide, makes them a rich environment for the discretization of various physical systems. DG methods use the basic idea of solving the PDE elementwise and then through appropriate numerical fluxes at the shared boundaries of the elements a stable discretization is ensured for the whole system. However, most DG methods are unable to preserve the mathematical or Hamiltonian structure of the PDEs (if the PDEs possesses one). Hence, linking DG methods with the concept of pH-systems has a two fold advantage. On one hand the pH formulation makes the boundary port-variables explicit, so choosing a numerical flux that preserves the mathematical structure for DG discretizations becomes easier. On the other hand DG discretizations provide a flexibile (in terms of mesh geometry and parallel computing) structure preserving discretization of pH-systems.\\
Since the power preserving interconnection of two Dirac structures again gives a Dirac structure,  \cite{cervera2007interconnection,  jacob2012linear}, the elementwise DG formulation is extended to the whole manifold using a power preserving interconnection structure. We also provide an a priori error analysis of the port-Hamiltonian DG discretization.\\
The paper is organized as follows. In order to define the mathematical framework and notation we start with a brief introduction to differential forms in Section \ref{sec:form}. In Section \ref{sec:conPH} we discuss the main aspects of distributed parameter pH-systems, including the definition of the Stokes-Dirac structure and the Hamiltonian, and then relate the port variables in the Stokes-Dirac structure to the Hamiltonian of the system. With this background we discuss in Section \ref{sec:DGFEMPH} the key results of this paper. We first define the function spaces for the port-variables of the Dirac structure. Next, we state the duality relations on the elements and their boundaries. Finally, at the end of this section, we state the definition of a generalized Stokes-Dirac structure (discontinuous finite element Stokes-Dirac structure) along with a definition of the interconnection Dirac structure (structure representing the connection between discontinuous finite elements). We state the power preserving coupling between finite elements  in Section \ref{sec:choiceinter}. In Section \ref{sec:DGPHFEM} we present the DG formulation of the discontinuous finite element Stokes-Dirac structure, together with the chosen numerical fluxes. In Section \ref{sec:Energy stability} we analyse the stability and in Section \ref{sec:error} we give an a priori error analysis of the DG discretization suitable for a large class of pH-systems. To support the theory we show some numerical results for the port-Hamiltonian discontinuous Galerkin finite element discretization of the scalar wave equation in Section \ref{sec:result}.
\section{Summary of some properties of differential forms}\label {sec:form}
In order to express the geometrical structure of port-Hamiltonian systems we will use the language of differential forms. To make the paper self contained we will summarize the main properties of differential forms and their function spaces used in this paper. For more details, we refer to \cite{  abraham2012manifolds,arnol2013mathematical,flanders1963differential,frankel2011geometry}.
\subsection{Smooth differential forms}\label {sec:smoothform}
Let $\Omega$ be an open, bounded, connected and oriented $n$-dimensional manifold in $\mathbb{R}^n$ with Lipschitz boundary $\partial\Omega$ and let $\Lambda^k(\Omega)$ be the space of smooth $k$-forms on $\Omega$. The space of $k$-forms $\Lambda^k(\Omega)$ and the space of $(n-k)$-forms $\Lambda^{n-k}(\Omega)$ have equal dimension and therefore are isomorphic. There exists a duality product between a $k$-form $\lambda$ and an $(n-k)$-form $\mu$, 
  \begin{equation}\label{eq:pair}
  \begin{aligned}
  \langle \lambda \vert \mu \rangle _\Omega = \int_\Omega \lambda \wedge \mu,
  \end{aligned}
  \end{equation}
  where $\wedge$ is the (usual) wedge product of differential forms.\\
  The continuous extension of a smooth $k$-form $\omega \in \Lambda^{k}(\Omega)$ to the boundary $\partial \Omega$ is done through the trace operator $\mathrm{tr}(\omega) \in \Lambda^{k}(\partial\Omega)$. For $\omega \in \Lambda^{n-1}(\Omega)$, $\mathrm{tr}(\omega) \in \Lambda^{n-1}(\partial \Omega)$, the generalized Stokes theorem is given by 
 \begin{equation}
  \begin{aligned}
  \int_\Omega d\omega &= \int_{\partial\Omega} \mathrm{tr}(\omega)  \label{eq:prop1c},
  \end{aligned}
  \end{equation}
  where $d$ is the exterior derivative operator for differential forms.\\
  Using the duality product \eqref{eq:pair}, the generalized Stokes theorem \eqref{eq:prop1c} and Leibniz rule \cite{abraham2012manifolds, frankel2011geometry, marsden2013introduction}, we state the integration by parts rule for smooth differential forms $\lambda \in \Lambda^k(\Omega)$ and $\mu \in \Lambda^{n-k-1}(\Omega)$,
      \begin{equation}\label{eq:intp}
  \begin{aligned}
  \langle d\lambda \vert \mu \rangle _\Omega =  \langle  \lambda \vert \mu \rangle _{\partial \Omega} + (-1)^{k-1} \langle \lambda \vert d\mu \rangle _\Omega,
    \end{aligned}
  \end{equation}
  where with an abuse of notation we write, $\langle\lambda\mid\mu\rangle_{\partial \Omega}  = \int_{\partial \Omega} \mathrm{tr} (\lambda) \wedge \mathrm{tr} (\mu) $, \cite{arnold2006finite}. 
 \subsection{Hodge Duality}\label{sec:conhodge1}
The inner product $g$ between two $k$-forms $\alpha^k, \beta^k$  is written as 
  \begin{equation}\label{eq:intro1}
 \begin{aligned}
g(\alpha^k,\beta^k) = \int_\Omega  \alpha^k(\mu) \beta^k(\mu) d\mu, 
 \end{aligned}
 \end{equation}
 where $d\mu$ is the Lebesque measure on $\Omega$ and $\alpha^k(\mu), \beta^k(\mu)$ are the coefficient functions of $\alpha^k,\beta^k$.\\
With this inner product we can construct an isomorphism between the spaces $\Lambda^k(\Omega)$ and $\Lambda^{n-k}(\Omega)$, which is denoted as $\ast : \Lambda^k \rightarrow \Lambda^{n-k}$ and called the Hodge star operator, \cite{abraham2012manifolds, flanders1963differential, frankel2011geometry}.
\begin{definition}\label{def:Hodgestar}
 Let $\Omega$ be an oriented manifold with $\dim (\Omega) = n$ and let $\Lambda^k(\Omega)$ be the space of differential $k$-forms. Let the inner product on the space of $k$ forms be given by $g : \Lambda^{k}(\Omega) \times \Lambda^{k}(\Omega) \rightarrow \mathbb{R}$. Then the Hodge star operator $\ast : \Lambda^k(\Omega) \rightarrow \Lambda^{n-k}(\Omega)$ is the linear mapping defined through 
 \begin{equation}\label{eq:Hodgestar}
 \begin{aligned}
 \int_\Omega\lambda \wedge \ast \omega = g(\omega, \lambda), \quad \quad \lambda, \ \omega \in \Lambda^{k}(\Omega).
 \end{aligned}
 \end{equation}
 \end{definition}
  If the Hodge star operator is applied twice on a $k$-form $\alpha^k \in \Lambda^k(\Omega)$ (Section 14.1, \cite{frankel2011geometry}), we have 
 \begin{equation}\label{eq:twice}
 \begin{aligned}
 \ast\ast\alpha^k = (-1)^{k(n-k)}\alpha^k.
 \end{aligned}
 \end{equation}
\subsection{Function spaces for differential forms}\label{sec:conspace}
We will now extend the above properties of differential forms to weaker smoothness conditions. For more information we refer to Section 4 of \cite{arnold2006finite} or Section 2 of \cite{arnold200finite}. The function spaces $H^s \Lambda^k(\Omega)$, with $s \in \mathbb{R}$, and $L^p \Lambda^k(\Omega)$ contain differential $k$-forms whose coefficient functions belong to the standard Sobolev space $H^s(\Omega)$ and the Lebesgue space $L^p(\Omega)$, respectively.\\
The weak exterior derivative for $\lambda \in L^2\Lambda^k(\Omega)$, is defined via the integration by parts formula \eqref{eq:intp}. Namely, for infinitely differentiable differential forms with compact support $\mu \in C^{\infty}_c\Lambda^{n-k-1}(\Omega),\ k \le n-1$, $d\lambda$ is defined to satisfy
  \begin{equation}\label{eq:sob1}
  \begin{aligned}
 \langle d\lambda \vert \mu \rangle_{\Omega} := - (-1)^k  \langle \lambda \vert d\mu \rangle_\Omega,
   \end{aligned}
  \end{equation} 
where we use the symbol $d$ for the weak exterior derivative as well. \\ 
This weak exterior derivative operator allows us to apply the exterior derivative to differential forms whose coefficients are not differentiable in the classical sense. In analogy with the definition of Sobolev spaces (Section 5.2.2, \cite{evans2010partial}), $H \Lambda^k(\Omega)$ is defined to be the space of differential forms in $ L^2 \Lambda^k(\Omega)$ with a weak exterior derivative in $ L^2 \Lambda^{k+1}(\Omega)$\ (\cite{arnold200finite}, page 17), 
\begin{equation}
\begin{aligned}
H \Lambda^k(\Omega) := \{\lambda^k \in L^2\Lambda^k(\Omega) \mid  d\lambda^k \in L^2\Lambda^{k+1}(\Omega) \}.
\end{aligned}
\end{equation}
We also extend the inner product $g$ defined in \eqref{eq:intro1} to $L^2 \Lambda^k(\Omega)$ for two $k$-forms $\alpha^k$ and $\beta^k$, Section 4, \cite{arnold200finite}
   \begin{equation}\label{eq:intro2}
 \begin{aligned}
\langle\alpha^k,\beta^k\rangle_{L^2\Lambda^k(\Omega)} = g(\alpha^k,\beta^k) = \int_\Omega  \alpha^k(\mu) \beta^k(\mu) d\mu.
 \end{aligned}
 \end{equation}
 Using \eqref{eq:intro2} and extending Definition \ref{def:Hodgestar} to $L^2 \Lambda^k(\Omega)$, Section 4, \cite{arnold200finite}, we obtain,
  \begin{equation}\label{eq:Hodgestar1}
 \begin{aligned}
 \int_\Omega\lambda \wedge \ast \omega = g(\omega, \lambda), \quad \quad \lambda, \ \omega \in L^2 \Lambda^{k}(\Omega).
 \end{aligned}
 \end{equation}
 Combining \eqref{eq:intro2} and \eqref{eq:Hodgestar1}, we finally obtain
   \begin{equation}\label{eq:intro}
 \begin{aligned}
\langle\alpha^k,\beta^k\rangle_{L^2\Lambda^k(\Omega)} = g(\alpha^k,\beta^k) = \int_\Omega \alpha^k \wedge \ast \beta^k.
 \end{aligned}
 \end{equation}
This also gives the Hilbert space $H\Lambda^{k}(\Omega)$ under the inner product, Section 4.2, \cite{arnold200finite},
   \begin{equation}\label{eq:sob2}
  \begin{aligned}
 \langle \alpha , \beta \rangle_{H\Lambda^k(\Omega)} =  \langle \alpha,\beta \rangle_{L^2 \Lambda^k(\Omega)} + \langle d\alpha,d\beta \rangle_{L^2 \Lambda^{k+1}(\Omega)}.
    \end{aligned}
  \end{equation}
  The space $H\Lambda^0(\Omega)$ coincides with $H^1\Lambda^0(\Omega)$ (or simply $H^1(\Omega)$), while the space  $H\Lambda^n(\Omega)$ coincides with $L^2\Lambda^n(\Omega)$.\\
 Since we deal with boundary controlled systems we need to extend the differential forms to domains with a boundary. Using the theory of trace operators in Sobolev spaces, we find that $ \mathrm{tr} : \Lambda^k(\Omega) \rightarrow \Lambda^k(\partial\Omega)$ extends by continuity to a mapping of $ H^1\Lambda^k(\Omega)$ onto the space $ H^{1/2}\Lambda^k(\partial\Omega)$. Next, consider the space $H \Lambda^k(\Omega)$. The trace of $\lambda \in H\Lambda^k(\Omega)$ is a bounded operator on $H\Lambda^k(\Omega)$ with values in $H^{-1/2}\Lambda^k(\partial \Omega)$. Given $\rho \in H^{1/2}\Lambda^k(\partial\Omega)$ the Hodge star of any $k$-form with respect to the boundary is given by $\bar{\ast} \rho \in H^{1/2}\Lambda^{n-k-1}(\partial\Omega)$. We can then state the integration by parts rule for $\lambda  \in H\Lambda^{k}(\Omega), \  \mu \in H^1\Lambda^{n-k-1}(\Omega)$ with $\mathrm{tr}\mu = \bar{\ast}\rho$ as
 \begin{equation}\label{eq:intp1}
 \begin{aligned}
 \langle d\lambda \vert \mu \rangle _\Omega =  \langle  \lambda \vert \mu \rangle _{\partial \Omega} - (-1)^k \langle \lambda \vert d\mu \rangle _\Omega ,
 \end{aligned}
 \end{equation}
 where the pairing at the boundary is interpreted as the duality pairing between $H^{-1/2}\Lambda^k(\partial\Omega)$ and $H^{1/2}\Lambda^{n-k-1}(\partial\Omega)$,  (\cite{arnold200finite}, page 19).
\section{Distributed port-Hamiltonian systems}\label {sec:conPH}
\subsection{Dirac structure}\label{sec:condirac}
The concept of a Dirac structure, as introduced in \cite{courant1988beyond,courant1990dirac}, is the geometrical notion formalizing general power-conserving interconnections, thereby allowing the Hamiltonian formulation of various kinds of combined systems.
The definition of a Dirac structure is as follows. Let $ \mathcal{F} $ and $\mathcal{E} $ be linear spaces, with a bilinear operation : $\mathcal{F} \times \mathcal{E} \rightarrow \mathbb{R} $. The bilinear product is denoted as $ \langle e \vert f \rangle, f \in \mathcal{F},e \in \mathcal{E}$. We call this bilinear pairing a non-degenerate duality pairing if $\langle e \vert f \rangle = 0$ for all $f \in \mathcal{F}$ implies $e = 0$ and if $\langle e \vert f \rangle = 0$ for all $e \in \mathcal{E}$ implies $f = 0$. By symmetrizing the pairing we get a symmetric bilinear pairing $\langle\langle \ , \ \rangle\rangle $ on $ \mathcal{F} \times \mathcal{E} $, with values in $\mathbb{R} $ given by 
\begin{equation}\label{eq:bilinearpair}
\begin{aligned}
 \langle \langle (f_1,e_1),(f_2,e_2) \rangle\rangle =  \langle e_1 \vert f_2 \rangle + \langle e_2 \vert f_1 \rangle,  \ \ \ \ \ \ \ (f_i,e_i) \in  \mathcal{F} \times \mathcal{E}.
 \end{aligned}
 \end{equation}
 A Dirac structure is a linear subspace $ \mathcal{D} \subset \mathcal{F} \times \mathcal{E} $ such that $\mathcal{D} = \mathcal{D}^{\perp} $ with $ {\perp}$ denoting the orthogonal complement with respect to the bilinear pairing  $\langle\langle \ , \ \rangle\rangle $, \cite{van2014port}.\\
From the above definition it follows that for any $ (f,e)$ in the Dirac structure $\mathcal{D} $ we have,
\begin{equation}\label{eq:power}
\begin{aligned}
0=\langle \langle (f,e),(f,e) \rangle\rangle =  2 \langle e \vert f \rangle. 
\end{aligned}
\end{equation}
Thus, if $(f,e)$ is a pair of power variables, where $\langle e \vert f \rangle$ equals the power, then the condition $ (f,e) \in \mathcal{D} $ implies power conservation $ \langle e \vert f \rangle =0 $.
\subsection{The Stokes-Dirac structure}\label{sec:stokesdirac}
The key concept in the treatment of the underlying geometric framework of a distributed-parameter port-Hamiltonian system with non-zero energy flow through the boundary is the introduction of a special type of Dirac structure. This structure is introduced by selecting suitable spaces of differential forms on the spatial domain and its boundary, making use of Stokes theorem.\\ 
We now define the linear space over the $n$-dimensional oriented manifold $\Omega$ 
\begin{equation}\label{eq:spacea}
\begin{aligned}
\mathcal{F}_{p,q} = L^2 \Lambda^p(\Omega)  \times  L^2 \Lambda^q(\Omega) \times  H^{1/2} \Lambda^{n-p}(\partial\Omega),
\end{aligned}
\end{equation}
for any pair $p$, $q$ of non-negative integers satisfying
\begin{equation}\label{eq:spaceb}
\begin{aligned}
 p+q = n+1,
 \end{aligned}
\end{equation}
and correspondingly we define
\begin{equation}\label{eq:spacec}
\begin{aligned}
\mathcal{E}_{p,q} = H^1\Lambda^{n-p}(\Omega)  \times H \Lambda^{n-q}(\Omega) \times   H^{-1/2}\Lambda^{n-q}(\partial\Omega).
\end{aligned}
\end{equation}
By Sections \ref{sec:smoothform}, \ref{sec:condirac} and \eqref{eq:bilinearpair} we have the symmetric non-degenerate bilinear pairing (using $L^2$ as the pivot space in the duality pairing)
\begin{equation}\label{eq:power_2}
\begin{aligned}
&\langle \langle (f_1^{p},f_1^{q},f_1^{b},e_1^{p},e_1^{q},e_1^{b}),(f_2^{p},f_2^{q},f_2^{b},e_2^{p},e_2^{q},e_2^{b}) \rangle\rangle \\ 
&=  \int_{\Omega} [e_1^{p} \wedge f_2^{p} + e_1^{q} \wedge f_2^{q} + e_2^{p} \wedge f_1^{p} + e_2^{q} \wedge f_1^{q}] + \int_{\partial \Omega} [e_1^{b} \wedge f_2^{b} +  e_2^{b} \wedge f_1^{b}] \ ,
\end{aligned}
\end{equation}
for $f_1^{p},f_2^{p} \in L^2\Lambda^p(\Omega)$, $f_1^{q},f_2^{q}  \in  L^2\Lambda^q(\Omega)$, $f_1^{b},f_2^{b}  \in  H^{1/2}\Lambda^{n-p}(\partial \Omega)$ and $e_1^{p},e_2^{p}  \in H^1\Lambda^{n-p}(\Omega)$, $e_1^{q},e_2^{q} \in H\Lambda^{n-q}(\Omega)$, $e_1^{b},e_2^{b} \in  H^{-1/2}\Lambda^{n-q}(\partial \Omega)$.\\
Note, using $H^1 \Lambda^{n-p}(\Omega) \subset L^2\Lambda^{n-p}(\Omega)$ and $H \Lambda^{n-q}(\Omega) \subset L^2\Lambda^{n-q}(\Omega)$ we have on $\Omega$ duality between $L^2\Lambda^p(\Omega)$ and $L^2\Lambda^{n-p}(\Omega)$, and between $L^2\Lambda^q(\Omega)$ and $L^2\Lambda^{n-q}(\Omega)$. At the boundary $\partial \Omega$ we have duality between $H^{1/2}\Lambda^{n-p}(\Omega)$ and $H^{-1/2}\Lambda^{n-q}(\Omega)$ using the relation $ p + q = n+1 $.\\
We define the linear subspace $\mathcal{D} $ of $\mathcal {F}_{p,q} \times \mathcal {E}_{p,q}$ as
\begin{equation}\label{eq:stodir}
\begin{aligned}
 \mathcal {D} = \Bigg\{ (f^p,f^q,f^b,e^p,e^q,e^b) \in  \mathcal {F}_{p,q} \times \mathcal {E}_{p,q} \Big \vert \left[ \begin{array}{c} f^p\\ f^q\\ \end{array} \right] =  \begin{bmatrix} 0  & (-1)^{r_1} \textit{d} \\  \textit{d} & 0 \\ \end{bmatrix} \left[ \begin{array}{c} e^p\\ e^q\\ \end{array} \right] ,\\
\left[ \begin{array}{c} f^b\\ e^b\\ \end{array} \right] =  \begin{bmatrix} 1  & 0  \\   0 & -(-1)^{n-q} \\ \end{bmatrix} \left[ \begin{array}{c} \mathrm{tr}(e^p)\\ \mathrm{tr}(e^q) \\ \end{array} \right] \Bigg \}, 
\end{aligned}
\end{equation}
 where $d$ is the exterior derivative operator and $r_1$ is defined as 
  \begin{equation}\label{eq:r}
  \begin{aligned}
  r_1 &= pq + 1.
  \end{aligned}
  \end{equation}
  For smooth differential forms, it is proven in \cite{van2002hamiltonian} that $\mathcal{D}$ is a Dirac structure. In \cite{Dirac_Sobolev} we will prove this also holds for the Sobolev spaces $\mathcal{F}_{p,q}$ and $\mathcal{E}_{p,q}$.
\subsection{Distributed-parameter port-Hamiltonian systems}\label{sec:distributed}
We consider the Hamiltonian density function $$ \mathcal{H} : L^2 {\Lambda}^p (\Omega) \times L^2 {\Lambda}^q (\Omega) \times \Omega \rightarrow  L^2 {\Lambda}^n (\Omega),$$ which results in the total Hamiltonian
$$ \textit{H}(\alpha^p,\alpha^q) = \int_\Omega \mathcal{H}(\alpha^p,\alpha^q,z)\in \mathbb{R}.$$
Following \cite{van2002hamiltonian}, if $\alpha^p, \partial \alpha^p \in L^2\Lambda^p(\Omega)$ and $\alpha^q, \partial \alpha^q \in L^2\Lambda^q(\Omega)$, then under weak smoothness conditions we have 
\begin{equation}\label{eq:wes}
\begin{aligned}
\mathit{H}(\alpha^p + \partial \alpha^p , \alpha^q + \partial \alpha^q) &= \int_\Omega \mathcal{H}(\alpha^p, \alpha^q,z) + [  \delta_p  \mathit{H} \wedge \partial \alpha^p + \delta_q \mathit{H}\wedge \partial \alpha^q] \\ 
& \ \ \ \ \ + \mathrm{ higher\ order\  terms\  in} \  \partial \alpha^p,  \partial \alpha^q,
\end{aligned}
\end{equation}
for certain differential forms 
\begin{equation}\label{eq:betadefinition}
\begin{aligned}
 \beta^p &= \delta_p\mathit{H} \in L^2{\Lambda}^{n-p} (\Omega),\\
 \beta^q&= \delta_q\mathit{H} \in  L^2{\Lambda}^{n-q} (\Omega),
 \end{aligned}
 \end{equation}
 with $(\delta_p\mathit{H}, \delta_q\mathit{H})$ the partial functional derivatives of $\mathit{H}$ \cite{marsden2013introduction}.\\
If we consider the time dependent function $(\alpha^p(t),\alpha^q(t)) \in L^2\Lambda^p(\Omega) \times L^2\Lambda^q(\Omega)$, where we assume that the coefficients are smooth functions of time, and the Hamiltonian $\textit{H}(t):=  \textit{H}(\alpha^p(t),\alpha^q(t))$ is evaluated along this trajectory, then by the chain rule
\begin{equation}\label{eq:hamilton1}
 \frac{d\textit{H}}{d\textit{t}} = \int_ \Omega \Bigg [ \beta^p \wedge \frac{\partial\alpha^{p}}{\partial t}+ \beta^q \wedge \frac{\partial\alpha^{q}}{\partial t}\Bigg ].
\end{equation}
The differential forms $(\frac{\partial \alpha^p}{\partial t}, \frac{\partial \alpha^q}{\partial t})$ are the generalized velocities of the energy variables $\alpha^p, \alpha^q$.
\begin{definition}\label{def:stokesdirac}
A \textbf{distributed-parameter port-Hamiltonian system} on an $n$-dimensional oriented manifold $\Omega$, with state space $\mathcal{F}_{p,q} \times \mathcal{E}_{p,q}$ (with $p+q = n +1$), Stokes-Dirac structure $ \mathcal{D}$ stated in \eqref{eq:stodir} and Hamiltonian \textit{H}, is given by the conservation laws 
\begin{equation}\label{eq:powbal}
\begin{aligned}
\left[ \begin{array}{c} -  \frac{\partial\alpha^{p}}{\partial t} \\ - \frac{\partial\alpha^{q}}{\partial t} \\ \end{array} \right]  =  \begin{bmatrix} 0  & (-1)^{r_1} \textit{d} \\  \textit{d} & 0 \\ \end{bmatrix} \left[ \begin{array}{c}  \beta^p \\  \beta^q \\ \end{array} \right] ,\\
\left[ \begin{array}{c} f^b\\ e^b\\ \end{array} \right] =  \begin{bmatrix} 1  & 0  \\   0 & -(-1)^{n-q} \\ \end{bmatrix} \left[ \begin{array}{c} \mathrm{tr}(\beta^p) \\ \mathrm{tr}(\beta^q) \\ \end{array} \right] ,
\end{aligned}
\end{equation}
where $r_1=pq+1$ and $\beta^p,\beta^q$ as in \eqref{eq:betadefinition}.
\end{definition}
Comparing \eqref{eq:stodir} and \eqref{eq:powbal}, we see that in \eqref{eq:stodir} the relation
 \begin{equation}\label{eq:flows}
\begin{aligned}
 \textit{f}^{\ p}  = -  \frac{\partial\alpha^{p}}{\partial t},&\qquad \textit{e}^{p}  =  \beta^p,\\ 
\textit{f}^{\ q}  = -  \frac{\partial\alpha^{q}}{\partial t}, &\qquad \textit{e}^{q}  =  \beta^q,
\end{aligned}
\end{equation}
has been substituted.\\
By the power conservation property \eqref{eq:power} of any Dirac structure it follows that for any $(f^p,f^q,f^b,{e^p},{e^q},e^b) \in \mathcal{F}_{p,q} \times \mathcal{E}_{p,q}$, in the Stokes-Dirac structure \eqref{eq:stodir}, there holds
\begin{equation}\label{eq:power_1}
 \int_\Omega [e^p \wedge f^p + e^q \wedge f^q] + \int_{\partial \Omega} e^b \wedge f^b = 0. 
\end{equation}
Using \eqref{eq:hamilton1}, \eqref{eq:flows} and \eqref{eq:power_1} we obtain
$$ \frac{d\textit{H}}{d\textit{t}} = \int_{\partial \Omega} e^b \wedge f^b,$$ which states that the increase in internal energy in the domain $\Omega$ is equal to the power supplied to the system through the boundary $\partial \Omega$.
\section{Discontinuous finite element port-Hamiltonian system}\label {sec:DGFEMPH}
In this section we extend the definition of a Stokes-Dirac structure to discontinuous finite element spaces. We start the discussion with the tessellation of an $n$-dimensional oriented manifold and later discuss the function spaces and duality relations between the port-variables. This provides us with the basic constituents to define a Stokes-Dirac structure for the discontinuous finite element setting, which we call generalized Stokes-Dirac structure.
 \subsection{Tessellation}\label {sec:dual}
A finite element can be defined as the triplet $(K, \mathcal{P}, \mathcal{N})$, where $K$ is the element domain, $\mathcal{P}$ represents the basis functions for the (energy and co-energy) variables and $\mathcal{N}$ the nodal variables.\\
 We start by introducing a tessellation $\mathcal{T}_h=\{K\}$ of an orientable manifold $\Omega$ with Lipschitz continuous boundary $\partial \Omega$ into shape-regular non-overlapping simplicial elements (e.g.: intervals, triangles, tetrahedra, etc.). Each $K \in \mathcal{T}_h$ should be a bounded open set with non-empty interior and piecewise smooth boundary. We denote the set of interior faces in the tessellation as $\mathcal{F}_i$. Thus, $\Gamma_i \in \mathcal{F}_i$ if and only if $\exists  {K}_L, {K}_R \in \mathcal{T}_h$ with $K_L \neq K_R$ such that $\Gamma_i = \bar{K}_L \cap \bar{K}_R \neq \emptyset$. Likewise, we denote the set of exterior faces in the tessellation with $\mathcal{F}_o$. Thus, $\Gamma_o \in \mathcal{F}_o$ if and only if $\exists K \in \mathcal{T}_h$ such that $\Gamma_0 = \bar{K} \cap \partial \Omega \neq \emptyset$. The set of all faces is $ \mathcal{F}_h= \mathcal{F}_i\cup\mathcal{F}_o$. In the tessellation we have $\mathcal{\dim}(K) =\mathcal{\dim}(\Omega)$, which means $K$ is an $n$-cell for an $n$-dimensional oriented manifold $\Omega$. Let $\Delta(K)$ be the set of all subelements of element $K$, e.g. the element, its faces, edges and vertices. By the definition of an orientable manifold it follows that each element $K \in \mathcal{T}_h$ inherits the orientation of the manifold $\Omega$ such that all elements have a positive Jacobian. Following this orientation all  internal faces connect to a right and a left element, whose orientation is in the opposite direction when viewed locally from the face. As we focus on simplices in this paper, we will refer to the elements $K \in \mathcal{T}_h$ as simplices and $\Delta(K)$ to be the set of all subsimplices of the simplex $K$. Note that the set of all subsimplices of simplex $K$, $\Delta(K)$, also includes the simplex $K$.
 \subsection{Function spaces for port variables}\label{sec:disspace}
With the $n$-dimensional oriented manifold $\Omega$ discretized into simplices $K \in \mathcal{T}_h$, we can extend the function spaces \eqref{eq:spacea} and \eqref{eq:spacec}  to broken function spaces for the port variables to be used in a discontinuous Galerkin discretization,
\begin{equation}\label{eq:disspace}
\begin{aligned}
\mathcal{F}_h &=   L^2 \Lambda^p(\Omega)  \times  L^2 \Lambda^q(\Omega) \times  H^{1/2} \Lambda^{n-p}(\mathcal{F}_h), \\
\mathcal{E}_h &=   H^1\Lambda^{n-p}(\mathcal{T}_h)  \times H\Lambda^{n-q}(\mathcal{T}_h) \times   H^{-1/2}\Lambda^{n-q}(\mathcal{F}_h)\ ,
\end{aligned}
\end{equation}
with the broken Sobolev spaces of differential forms 
\begin{equation}
\begin{aligned}
H^1\Lambda^k(\mathcal{T}_h) &:= \{\lambda^k_h \in L^2\Lambda^k(\Omega) \mid  \lambda^k_h\vert_{K} \in H^1\Lambda^k(K) \quad \forall K\in \mathcal{T}_h\},\\
H \Lambda^k(\mathcal{T}_h) &:= \{\lambda^k_h \in L^2\Lambda^k(\Omega) \mid  \lambda^k_h\vert_{K} \in H\Lambda^k(K) \quad \forall K\in \mathcal{T}_h\},\\
H^{1/2}\Lambda^k(\mathcal{F}_h) &:= \{\lambda^k_h \in L^2\Lambda^k(\mathcal{F}_h) \mid  \lambda^k_h\vert_{F} \in H^{1/2}\Lambda^k(F) \quad \forall F \in \mathcal{F}_h\},\\
H^{-1/2}\Lambda^k(\mathcal{F}_h) &:= \{\lambda^k_h \in L^2\Lambda^k(\mathcal{F}_h) \mid  \lambda^k_h\vert_{F} \in H^{-1/2}\Lambda^k(F) \quad \forall F\in \mathcal{F}_h\}.
\end{aligned}
\end{equation}
Before defining the discontinuous finite element function spaces for the port variables, we will briefly describe the polynomial and broken polynomial spaces for differential forms. Let $\mathcal{P}_r(\mathbb{R}^n)$ and $\mathcal{H}_r(\mathbb{R}^n)$ denote, respectively, the spaces of polynomials in $n$ variables of degree at most $r$ and homogeneous polynomials of degree $r$. For more in depth understanding of the degree and basis of these polynomial spaces we refer to Sections 3 and 4 in \cite{arnold200finite}.\\
The space of all polynomials is $\mathcal{P}(\mathbb{R}^n) = \oplus^{\infty}_{r=0} \mathcal{H}_r(\mathbb{R}^n)$. The space of polynomial differential forms $\mathcal{P}_r\Lambda^k(\mathbb{R}^n)$ is defined to be the space of differential forms whose coefficient functions are from $\mathcal{P}_r(\mathbb{R}^n)$. Similarly, the space of homogeneous polynomial differential forms $\mathcal{H}_r\Lambda^k(\mathbb{R}^n)$ is the space of differential forms whose coefficient functions are from $\mathcal{H}_r(\mathbb{R}^n)$, \cite{arnold2013spaces, arnold200finite}. For $r < 0$ the spaces of polynomial differential forms $\mathcal{P}_r\Lambda^k(\mathbb{R}^n)$ and $\mathcal{H}_r\Lambda^k(\mathbb{R}^n)$ are the zero space. For each polynomial degree $r\ge0$, a homogeneous polynomial subcomplex of the de Rham complex (\cite{arnold200finite}, page 29) is, 
\begin{equation}\label{eq:homocomplex}
\begin{aligned}
0 \longrightarrow \mathcal{H}_r\Lambda^0\stackrel{d}{\longrightarrow} \mathcal{H}_{r-1}\Lambda^1\stackrel{d}{\longrightarrow} \dotsm \stackrel{d}{\longrightarrow}  \mathcal{H}_{r-n}\Lambda^n\longrightarrow0,
\end{aligned}
\end{equation}
which is exact, where the cohomology vanishes for $r<0$ and also for $r=0$ except for the lowest degree where the cohomology space is $\mathbb{R}$. Taking the direct sum over all polynomial with degree up to $r$ of the homogeneous polynomial de Rham complex \eqref{eq:homocomplex} gives the polynomial de Rham complex
\begin{equation}\label{eq:rham}
\begin{aligned}
\mathbb{R} \hookrightarrow \mathcal{P}_r\Lambda^0 \stackrel{d}{\longrightarrow} \mathcal{P}_{r-1}\Lambda^1 \stackrel{d}{\longrightarrow} \dotsm \stackrel{d}{\longrightarrow}  \mathcal{P}_{r-n}\Lambda^n \longrightarrow 0,
\end{aligned}
\end{equation}
which is exact for $r \ge 0$ \cite{arnold2006finite, arnold200finite}.\\
Using the polynomial differential form space $\mathcal{P}_r\Lambda^k(\mathbb{R}^n)$, the homogeneous polynomial differential form space $\mathcal{H}_r\Lambda^k(\mathbb{R}^n)$ and the Koszul differential operator $\kappa$, \cite{ arnold2013spaces, arnold200finite, loday2013cyclic}, we can define a third kind of polynomial de Rham complex as 
\begin{equation}\label{eq:negcomplex}
\begin{aligned}
\mathbb{R} \hookrightarrow \mathcal{P}^-_r\Lambda^0 \stackrel{d}{\longrightarrow} \mathcal{P}^-_r\Lambda^1\stackrel{d}{\longrightarrow} ....\stackrel{d}{\longrightarrow}  \mathcal{P}^-_r\Lambda^n\longrightarrow0, 
\end{aligned}
\end{equation}
which is exact for $r > 0$ and where $\mathcal{P}^-_r\Lambda^k= \mathcal{P}_{r-1}\Lambda^k+ \kappa \mathcal{H}_{r-1}\Lambda^{k+1}$, \cite{arnold2006finite, arnold200finite}.\\
For these polynomial differential form spaces we define the broken-polynomial differential form spaces as 
\begin{equation}\label{eq:brokenpoly}
\begin{aligned}
\mathcal{H}_r\Lambda^k(\mathcal{T}_h) &:= \{\lambda^k_h \in L^2\Lambda^k(\Omega) \mid \lambda^k_h\vert_{K} \in \mathcal{H}_r\Lambda^k(K)\ \forall K\in \mathcal{T}_h \},\\
\mathcal{P}_r\Lambda^k\mathcal{T}_h) &:= \{\lambda^k_h \in L^2\Lambda^k(\Omega) \mid \lambda^k_h\vert_{K} \in \mathcal{P}_r\Lambda^k(K)\ \forall K\in \mathcal{T}_h \},\\
\mathcal{P}^-_r\Lambda^k(\mathcal{T}_h) &:= \{\lambda^k_h \in L^2\Lambda^k(\Omega) \mid \lambda^k_h\vert_{K} \in \mathcal{P}^-_r\Lambda^k(K)\ \forall K\in \mathcal{T}_h\},
\end{aligned}
\end{equation}
with $\mathcal{H}_r\Lambda^k(K)= \mathcal{H}_r\Lambda^k(\mathbb{R}^n)\vert_K$, the restriction of $\mathcal{H}_r\Lambda^k(\mathbb{R}^n)$ to element $K \in \mathcal{T}_h$, and similar expressions for $\mathcal{P}_r \Lambda^k(K)$ and $\mathcal{P}^-_r\Lambda^k(K)$.\\
For any subsimplex $f$ of $K$ we restrict the broken polynomial differential form spaces to $f$ as
\begin{equation}\label{eq:brokenpoly1}
\begin{aligned}
\mathcal{H}_r\Lambda^k(f) &= \mathrm{tr}_{K,f}\mathcal{H}_r\Lambda^k(K),\\
\mathcal{P}_r\Lambda^k(f) &=  \mathrm{tr}_{K,f}\mathcal{P}_r\Lambda^k(K),\\
\mathcal{P}^-_r\Lambda^k(f) &= \mathrm{tr}_{K,f}\mathcal{P}^-_r\Lambda^k(K),
\end{aligned}
\end{equation}
where the notation $\mathrm{tr}_{K,f}$ represents the tangential trace at the subsimplex  $f \in \Delta(K)$ of the polynomial differential forms defined on the simplex $K$. The restriction \eqref{eq:brokenpoly1} is only valid if $k \leq \mathrm{\dim}(f)\leq n$, because a $k$-form cannot be defined for subsimplices with $\mathrm{\dim}(f) < k$.\\
The restriction of the spaces on element $K$ to the subsimplices $f \in \Delta(K)$ can be used to form the following de Rham complexes on the subsimplices $f$ of simplex $K$
\begin{equation}\label{eq:rham1}
\begin{aligned}
0 \longrightarrow \mathcal{P}_r\Lambda^0(f)\stackrel{d}{\longrightarrow} \mathcal{P}_{r-1}\Lambda^1(f) \stackrel{d}{\longrightarrow} \dotsm \stackrel{d}{\longrightarrow}  \mathcal{P}_{r-\mathrm{\dim}(f)}\Lambda^{\mathrm{\dim}(f)}(f)\longrightarrow0, \ r \ge 0,
\end{aligned}
\end{equation}
and 
\begin{equation}\label{eq:negcomplex1}
\begin{aligned}
0 \longrightarrow \mathcal{P}^-_r\Lambda^0(f) \stackrel{d}{\longrightarrow} \mathcal{P}^-_r\Lambda^1(f) \stackrel{d}{\longrightarrow} ....\stackrel{d}{\longrightarrow}  \mathcal{P}^-_r\Lambda^{\mathrm{\dim}(f)}(f) \longrightarrow 0, \ r>0.
\end{aligned}
\end{equation}

Consider a simplex $K \in \mathcal{T}_h$ and let $f \in \Delta(K)$. The dual of the polynomial differential form space $\mathcal{P}_r \Lambda^k(K)$ for $0 \leq k \leq n, r \ge 1$ is denoted by $\mathcal{P}_r \Lambda^k(K)^{*}$, and can be understood in the following manner. For $f \in \Delta(K)$ with $k\le\mathrm{dim}(f)\le k+r-1$, let $\eta_f$ be an element of $\mathcal{P}^{-}_{r+k-\mathrm{\dim}(f)}\Lambda^{\mathrm{\dim}(f)-k}(f)$. Note, the highest dimension of any subsimplex in $\Delta(K)$ for an $n$-dimensional simplex $K$ is $n$, hence we also have $\mathrm{\dim}(f) \leq \mathrm{\min}(n, k+r-1)$. Then the following defines a linear functional on $\mathcal{P}_r \Lambda^k(K)$ 
\begin{equation}\label{eq:linearpro1}
\begin{aligned}
\phi(\omega) = \int_f \mathrm{tr}_{K,f} \omega \wedge \eta_f, \qquad \omega \in \mathcal{P}_r \Lambda^k(K).
\end{aligned}
\end{equation}
Hence, for every $f$ and $\eta_f$ satisfying these conditions, we can identify a dual element of $\mathcal{P}_r \Lambda^k(K)$. From Theorems 4.8 and 4.10, \cite{arnold200finite}, it follows that this forms the basis to represent all dual elements of $\mathcal{P}_r\Lambda^k(K)^{*}$ in a unique way, i.e.,
\begin{equation}\label{eq:space1}
\begin{aligned}
 \mathcal{P}_r \Lambda^k(K)^{*}\cong \underset{f \in \Delta(K),\ \mathrm{\dim}(f)\in[k,\ k+r-1]}\oplus \mathcal{P}^{-}_{r+k-\mathrm{\dim}(f)}\Lambda^{\mathrm{\dim}(f)-k}(f),
 \end{aligned}
 \end{equation}
 with $0 \leq k \leq n,\ r \ge 1$.\\
 So, any element $\mu^k_h \in \mathcal{P}_r\Lambda^k(K)^{*}$ can be written as a vector comprising of elements $\mu^k_{h,f} \in \mathcal{P}^{-}_{r+k-\mathrm{\dim}(f)}\Lambda^{\mathrm{\dim}(f)-k}(f)$ defined on the subelements $f \in \Delta(K)$, i.e.
 \begin{equation}\label{eq:vectorformulation}
 \begin{aligned}
 \mu^k_h = (\mu^k_{h,f_1}, \cdots, \mu^k_{h,f_m}), \quad f_1, \cdots, f_m \in \Delta(K).
 \end{aligned}
 \end{equation}
Using the representation \eqref{eq:vectorformulation}, for $\lambda^ k_h \in \mathcal{P}_r \Lambda^k(K)$, $\mu^k_h \in \mathcal{P}_r\Lambda^k(K)^{*}$ and \\
$\mu^k_{h,f}  \in  \mathcal{P}^{-}_{r+k-\mathrm{\dim}(f)}\Lambda^{\mathrm{\dim}(f)-k}(f)$, $0 \leq k \leq n,\ r \ge 1$, we define the duality product
\begin{equation}\label{eq:dualpro1}
\begin{aligned}
\langle \lambda^k_h \mid \mu^k_h \rangle_K= \sum_{f \in \Delta(K),\ \mathrm{\dim}(f)\in[k,\ k+r-1]}  \int_f  \mathrm{tr}_{K,f} \lambda^k_h \wedge \mu^k_{h,f}.
\end{aligned}
\end{equation}
The dual of the polynomial differential form space $\mathcal{P}^{-}_r \Lambda^k(K)$ for $0 \leq k \leq n ,r \ge1$ is denoted by $\mathcal{P}^{-}_r \Lambda^k(K)^{*}$, and is constructed in a similar manner. For each $f \in \Delta(K)$ with $k\le\mathrm{dim}(f)\le k+r-1$, let $\eta_f$ be an element of $\mathcal{P}_{r+k-\mathrm{\dim}(f)-1}\Lambda^{\mathrm{\dim}(f)-k}(f)$. Then the following defines a linear functional on $\mathcal{P}^{-}_r \Lambda^k(K)$ 
\begin{equation}\label{eq:linearpro2}
\begin{aligned}
\phi(\omega) = \int_f \mathrm{tr}_{K,f} \omega \wedge \eta_f, \qquad \omega \in \mathcal{P}^{-}_r \Lambda^k(K).
\end{aligned}
\end{equation}
Hence, for every $f$ and $\eta_f$ satisfying these conditions, we can identify a dual element of $\mathcal{P}^{-}_r \Lambda^k(K)$. From Theorem 4.14, \cite{arnold200finite}, it follows that this forms the basis to represent all dual elements of $\mathcal{P}^-_r\Lambda^k(K)^{*}$ in a unique way,
\begin{equation}\label{eq:space2}
\begin{aligned}
\mathcal{P}^{-}_r \Lambda^k(K)^{*}\cong \underset{f \in \Delta(K),\mathrm{\dim}(f)\in[k,\ k+r-1]} \oplus\mathcal{P}_{r+k-\mathrm{\dim}(f)-1}\Lambda^{\mathrm{\dim}(f)-k}(f).
 \end{aligned}
 \end{equation}
We use this identification for the dual and the duality product on $\mathcal{P}^{-}_r \Lambda^k(K)$ for the port-Hamiltonian variables $\lambda^k_h$ and $\mu^k_h$ . So similar to \eqref{eq:dualpro1}, for $\lambda^ k_h \in \mathcal{P}^{-}_r \Lambda^k(K)$, $\mu^k_h \in \mathcal{P}^{-}_r \Lambda^k(K)^{*}$ and $\mu^k_{h,f} \in  \mathcal{P}_{r+k-\mathrm{\dim}(f)-1}\Lambda^{\mathrm{\dim}(f)-k}(f)$, $0 \leq k \leq n,\ r \ge 1$, we have the duality product
\begin{equation}\label{eq:dualpro2}
\begin{aligned}
\langle \lambda^k_h \mid \mu^k_h \rangle_K = \sum_{f \in \Delta(K),\mathrm{\dim}(f)\in[k,\ k+r-1]} \int_f \mathrm{tr}_{K,f} \lambda^k_h \wedge \mu^k_{h,f},
\end{aligned}
\end{equation}
where $\mu^k_h = (\mu^k_{h,f_1}, \cdots, \mu^k_{h,f_m}), f_1, \cdots, f_m \in \Delta(K)$, as was explained in \eqref{eq:vectorformulation}.
\par
For $p+q=n+1$ we now consider two cases for the pairs of port variables.\\
\textbf{Case 1:}
\begin{equation} \label{eq:funcsp1}
\begin{aligned}
E_q(K) &:=   \mathcal{P}_{r+1}\Lambda^{n-q}(K),\ F_p(K) := \mathcal{P}_{r}\Lambda^{p}(K),
\end{aligned}
\end{equation}
These two spaces of polynomial differential forms are related via the exterior derivative operation 
\begin{equation}\label{eq:derivative1}
\begin{aligned}
d_{pq} : E_q(K) \rightarrow F_p(K),
\end{aligned}
\end{equation}
where $d_{pq}$ is the (usual) exterior derivative for differential forms, see also \eqref{eq:rham1}.\\
The dual of the differential form space $F_p$ is denoted by $E_p$ and the dual of the differential form space $E_q$ is denoted by $F_q$. Using \eqref{eq:space1}, the spaces $E_p$ and $F_q$ for Case 1 become
\begin{equation} \label{eq:funcsp2}
\begin{aligned}
E_p(K) &:=\underset{f \in \Delta(K),\mathrm{\dim}(f)\in[p,\ p+r-1]}{\oplus}\mathcal{P}^{-}_{r+p-\mathrm{\dim}(f)}\Lambda^{\mathrm{\dim}(f) -p}(f) \\
&=: \underset{f \in \Delta(K),\mathrm{\dim}(f)\in[p,\ p+r-1]}{\oplus} E_p(f),\\
F_q(K) &:= \underset{f \in \Delta(K),\mathrm{\dim}(f) \in[n-q,\ n-q+r]}{\oplus}\mathcal{P}^{-}_{r+1+n-q-\mathrm{\dim}(f)}\Lambda^{\mathrm{\dim}(f) -n+q}(f)\\
&=: \underset{f \in \Delta(K),\mathrm{\dim}(f) \in[n-q,\ n-q+r]}{\oplus} F_q(f)
\end{aligned}
\end{equation}
with $0 \leq k \leq n,\ r \ge 1$.\\
We have a relationship between $E_p$ and $F_q$ via the exterior derivative operation, that is 
\begin{equation}\label{eq:derivative2}
\begin{aligned}
d_{qp} : E_p \rightarrow F_q,
\end{aligned}
\end{equation}
where $d_{qp}$ is the (usual) exterior derivative for differential forms applied on the subsimplices $f$ of simplex $K$, see also \eqref{eq:negcomplex1}. \\
\textbf{Case 2:}
\begin{equation} \label{eq:funcsp11}
\begin{aligned}
E_q(K) &:=   \mathcal{P}^{-}_{r}\Lambda^{n-q}(K),\ F_p(K) := \mathcal{P}^{-}_{r}\Lambda^{p}(K).
\end{aligned}
\end{equation}
These two spaces of polynomial differential forms are related via the exterior derivative operation 
\begin{equation}\label{eq:derivative11}
\begin{aligned}
d_{pq} : E_q(K) \rightarrow F_p(K),
\end{aligned}
\end{equation}
where $d_{pq}$ is the (usual) exterior derivative for differential forms, see also \eqref{eq:negcomplex1}.\\
The dual of the differential form space $F_p$ is denoted by $E_p$ and the dual of the differential form space $E_q$ is denoted by $F_q$. Using \eqref{eq:space2}, the spaces $E_p$ and $F_q$ become
\begin{equation} \label{eq:funcsp12}
\begin{aligned}
E_p(K) &:=\underset{f \in \Delta(K),\mathrm{\dim}(f)\in[p,\ p+r-1]}{\oplus}\mathcal{P}_{r+p-\mathrm{\dim}(f)-1}\Lambda^{\mathrm{\dim}(f) -p}(f) \\
&=:\underset{f \in \Delta(K),\mathrm{\dim}(f)\in[p,\ p+r-1]}{\oplus}E_p(f),\\
F_q(K) &:= \underset{f \in \Delta(K),\mathrm{\dim}(f)\in[n-q,\ n-q+r-1]}{\oplus}\mathcal{P}_{r+n-q-\mathrm{\dim}(f)-1}\Lambda^{\mathrm{\dim}(f) -n+q}(f)\\
&=:\underset{f \in \Delta(K),\mathrm{\dim}(f)\in[n-q,\ n-q+r-1]}{\oplus}F_q(f).
\end{aligned}
\end{equation}
We have a relationship between $E_p$ and $F_q$ via the exterior derivative operation, that is 
\begin{equation}\label{eq:derivative22}
\begin{aligned}
d_{qp} : E_p \rightarrow F_q,
\end{aligned}
\end{equation}
where $d_{qp}$ is the (usual) exterior derivative for differential forms applied on the subsimplices $f$ of simplex $K$, see also \eqref{eq:rham1}.
\subsection{Duality product on element $K$ for discrete dual port-variable pairs }\label{sec:disdualityelement}
The port-variables have a well-defined non-degenerate duality product \eqref{eq:pair}. We extend the duality product now to the discrete dual port-variable pairs. On each subsimplex $f \in \Delta(K)$, we have a duality product between $e^p_{h,f} \in E_p(f)$ and $\mathrm{tr}_{K,f}f^p_h \in \mathrm{tr}_{K,f} F_p(K)$. Summing over all $f \in \Delta(K)$ and using the fact that $\{E_p(f),\mathrm{tr}_{K,f}F_p(K)\}$ are dual pairs we have 
\begin{equation}
\begin{aligned}
\langle e^p_h \mid f^p_h \rangle_K =  (-1)^{p(n-p)} \langle f^p_h \mid e^p_h \rangle = (-1)^{p(n-p)} \underset{f \in \Delta(K)} \sum   \int_f  \mathrm{tr}_{K,f} f^p_{h} \wedge  e^p_{h,f},
\end{aligned}
\end{equation}
where $e^p_h = (e^p_{h,f_1}, \cdots, e^p_{h,f_m}), f_1, \cdots, f_m \in \Delta(K)$, as explained in \eqref{eq:vectorformulation}.\\
Using Theorem 4.8 in \cite{arnold200finite}, we can prove that if 
\begin{equation}\label{eq:duality13}
\begin{aligned}
\langle e^p_h \mid f^p_h \rangle_K = (-1)^{p(n-p)} \underset{f \in \Delta(K)} \sum   \int_f  \mathrm{tr}_{K,f} f^p_{h} \wedge  e^p_{h,f} = 0 \quad \forall e^p_{h,f} \in E_p(f), 
\end{aligned}
\end{equation}
then $f^p_h =0$.\\
Similarly, on each subsimplex $f \in \Delta(K)$ we have a duality product between $\mathrm{tr}_{K,f} e^q_h \in \mathrm{tr}_{K,f} E_q(K)$ and $f^q_{h,f} \in F_q(f)$. Summing over all $f \in \Delta(K)$ and using the fact that $\{\mathrm{tr}_{K,f}E_q(K),F_q(f)\}$ are dual pairs we have that if
\begin{equation}\label{eq:duality3}
\begin{aligned}
\langle e^q_h \mid f^q_h \rangle_K = \underset{f \in \Delta(K)}\sum \int_f  \mathrm{tr}_{K,f}  e^q_{h} \wedge f^q_{h,f}  = 0  \qquad \forall f^q_{h,f} \in F_q(f),
\end{aligned}
\end{equation}
where $f^q_h = (f^q_{h,f_1}, \cdots, f^q_{h,f_m}), f_1, \cdots, f_m \in \Delta(K)$, as explained in \eqref{eq:vectorformulation}, then $e^q_h =0$, see Theorem 4.12 in \cite{arnold200finite}. Hence, on element $K$ there exists a non-degenerate duality product between $E_p(K)$ and $F_p(K)$ and a non-degenerate duality product between $E_q(K)$ and $F_q(K)$, given by \eqref{eq:dualpro1} and \eqref{eq:dualpro2}.
\subsection{Duality product on subelements of element $K$ for discrete dual port-variable pairs}\label{sec:disdualityboundary}
The trace operator $\mathrm{tr}_{f,\partial f}$ restricts the spaces defined on the subsimplices $f \in \Delta(K)$ to the boundaries $\partial f$. For example: A $(k-1)$-form, $\lambda_{h,f}$ defined on a subsimplex $f \in \Delta(K)$ with $k \le \dim(f) \leq n$, is restricted on the boundary $\partial f$ as 
\begin{equation}\label{eq:boundspace}
\begin{aligned}
\lambda^k_{h,f}(\partial f) = \mathrm{tr}_{f,\partial f}\lambda^k_{h,f}. 
\end{aligned}
\end{equation}
We stated the duality product over the element $K$ as the sum of duality products over the subsimplices $f \in \Delta(K)$. Similarly, we state the duality product over the boundaries of $K$ as the sum of duality products over the boundaries $ \partial f$ of $f \in \Delta (K)$, which are denoted as $\partial\Delta(K)$.\\
Considering the theory of port-Hamiltonian systems, at the boundaries $\partial \Omega$ of the oriented manifold $\Omega$ we have using the Dirac structure \eqref{eq:stodir} duality between the trace of efforts at the boundaries. Hence, at the boundaries $\partial f$ of each $f \in \Delta(K)$ we have to express the duality between $E_p(\partial \Delta(K))$ and $E_q(\partial \Delta(K))$. Since the spaces $E_p(\partial \Delta(K))$ and $E_q(\partial \Delta(K))$ will in general not have the same dimension we need to do this in two steps. First, we define a bilinear product between $E_p(\partial \Delta(K))$ and $E_q(\partial \Delta(K))$. Next, we will introduce a conversion operator $Q$ that accounts for the differences in the dimension of $E_p(\partial \Delta(K))$ and $E_q(\partial \Delta(K))$.\\ Using \eqref{eq:boundspace} and Theorems 4.15 and 4.21 in \cite{arnold200finite}, these two spaces at the boundary $\partial f$ of $f \in \Delta(K)$ are,
\begin{equation}\label{eq:bounddef}
\begin{aligned}
E_p(\partial \Delta(K)) &=\underset{f \in \Delta(K), \mathrm{\dim}(f)\in[p,r+p-1]}\oplus \mathrm{tr}_{f,\partial f} E_p(f)\\
E_q(\partial \Delta(K)) &=\underset{f \in \Delta(K), \mathrm{\dim}(f)\in[p,r+p-1]}\oplus \mathrm{tr}_{K,\partial f}E_q(K)\\
\end{aligned}
\end{equation}
Using the relation $p+q = n+1$, these spaces for Case 1 and Case 2 then become\\
\textbf{Case 1:} 
\begin{equation}
\begin{aligned}
E_p(\partial \Delta(K)) &=\underset{f\in\Delta(K), \mathrm{\dim}(f)\in[p,r+p-1]}\oplus \mathrm{tr}_{f,\partial f}\mathcal{P}^{-}_{r+p-\mathrm{\dim}(f)}\Lambda^{\mathrm{\dim}(f)-p}(f),\\
E_q(\partial \Delta(K))&=\underset{f\in\Delta(K), \mathrm{\dim}(f)\in[p,r+p-1]}\oplus \mathrm{tr}_{K,\partial f} \mathcal{P}_{r+1}\Lambda^{n-q}(K).
\end{aligned}
\end{equation}
\textbf{Case 2:}
\begin{equation}
\begin{aligned}
E_p(\partial \Delta(K)) &=\underset{f\in\Delta(K), \mathrm{\dim}(f)\in[p,r+p-1]}\oplus \mathrm{tr}_{f,\partial f}\mathcal{P}_{r+p-\mathrm{\dim}(f)-1}\Lambda^{\mathrm{\dim}(f)-p}(f),\\
E_q(\partial \Delta(K))&=\underset{f\in\Delta(K), \mathrm{\dim}(f)\in[p,r+p-1]}\oplus \mathrm{tr}_{K,\partial f} \mathcal{P}^{-}_{r}\Lambda^{n-q}(K).
\end{aligned}
\end{equation}
The bilinear product between $u^p_h \in E_p(\partial \Delta(K)), u^q_h \in E_q(\partial \Delta(K))$ is stated as 
\begin{equation}\label{eq:boudual}
\begin{aligned}
\langle u^q_h\mid u^p_h\rangle_{\partial K} &=\langle \mathrm{tr} e^q_h\mid \mathrm{tr}  e^p_h\rangle_{\partial K}\\
 &=\underset{f\in\Delta(K), \mathrm{\dim}(f)\in[p,r+p-1]}\sum\int_{\partial f}  \mathrm{tr}_{K,\partial f} e^q_h\wedge \mathrm{tr}_{f,\partial f}e^p_{h,f},
\end{aligned}
\end{equation}
with $e^q_h \in E_p(K) , e^p_h \in E_p(K), e^p_{h,f} \in E_p(f)$ and $e^p_h$ can be represented as $e^p_h = (e^p_{h,f_1}, \cdots, e^p_{h,f_m}),$ $f_1, \cdots, f_m \in \Delta(K)$ as explained in \eqref{eq:vectorformulation}.\\ 
In general the function spaces $E_p(\partial \Delta(K))$ and $E_q(\partial \Delta(K))$ do not have the same dimensions. This poses the crucial problem for the bilinear product \eqref{eq:boudual} to be 
non-degenerate. Hence, in general \eqref{eq:boudual} is a bilinear product, but not a duality product. To convert this bilinear product into a duality product we will use a linear (surjective) operator $Q$. We will discuss the general construction of the operator $Q$ in the next subsection. Note that from now on with an abuse of notation we will write $\partial \Delta(K)$ as $\partial K$.
\subsubsection{The conversion operator Q}
Let us take two finite-dimensional linear spaces $X(\partial K)$ and $Y(\partial K)$ such that $\mathrm{\dim}(X(\partial K))\ge\mathrm{\dim}(Y(\partial K))$. Suppose that, as in \eqref{eq:boudual}, we have the bilinear product between $X(\partial K)$ and $Y(\partial K)$, $\langle x \mid y \rangle_{\partial K}$, $x \in X(\partial K)$ and $y \in Y(\partial K)$.\\
For fixed $x \in X(\partial K)$, $q_x(y) = \langle x \mid y \rangle_{\partial K}$, is a linear map from $Y(\partial K)$ to $\mathbb{R}$. Thus, there exists a unique ${y}^{\ast} \in Y^{\ast}(\partial K)$, ($Y^{\ast}(\partial K)$ is the dual function space to $Y(\partial K)$) such that $q_x(y) =  \langle y^{\ast} \mid y\rangle_{{Y(\partial K)}^{\ast} \times Y(\partial K)}$.\\
So, to $x \in X(\partial K)$ we have uniquely assigned a ${y}^{\ast} \in Y^{\ast}(\partial K)$. This we can do for all $x\in X(\partial K)$, and we define the operator $Q$,
\begin{equation}\label{eq:operatordual}
\begin{aligned}
Q : X(\partial K) \rightarrow Y(\partial K)^{\ast},
\end{aligned}
\end{equation}
such that $Q(x) = y^{\ast}$, and thus 
\begin{equation}\label{eq:operatordual1}
\begin{aligned}
q_x(y) = \langle x \mid y \rangle_{\partial K} = \langle Q(x) \mid y \rangle_{{Y(\partial K)}^{\ast} \times Y(\partial K)}.
\end{aligned}
\end{equation}
Since, ${y}^{\ast}$ is unique the operator $Q$ is well-defined. Also, for all $y \in Y(\partial K)$ 
\begin{equation}
\begin{aligned}
q_{\alpha x + \beta \tilde{x}}(y)&= \langle \alpha x + \beta \tilde{x} \mid y\rangle_{\partial K}\\
&=\alpha \langle x\mid y\rangle_{\partial K} + \beta \langle \tilde{x} \mid y\rangle_{\partial K}\\
&=\alpha \langle Q(x) \mid y \rangle_{{Y(\partial K)}^{\ast} \times Y(\partial K)} + \beta \langle  Q(\tilde{x}) \mid y\rangle_{{Y(\partial K)}^{\ast} \times Y(\partial K)}.
\end{aligned}
\end{equation}
On the other hand,
\begin{equation}
\begin{aligned}
q_{\alpha x + \beta \tilde{x}}(y)&= \langle Q(\alpha x + \beta \tilde{x}) \mid y \rangle_{{Y(\partial K)}^{\ast} \times Y(\partial K)}.
\end{aligned}
\end{equation}
Thus, $Q$ is a linear operator.
\begin{definition}
The bilinear product $\langle x \mid y \rangle_{\partial K}$ is defined to be half-degenerate if for any $y \in Y(\partial K)$ the assertion: $\langle x \mid y \rangle_{\partial K} = 0 \ \forall x \in X(\partial K)$ implies that $y=0$.
\end{definition}
\begin{lemma}\label{lem:operatorq}
The operator $Q$ of \eqref{eq:operatordual1} is surjective if and only if the bilinear product $\langle x \mid y \rangle_{\partial K}$ is half-degenerate.
\end{lemma}
\begin{proof}
Assume that $Q$ is surjective and let $y_0 \in Y(\partial K)$ be such that $\langle x \mid y_0 \rangle_{\partial K} = 0, \quad \forall x \in X(\partial K)$. This gives that $\langle y^{\ast} \mid y_0\rangle_{{Y(\partial K)}^{\ast} \times Y(\partial K)}=0, \ \forall y^{\ast} \in {Y(\partial K)}^{\ast}$, which implies $y_0=0$.\\
Now we prove the other implication. Assume that $Q$ is not surjective. Since the range of $Q$ is a linear subspace of ${Y(\partial K)}^{\ast}$ and $Y(\partial K)$ has the dual ${Y(\partial K)}^{\ast}$, there exists a non-zero $y_0 \in Y(\partial K)$ such that $\langle Q(x)\mid y_0\rangle_{{Y(\partial K)}^{\ast} \times Y(\partial K)}=0, \forall x \in X(\partial K)$. By \eqref{eq:operatordual1} we see that $\langle x \mid y_0\rangle_{\partial K} =0, \forall x \in X(\partial K)$ for a nonzero $y_0$, providing the contradiction.
\end{proof} 
As the operator $Q$ converts the bilinear product on $X(\partial K) \times Y(\partial K)$ into a duality product on ${Y(\partial K)}^{\ast} \times Y(\partial K)$, we can define dual operators with respect to the bilinear form $\langle \cdot \mid \cdot \rangle_{\partial K}$. Let $S$ be a linear operator, that maps the space $X(K)$ to the space $X(\partial K)$, i.e., $S : X(K) \rightarrow X(\partial K)$. So for $x = S(x_h)$, with $x_h \in X(K)$, and $y \in Y(\partial K)$, we have 
\begin{equation}
\begin{aligned}
\langle x \mid y \rangle_{\partial K} = \langle S(x_h) \mid y \rangle_{\partial K}.
\end{aligned}
\end{equation}
Using the operator $Q$, \eqref{eq:operatordual1},
\begin{equation}\label{eq:conjugatep}
\begin{aligned}
\langle S(x_h) \mid y \rangle_{\partial K} &=\langle QS(x_h) \mid y \rangle_{{Y(\partial K)}^{\ast} \times Y(\partial K)}\\
&=\langle x_h \mid (QS)^{\ast}(y)  \rangle_{X(K) \times {X(K)}^{\ast}}.
\end{aligned}
\end{equation}
Here, $(QS)^{\ast} : Y(\partial K) \rightarrow {X(K)}^{\ast}$, denotes the (standard) dual of $QS$.\\
Similarly, let $T$ be an operator that maps the space $Y(K)$ to the space $Y(\partial K)$, i.e.,\\
$T : Y(K) \rightarrow Y(\partial K)$. So for $y = T(y_h)$, with $y_h \in Y(K)$, and $x \in X(\partial K)$, we have
\begin{equation}\label{eq:conjugateq}
\begin{aligned}
\langle x \mid y \rangle_{\partial K} = \langle x  \mid T(y_h) \rangle_{\partial K}.
\end{aligned}
\end{equation}
Using the operator $Q$, \eqref{eq:operatordual1},
\begin{equation}
\begin{aligned}
\langle x  \mid T(y_h) \rangle_{\partial K} &= \langle Q(x) \mid T(y_h) \rangle_{{Y(\partial K)}^{\ast} \times Y(\partial K)}\\
&=\langle T^{\ast}Q(x) \mid y_h \rangle_{{Y(K)}^{\ast} \times Y(K)}.
\end{aligned}
\end{equation}
Here, $(T^{\ast}Q) : X(\partial K) \rightarrow {Y(K)}^{\ast}$.
\subsubsection{Duality on the boundaries of $K$}
We assume that $\mathrm{\dim}(E_q(\partial K))\ge\mathrm{\dim}(E_p(\partial K))$. Note that this is just an assumption it can very well happen that $\mathrm{\dim}(E_q(\partial K))\le\mathrm{\dim}(E_p(\partial K))$, but the analysis is similar for this case.\\
On $E_q(\partial K) \times E_p(\partial K)$ we have the bilinear product \eqref{eq:boudual}. Let $Q : E_q(\partial K) \rightarrow {E_p(\partial K)}^{\ast}$ be the mapping associated to this product, see \eqref{eq:operatordual} and \eqref{eq:operatordual1}.\\
The trace operator $\mathrm{tr}_p$ is the map $\mathrm{tr}_p : E_p(K)\rightarrow E_p(\partial K)$ and the trace operator $\mathrm{tr}_q$ is the map $\mathrm{tr}_q : E_q(K)\rightarrow E_q(\partial K)$. For port variables $e^p_h \in E_p(K)$ and $e^q_h \in E_q(K)$, we have \begin{equation}\label{eq:boudual1}
\begin{aligned}
\langle \mathrm{tr}_qe^q_h \mid \mathrm{tr}_pe^p_h\rangle_{\partial K} &= \langle Q(\mathrm{tr}_qe^q_h) \mid \mathrm{tr}_pe^p_h\rangle_{{E_p(\partial K)}^{\ast} \times E_p(\partial K)}\\
&= \langle \mathrm{tr}^{\ast}_p(Q(\mathrm{tr}_qe^q_h)) \mid e^p_h\rangle_{{E_p(K)}^{\ast} \times E_p(K)}.
\end{aligned}
\end{equation}
Here, $ (\mathrm{tr}^{\ast}_p \circ Q) : E_q(\partial K) \rightarrow E_p(K)^{\ast}$.\\
Similarly, 
\begin{equation}\label{eq:boudual2}
\begin{aligned}
\langle \mathrm{tr}_qe^q_h \mid \mathrm{tr}_pe^p_h\rangle_{\partial K} &= \langle Q(\mathrm{tr}_qe^q_h) \mid \mathrm{tr}_pe^p_h\rangle_{{E_p(\partial K)}^{\ast} \times E_p(\partial K)}\\
&= \langle e^q_h \mid (Q\mathrm{tr}_q)^{\ast} (\mathrm{tr}_pe^p_h)\rangle_{E_q(K) \times {E_q(K)}^{\ast}}.
\end{aligned}
\end{equation}
Here, $ (Q \circ \mathrm{tr}_q)^{\ast} : E_p(\partial K) \rightarrow E_q(K)^{\ast}$.
\par
For $\mathrm{tr}_pe^p_h \in E_p(\partial K)$ and $\mathrm{tr}_qe^q_h \in E_q(\partial K)$ with $\mathrm{\dim}(E_q(\partial K))\ge\mathrm{\dim}(E_p(\partial K))$ the bilinear product $\langle \mathrm{tr}_qe^q_h \mid \mathrm{tr}_pe^p_h\rangle_{\partial K}$ is half-degenerate. So, using Lemma \ref{lem:operatorq} the operator $Q$ is surjective.
\begin{lemma}
For $e^p_h, \nu^p_h \in E_p(K), \nu^p_{h,f} \in E_p(f)$ and $e^q_h ,\nu^q_h \in E_q(K)$ with $p +q = n +1$, the integration by parts rule \eqref{eq:intp1} for an element $K$ gives
 \begin{equation}
\label{eq:int1}
\langle  d_{pq} e^q_h \mid \nu^p_h \rangle_K =(-1)^p\langle e^q_h\mid d_{qp}\nu^p_h\rangle_K + \langle \mathrm{tr}_qe^q_h \mid \mathrm{tr}_p\nu^p_h \rangle _{\partial K}.
\end{equation}
\end{lemma}
\begin{proof}
For $e^q_h \in E_q(K), \nu^p_h \in E_p(K)$, using the exterior derivative operator $d_{pq}(e^q_h(K)) \in F_p(K)$ as the exterior derivative operator operator $d_{pq}$ is defined as the map $d_{pq} : E_q(K) \rightarrow F_p(K)$. Using the duality product between $E_p(K)$ and $F_p(K)$, \eqref{eq:dualpro1} or \eqref{eq:dualpro2}, we write 
\begin{equation}
\begin{aligned}
\langle d_{pq} e^q_h \mid \nu^p_h \rangle_K &= \underset{f\in\Delta(K),p\leq\mathrm{\dim}(f)\leq r+p-1}\sum\int_f  \mathrm{tr}_{K,f} (de^q_h) \wedge \nu^p_{h,f},
\end{aligned}
\end{equation}
where $\nu^p_h$ can be represented as $\nu^p_h = (\nu^p_{h,f_1}, \cdots, \nu^p_{h,f_m}), f_1, \cdots, f_m \in \Delta(K)$, as explained in \eqref{eq:vectorformulation}.\\
As the exterior derivative operator $d$ and the trace operator $\mathrm{tr}$ are commutative, we can simplify the above as 
\begin{equation}
\begin{aligned}
\langle    d_{pq} e^q_h \mid \nu^p_h\rangle_K &= \underset{f\in\Delta(K),p\leq\mathrm{\dim}(f)\leq r+p-1}\sum\int_f  d(\mathrm{tr}_{K,f}e^q_h) \wedge \nu^p_{h,f}.
\end{aligned}
\end{equation}
Using \eqref{eq:intp1} and $E_q(K) \subset \Lambda^{n-k}(K)$ in the first step and then \eqref{eq:dualpro1}, \eqref{eq:boudual} we obtain
\begin{equation}
\begin{aligned}
\langle  d_{pq} e^q_h \mid \nu^p_h \rangle_K &= \underset{f\in\Delta(K),p\leq\mathrm{\dim}(f)\leq r+p-1}\sum\Bigg((-1)^p\Big(\int_f  \mathrm{tr}_{K,f} e^q_h\wedge d\nu^p_{h,f}\Big)\\
& \ \ \ \ \ \ \ \ \ \ \ \ \ \ \ \ \ \ \ \ \ \ \ \ \ \ \ \ \ \ \ \ \ \ \ \ \ \ \ \ \ \ \ \ \ \  + \int_{\partial f} \mathrm{tr}_{K,\partial f} e^q_h\wedge \mathrm{tr}_{f,\partial f}\nu^p_{h,f}\Bigg)\\
&=(-1)^p\langle e^q_h\mid d_{qp}\nu^p_h\rangle_K + \langle \mathrm{tr}_qe^q_h \mid \mathrm{tr}_p\nu^p_h \rangle _{\partial K}.
\end{aligned}
\end{equation}
\end{proof}
Using the duality pairings \eqref{eq:dualpro1}, \eqref{eq:dualpro2} and \eqref{eq:boudual}, we define a symmetric bilinear product for discrete port variables.
\begin{definition}[Extended bilinear form]
Let $K$ be an element of $\mathcal{T}_h$. For $\{f^p_h,f'^p_h\}\in F_p(K)$, 
$\{f^q_h,f'^q_h\}  \in F_q(K)$, $\{e^p_h,e'^p_h\} \in E_p(K)$,\ $\{e^q_h,e'^q_h\} \in E_q(K)$, the symmetric extended bilinear form is defined as 
\begin{equation}\label{eq:dispower}
\begin{aligned}
&\langle \langle (f^p_h,f^q_h,e^p_h,e^q_h), (f'^p_h,f'^q_h,e'^p_h,e'^q_h)\rangle\rangle_h \\ 
&=  \langle e'^p_h \mid f^p_h\rangle_K + \langle e^p_h \mid f'^p_h \rangle_K +\langle e'^q_h \mid f^q_h\rangle_K  + \langle e^q_h \mid f'^q_h \rangle_K \\
&\ \ \ \ \ \ + \langle \mathrm{tr}e'^q_h \mid \mathrm{tr}e^p_h\rangle_{\partial K} +\langle \mathrm{tr}e^q_h \mid \mathrm{tr}e'^p_h\rangle_{\partial K} + \langle \mathrm{tr}e'^p_h \mid \mathrm{tr}e^q_h\rangle_{\partial K} +\langle \mathrm{tr}e^p_h \mid \mathrm{tr}e'^q_h\rangle_{\partial K}.
\end{aligned}
\end{equation}
\end{definition}
\subsection{Discontinuous Finite Element Dirac structure}\label{sec:dfed}
In the discontinuous finite element discretization we consider elements to be independent from each other. The interaction occurs only through common boundaries. We define now a Dirac structure, called generalized Stokes-Dirac structure that is suitable for a discontinuous finite element discretization. We also state the interconnection between adjacent elements in the spatial domain using a Dirac structure at the element boundaries, called interconnection Dirac structure. This will provide the mathematical framework for defining port-Hamiltonian discontinuous Galerkin discretizations.
\subsubsection{The generalized Stokes-Dirac structure}
The polynomial differential form spaces on which the port-variables are projected are denoted as 
\begin{equation}\label{eq:polydifspace}
\begin{aligned}
F_{p,q} &:= F_p(K) \times F_q(K) \times E_p(\partial K) \times E_q(\partial K),\\
E_{p,q} &:= E_p(K) \times E_q(K) \times E_q(\partial K) \times E_p(\partial K).
\end{aligned}
\end{equation}
For port variables $e^p_h \in E_p(K), e^q_h \in E_q(K), f^p_h \in F_p(K), f^q_h \in F_q(K)$ and the input and output port boundary pairs  $\{y^p_h,u^p_h\} \in E_p(\partial K), \{y^q_h,u^q_h\} \in E_q(\partial K)$, we have the following bilinear product 
\begin{equation}\label{eq:power1}
\begin{aligned}
&\langle e^p_h \mid f^p_h \rangle_K + \langle e^q_h \mid f^q_h \rangle_K + \langle u^q_h \mid y^p_h \rangle_{\partial K} + \langle   y^q_h \mid  u^p_h \rangle _{\partial K}.
\end{aligned}
\end{equation}
However since, $\mathrm{\dim}(E_p(\partial K)) \neq \mathrm{\dim}(E_q(\partial K))$, the bilinear product \eqref{eq:power1} is degenerated. As explained in subsection 4.3.4 we use the linear (surjective) operator $Q$, \eqref{eq:operatordual1} to transform this bilinear product into a non-degenerate duality product. Recall that we assume $\mathrm{\dim}(E_q(\partial K))\ge\mathrm{\dim}(E_p(\partial K))$. Using the abstract operator $Q$, \eqref{eq:operatordual1}, we can map the boundary port variable space $E_q(\partial K)$ to the dual of $E_p(\partial K)$, while maintaining the bilinear product. That is, for each $u^q_h \in E_q(\partial K)$, we have an element ${u^p_h}^{\ast} \in {E_p(\partial K)}^{\ast}$, such that $Q(u^q_h) = {u^p_h}^{\ast}$ and $\langle u^q_h \mid y^p_h \rangle_{\partial K} = \langle Q(u^q_h) \mid y^p_h \rangle_{{E_p(\partial K)}^{\ast}\times E_p(\partial K)}$. Using this we can rewrite the bilinear product \eqref{eq:power1} as 
 \begin{equation}\label{eq:power11}
\begin{aligned}
&\langle e^p_h \mid f^p_h \rangle_K + \langle e^q_h \mid f^q_h \rangle_K + \langle  u^q_h \mid y^p_h \rangle_{\partial K} + \langle   y^q_h \mid  u^p_h \rangle _{\partial K}\\
&= \langle e^p_h \mid f^p_h \rangle_K + \langle e^q_h \mid f^q_h \rangle_K + \langle Q(u^q_h) \mid y^p_h \rangle_{{E_p(\partial K)}^{\ast} \times E_p(\partial K)}\\
&\ \ \ \ + \langle   Q(y^q_h)  \mid  u^p_h \rangle _{{E_p(\partial K)}^{\ast} \times E_p(\partial K)}\\
&= \langle e^p_h \mid f^p_h \rangle_K + \langle e^q_h \mid f^q_h \rangle_K + \langle {u^p_h}^{\ast} \mid y^p_h \rangle_{{E_p(\partial K)}^{\ast} \times E_p(\partial K)}\\
&\ \ \ \ + \langle   {y^p_h}^{\ast} \mid  u^p_h \rangle _{{E_p(\partial K)}^{\ast} \times E_p(\partial K)}.
\end{aligned}
\end{equation}
We replace the polynomial differential form spaces, \eqref{eq:polydifspace}, with the (abstract) spaces
\begin{equation}
\begin{aligned}
F'_{p,q} &:= F_p(K) \times F_q(K) \times E_p(\partial K) \times {E_p(\partial K)}^{\ast},\\
E'_{p,q} &:= E_p(K) \times E_q(K) \times {E_p(\partial K)}^{\ast} \times E_p(\partial K).
\end{aligned}
\end{equation} 
\begin{definition}
For the port variables $e^p_h \in E_p(K), e^q_h \in E_q(K), f^p_h \in F_p(K), f^q_h \in F_q(K)$, the input and output port boundary pairs  $\{y^p_h,u^p_h\} \in E_p(\partial K), \ \{{y^p_h}^{\ast},{u^p_h}^{\ast}\} \in {E_p(\partial K)}^{\ast}$, we have the following duality product on $F'_{p,q} \times E'_{p,q}$ 
 \begin{equation}\label{eq:power12}
\begin{aligned}
&\langle e^p_h \mid f^p_h \rangle_K + \langle e^q_h \mid f^q_h \rangle_K + \langle {y^p_h}^{\ast} \mid u^p_h \rangle_{{E_p(\partial K)}^{\ast} \times E_p(\partial K)}+\langle   {u^p_h}^{\ast} \mid  y^p_h \rangle _{{E_p(\partial K)}^{\ast} \times E_p(\partial K)}.
\end{aligned}
\end{equation}
\end{definition}
Using the duality product \eqref{eq:power12} as the power we can now define the generalized Stokes-Dirac structure.
\begin{definition}[Generalized Stokes-Dirac Structure]\label{def:DGDS}
For $K \in \mathcal{T}_h$, we define the subspace $D_h$ of $F'_{p,q} \times E'_{p,q}$ as
\begin{equation}\label{eq:disdirac}
\begin{aligned}
&D_h = \Bigg\{(f^p_h, f^q_h, y^p_h, {y^p_h}^{\ast}, e^p_h, e^q_h, {u^p_h}^{\ast}, u^p_h) \in F'_{p,q} \times E'_{p,q} \mid \\
 &   \left[ \begin{array}{c} f^p_h \\ f^q_h \end{array}\right]=
  \left[ \begin{array}{cc} 0& (-1)^{r_1}d_{pq} + (-1)^{r_1+q} \frac{1}{2}  \mathrm{tr}^{\ast}_p Q(\mathrm{tr}_q)   \\ d_{qp} + (-1)^p\frac{1}{2}(Q\mathrm{tr}_q)^{\ast} \mathrm{tr}_p & 0 \end{array}\right]\left[ \begin{array}{c} e^p_h \\ e^q_h\end{array}\right]\\
  & \ \ \ \ \ \ \ \ \ \ \ +  \left[ \begin{array}{cc} 0 & (-1)^{r_1+q}\mathrm{tr}^{\ast}_p  \\ (-1)^p(Q\mathrm{tr}_q)^{\ast} & 0 \end{array}\right]\left[ \begin{array}{c} u^{p}_h\\ {u^p_h}^{\ast} \end{array}\right],\\
 &\left[ \begin{array}{c}  {y^p_h}^{\ast} \\ y^p_h \end{array}\right]  = \left[ \begin{array}{cc} 0 & -Q(\mathrm{tr}_q) \\ -(-1)^p\mathrm{tr}_p & 0 \end{array}\right] \left[ \begin{array}{c} e^p_h  \\ e^q_h \end{array}\right]\Bigg\},
 \end{aligned}
\end{equation}
where $p+q = n+1, r_1 = pq+1 $, $d_{pq}$ and $d_{qp}$ as defined in \eqref{eq:derivative1} and \eqref{eq:derivative2}, the trace operation defined as in \eqref{eq:boundspace} and the operator $Q$ defined as in \eqref{eq:operatordual}, \eqref{eq:operatordual1}.
\end{definition}
We now prove that the generalized Stokes-Dirac structure given in the definition above is indeed a Dirac structure. For that we use the following lemma. Although the statement in the following lemma is well known, it is difficult to find the proof, so we state the proof as well.
\begin{lemma}\label{lem:Dirac} 
Consider the finite dimensional linear spaces $F_h$ and $E_h$, having a non-degenerate bilinear product defined on $F_h \times E_h$. A subspace $D_h \in F_h \times E_h$ is a Dirac structure if and only if the power is conserved and $\mathrm{\dim}(D_h) = \mathrm{\dim}(F_h)$.
\end{lemma}
\begin{proof} Given, $\langle e \mid f\rangle=0, \forall (f,e) \in D_h$ and $\mathrm{\dim}(D_h)=\mathrm{\dim}(F_h)$, we have to prove that $D_h$ is a Dirac structure, i.e., $D_h =(D_h)^{\perp}$. And, conversely if $D_h$ is a Dirac structure, i.e., $D_h ={D_h}^{\perp}$, we have to prove that $\mathrm{\dim}(D_h)=\mathrm{\dim}(F_h)$.\\
Suppose $\forall (f,e) \in D_h$, we have 
\begin{equation}
\begin{aligned}
\langle\langle (f,e),(f,e)\rangle\rangle_h = \langle e  \mid f\rangle +\langle e\mid f \rangle= 2\langle e\mid f\rangle = 0,
\end{aligned}
\end{equation}
then using \eqref{eq:power} $D_h \subseteq  (D_h)^{\perp}$. Let $n$ be the dimension of $F_h$ and let $\{[f_1,e_1]^T, . . . ,$\\
$[f_n,e_n]^T\}$ be a basis of $D_h$. For $[\tilde{f},\tilde{e}]^T \in D^{\perp}_h$ it implies
\begin{equation}\label{eq:dproduct}
\begin{aligned}
\langle\langle (f,e),(\tilde{f},\tilde{e})\rangle\rangle_h = \langle e_k  \mid \tilde{f}\rangle +\langle \tilde{e} \mid f_k \rangle =0, \qquad k=1,\cdots, n.
\end{aligned}
\end{equation}
This gives $n$ independent linear equations to be solved in a $2n$ dimensional space. This implies that the solution set is of $2n-n=n$ dimensions. So, $\mathrm{\dim}(D^{\perp}_h) =n$.
Since, $\mathrm{\dim}(D_h)=\mathrm{\dim}(E_h)=\mathrm{\dim}(F_h)=n$ and $D_h \subseteq (D_h)^{\perp}$, we have $ D_h= (D_h)^{\perp}$.\\
Conversely, if $D_h$ is a Dirac structure, then $D_h = (D_h)^{\perp}$, and \eqref{eq:power} implies in particular that
\begin{equation}
\begin{aligned}
 \langle e\mid f\rangle = 0, \quad \forall (f,e)\in D_h.
\end{aligned}
\end{equation}
Let $\{[f_1, e_1]^T, . . . , [f_m, e_m]^T\}$ be a basis of $D_h$, for $[\tilde{f},\tilde{e}]^T \in (D_h)^{\perp}$ it implies
\begin{equation}\label{eq:dproduct1}
\begin{aligned}
\langle e_k \mid \tilde{f}\rangle +\langle \tilde{e} \mid f_k\rangle =0, \qquad k=1,\cdots, m.
\end{aligned}
\end{equation}
This gives $m$ independent linear equations to be solved in a $2n$ dimensional space. This implies that the solution set is of $2n-m$ dimensions. So, $\mathrm{\dim}(D^{\perp}) =2n-m$. Since, $D_h = (D_h)^{\perp}$, we have $2n-m=n$, or $m=n$. This proves that $\mathrm{\dim}(D_h)^{\perp}=\mathrm{\dim}(D_h)=\mathrm{\dim}(F_h)=n$.
\end{proof}
Lemma \ref{lem:Dirac} states the conditions for a linear subspace to be a Dirac structure. Using this lemma we prove now that the subspace \eqref{eq:disdirac} is a Dirac structure. Along with Lemma \ref{lem:Dirac}, we state two more lemmas, which will be used to prove that the structure \eqref{eq:disdirac} is a Dirac structure.
\begin{lemma}\label{lem:intbyparts}
For port variables $e^p_h \in E_p(K), e^q_h \in E_q(K), f^p_h \in F_p(K), f^q_h \in F_q(K)$, input and output port boundary pairs  $\{y^p_h,u^p_h\} \in E_p(\partial K), \ \{{y^p_h}^{\ast},{u^p_h}^{\ast}\} \in {E_p(\partial K)}^{\ast}$ the following holds. If
\begin{equation}
\begin{aligned}
f^p_h &= ((-1)^{r_1}d_{pq} + (-1)^{r_1+q}\frac{1}{2} \mathrm{tr}_p^{\ast}Q(\mathrm{tr_q})) e^q_h +  (-1)^{r_1+q}\mathrm{tr}_p^{\ast} {u^p_h}^{\ast},\\
f^q_h &=(d_{qp} +(-1)^p \frac{1}{2}(Q\mathrm{tr}_q)^{\ast}\mathrm{tr}_p)e^p_h + (-1)^p(Q\mathrm{tr}_q)^{\ast} u^p_h,
\end{aligned}
\end{equation}
then for $\nu^p_h \in E_p(K)$, $\nu^q_h \in E_q(K)$,
\begin{equation}\label{eq:newlemma1}
\begin{aligned}
\langle \nu^p_h \mid f^p_h \rangle_K &= - \langle e^q_h \mid d_{qp} \nu^p_h\rangle_K + (-1)^{p+1} \frac{1}{2}  \langle  Q(\mathrm{tr}_qe^q_h)\mid\mathrm{tr}_p\nu^p_h\rangle_{{E_p(\partial K)}^{\ast} \times E_p(\partial K)}\\
&\ \ \ \ \ +(-1)^{p} \langle {u^p_h}^{\ast} \mid\mathrm{tr}_p\nu^p_h\rangle_{{E_p(\partial K)}^{\ast} \times E_p(\partial K)}\\
\langle \nu^q_h \mid f^q_h \rangle_K &= (-1)^p\langle d_{pq}\nu^q_h \mid e^p_h \rangle_K +(-1)^{p+1} \frac{1}{2} \langle Q(\mathrm{tr}_q\nu^q_h)\mid\mathrm{tr}_pe^p_h\rangle_{{E_p(\partial K)}^{\ast} \times E_p(\partial K)} \\
&\ \ \ \ \ +(-1)^p \langle Q\mathrm{tr}_q(\nu^q_h) \mid u^p_h\rangle_{{E_p(\partial K)}^{\ast} \times E_p(\partial K)}
\end{aligned}
\end{equation}
with $p+q=n+1$ and $r_1=pq+1$.
\end{lemma}
\begin{proof}
For the test differential forms $\nu^p_h \in E_p(K)$ and $\nu^q_h \in E_q(K)$ we have
\begin{equation}\label{eq:neweq1}
\begin{aligned}
\langle \nu^p_h \mid f^p_h \rangle_K &= \langle \nu^p_h \mid ((-1)^{r_1}d_{pq} + (-1)^{r_1+q}\frac{1}{2} \mathrm{tr}_p^{\ast}Q(\mathrm{tr_q})) e^q_h \rangle_K +  (-1)^{r_1+q} \langle  \nu^p_h \mid \mathrm{tr}_p^{\ast} {u^p_h}^{\ast} \rangle_K\\
\langle \nu^q_h \mid f^q_h \rangle_K &= \langle \nu^q_h \mid (d_{qp} + (-1)^p\frac{1}{2}(Q\mathrm{tr}_q)^{\ast}\mathrm{tr}_p)e^p_h \rangle_K + (-1)^p\langle \nu^q_h \mid (Q\mathrm{tr}_q)^{\ast} u^p_h \rangle_K
\end{aligned}
\end{equation}
Using \eqref{eq:boudual1} and \eqref{eq:boudual2}, \eqref{eq:neweq1} can be written as 
\begin{equation}\label{eq:neweq2}
\begin{aligned}
\langle \nu^p_h \mid f^p_h \rangle_K &= (-1)^{r_1} \langle \nu^p_h \mid d_{pq}e^q_h \rangle_K + (-1)^{r_1+q} \frac{1}{2}\langle \mathrm{tr}_p\nu^p_h\mid Q(\mathrm{tr}_qe^q_h)\rangle_{E_p(\partial K) \times {E_p(\partial K)}^{\ast}}\\
&\ \ \ \ \ +  (-1)^{r_1+q}\langle \mathrm{tr}_p\nu^p_h \mid {u^p_h}^{\ast} \rangle_{E_p(\partial K) \times {E_p(\partial K)}^{\ast}}\\
\langle \nu^q_h \mid f^q_h \rangle_K &=\langle \nu^q_h \mid d_{qp}e^p_h \rangle_K +  (-1)^p \frac{1}{2}\langle Q(\mathrm{tr}_q\nu^q_h)\mid\mathrm{tr}_pe^p_h\rangle_{{E_p(\partial K)}^{\ast} \times E_p(\partial K)}\\
&\ \ \ \ + (-1)^p \langle Q\mathrm{tr}_q(\nu^q_h) \mid u^p_h\rangle_{{E_p(\partial K)}^{\ast} \times E_p(\partial K)}.
\end{aligned}
\end{equation} 
Using the relation $p+q=n+1$, we rewrite \eqref{eq:neweq2} as 
\begin{equation}\label{eq:neweq21}
\begin{aligned}
\langle \nu^p_h \mid f^p_h \rangle_K &= (-1)^{p+1}\langle d_{pq}e^q_h  \mid \nu^p_h \rangle_K + (-1)^p \frac{1}{2}\langle Q(\mathrm{tr}_qe^q_h)\mid \mathrm{tr}_p\nu^p_h\rangle_{{E_p(\partial K)}^{\ast} \times E_p(\partial K)}\\
&\ \ \ \ \ + (-1)^p \langle {u^p_h}^{\ast} \mid \mathrm{tr}_p\nu^p_h\rangle_{{E_p(\partial K)}^{\ast} \times E_p(\partial K)}\\
\langle \nu^q_h \mid f^q_h \rangle_K &=\langle \nu^q_h \mid d_{qp}e^p_h \rangle_K + (-1)^p\frac{1}{2} \langle Q(\mathrm{tr}_q\nu^q_h)\mid\mathrm{tr}_pe^p_h\rangle_{{E_p(\partial K)}^{\ast} \times E_p(\partial K)}\\
&\ \ \ \ + (-1)^p\langle Q\mathrm{tr}_q(\nu^q_h) \mid u^p_h\rangle_{{E_p(\partial K)}^{\ast} \times E_p(\partial K)}.
\end{aligned}
\end{equation} 
Using integration by parts rule, \eqref{eq:int1} and the relation $p+q=n+1$,
\begin{equation}\label{eq:neweq3}
\begin{aligned}
\langle \nu^p_h \mid f^p_h \rangle_K &= - \langle e^q_h \mid d_{qp} \nu^p_h\rangle_K + (-1)^{p+1}\langle \mathrm{tr}_q e^q_h \mid \mathrm{tr}_p \nu^p_h \rangle_{\partial K}\\
&\ \ \ \ \ + (-1)^p \frac{1}{2}\langle Q(\mathrm{tr}_qe^q_h)\mid \mathrm{tr}_p\nu^p_h\rangle_{{E_p(\partial K)}^{\ast} \times E_p(\partial K)}\\
&\ \ \ \ \ + (-1)^p\langle {u^p_h}^{\ast} \mid \mathrm{tr}_p\nu^p_h\rangle_{{E_p(\partial K)}^{\ast} \times E_p(\partial K)}\\
\langle \nu^q_h \mid f^q_h \rangle_K &= (-1)^p\langle d_{pq}\nu^q_h \mid e^p_h \rangle_K +(-1)^{p+1} \langle \mathrm{tr}_q \nu^q_h \mid \mathrm{tr}_pe^p_h \rangle_{\partial K}\\
&\ \ \ \ \ +(-1)^p \frac{1}{2}\langle  Q(\mathrm{tr}_q\nu^q_h)\mid\mathrm{tr}_pe^p_h\rangle_{{E_p(\partial K)}^{\ast} \times E_p(\partial K)}\\
&\ \ \ \ +(-1)^p \langle Q\mathrm{tr}_q(\nu^q_h) \mid u^p_h\rangle_{{E_p(\partial K)}^{\ast} \times E_p(\partial K)} \nonumber.
\end{aligned}
\end{equation}
Using \eqref{eq:operatordual1}, we get \eqref{eq:newlemma1}.
\end{proof}
\begin{theorem}\label{the:elestu}
The structure \eqref{eq:disdirac} is a Dirac structure, which implies that
\begin{equation}
\begin{aligned}
\langle e^p_h \mid f^p_h \rangle_K + \langle e^q_h \mid f^q_h \rangle_K &+ \langle {y^p_h}^{\ast} \mid u^p_h \rangle_{{E_p(\partial K)}^{\ast} \times E_p(\partial K)}\\
&+ \langle   {u^p_h}^{\ast} \mid  y^p_h \rangle _{{E_p(\partial K)}^{\ast} \times E_p(\partial K)}= 0.
\end{aligned}
\end{equation}
\end{theorem}
\begin{proof}
The power for the Dirac structure \eqref{eq:disdirac} is 
\begin{equation}\label{eq:discretepower}
\begin{aligned}
\langle e^p_h \mid f^p_h \rangle_K + \langle e^q_h \mid f^q_h \rangle_K + \langle {y^p_h}^{\ast} \mid u^p_h \rangle_{{E_p(\partial K)}^{\ast} \times E_p(\partial K)} + \langle   {u^p_h}^{\ast} \mid  y^p_h \rangle _{{E_p(\partial K)}^{\ast} \times E_p(\partial K)}
\end{aligned}
\end{equation}
Using Lemma \ref{lem:intbyparts} and replacing $\nu^p_h \in E_p,  \nu^q_h \in E_q$ with $e^p_h \in E_p,  e^q_h \in E_q$, we get 
\begin{equation}
\begin{aligned}
&\langle e^p_h \mid f^p_h \rangle_K + \langle e^q_h \mid f^q_h \rangle_K + \langle {y^p_h}^{\ast} \mid u^p_h \rangle_{{E_p(\partial K)}^{\ast} \times E_p(\partial K)} + \langle   {u^p_h}^{\ast} \mid  y^p_h \rangle _{{E_p(\partial K)}^{\ast} \times E_p(\partial K)}\\
&=-\langle e^q_h \mid d_{qp} e^p_h\rangle_K + (-1)^{p+1} \frac{1}{2}  \langle  Q(\mathrm{tr}_qe^q_h)\mid\mathrm{tr}_pe^p_h\rangle_{{E_p(\partial K)}^{\ast} \times E_p(\partial K)}\\
&\ \ \ \ \ + (-1)^p\langle {u^p_h}^{\ast} \mid\mathrm{tr}_pe^p_h\rangle_{{E_p(\partial K)}^{\ast} \times E_p(\partial K)}\\
&\ \ \ \ +  (-1)^p\langle d_{pq}e^q_h \mid e^p_h \rangle_K  + (-1)^{p+1} \frac{1}{2} \langle Q(\mathrm{tr}_qe^q_h)\mid\mathrm{tr}_pe^p_h\rangle_{{E_p(\partial K)}^{\ast} \times E_p(\partial K)}\\
&\ \ \ \ + (-1)^p\langle Q\mathrm{tr}_q(e^q_h) \mid u^p_h\rangle_{{E_p(\partial K)}^{\ast} \times E_p(\partial K)}\\
&\ \ \ \ + \langle {y^p_h}^{\ast} \mid u^p_h \rangle_{{E_p(\partial K)}^{\ast} \times E_p(\partial K)} + \langle   {u^p_h}^{\ast} \mid  y^p_h \rangle _{{E_p(\partial K)}^{\ast} \times E_p(\partial K)}
\end{aligned}
\end{equation}
Using \eqref{eq:int1}, we can further simplify 
\begin{align*}
&-\langle e^q_h \mid d_{qp} e^p_h\rangle_K + (-1)^{p+1} \frac{1}{2}  \langle  Q(\mathrm{tr}_qe^q_h)\mid\mathrm{tr}_pe^p_h\rangle_{{E_p(\partial K)}^{\ast} \times E_p(\partial K)}\\
&\ \ \ \ \ + (-1)^p\langle {u^p_h}^{\ast} \mid\mathrm{tr}_pe^p_h\rangle_{{E_p(\partial K)}^{\ast} \times E_p(\partial K)}\\
&\ \ \ \ +  (-1)^p\langle d_{pq}e^q_h \mid e^p_h \rangle_K  + (-1)^{p+1} \frac{1}{2} \langle Q(\mathrm{tr}_qe^q_h)\mid\mathrm{tr}_pe^p_h\rangle_{{E_p(\partial K)}^{\ast} \times E_p(\partial K)}\\
&\ \ \ \ + (-1)^p\langle Q\mathrm{tr}_q(e^q_h) \mid u^p_h\rangle_{{E_p(\partial K)}^{\ast} \times E_p(\partial K)}\\
&\ \ \ \ + \langle {y^p_h}^{\ast} \mid u^p_h \rangle_{{E_p(\partial K)}^{\ast} \times E_p(\partial K)} + \langle   {u^p_h}^{\ast} \mid  y^p_h \rangle _{{E_p(\partial K)}^{\ast} \times E_p(\partial K)}\\
&= (-1)^p\langle \mathrm{tr}_qe^q_h \mid \mathrm{tr}_p e^p_h \rangle_{\partial K} + (-1)^{p+1} \langle  Q(\mathrm{tr}_qe^q_h)\mid\mathrm{tr}_pe^p_h\rangle_{{E_p(\partial K)}^{\ast} \times E_p(\partial K)}\\
&\ \ \ \ \ + (-1)^p\langle {u^p_h}^{\ast} \mid\mathrm{tr}_pe^p_h\rangle_{{E_p(\partial K)}^{\ast} \times E_p(\partial K)}  + (-1)^{p}\langle Q\mathrm{tr}_q(e^q_h) \mid u^p_h\rangle_{{E_p(\partial K)}^{\ast} \times E_p(\partial K)}\\
&\ \ \ \ \ + \langle {y^p_h}^{\ast} \mid u^p_h \rangle_{{E_p(\partial K)}^{\ast} \times E_p(\partial K)} + \langle   {u^p_h}^{\ast} \mid  y^p_h \rangle _{{E_p(\partial K)}^{\ast} \times E_p(\partial K)}
 \end{align*}
Using \eqref{eq:operatordual1}, we can further simplify as follows,
\begin{equation}\label{eq:discretepower2}
\begin{aligned}
&(-1)^p\langle \mathrm{tr}_qe^q_h \mid \mathrm{tr}_p e^p_h \rangle_{\partial K} + (-1)^{p+1} \langle  Q(\mathrm{tr}_qe^q_h)\mid\mathrm{tr}_pe^p_h\rangle_{{E_p(\partial K)}^{\ast} \times E_p(\partial K)}\\
&\ \ \ \ \ + (-1)^p\langle {u^p_h}^{\ast} \mid\mathrm{tr}_pe^p_h\rangle_{{E_p(\partial K)}^{\ast} \times E_p(\partial K)}  + (-1)^{p}\langle Q\mathrm{tr}_q(e^q_h) \mid u^p_h\rangle_{{E_p(\partial K)}^{\ast} \times E_p(\partial K)}\\
&\ \ \ \ \ + \langle {y^p_h}^{\ast} \mid u^p_h \rangle_{{E_p(\partial K)}^{\ast} \times E_p(\partial K)} + \langle   {u^p_h}^{\ast} \mid  y^p_h \rangle _{{E_p(\partial K)}^{\ast} \times E_p(\partial K)}\\
&=(-1)^p\langle {u^p_h}^{\ast} \mid\mathrm{tr}_pe^p_h\rangle_{{E_p(\partial K)}^{\ast} \times E_p(\partial K)}  + (-1)^{p}\langle Q\mathrm{tr}_q(e^q_h) \mid u^p_h\rangle_{{E_p(\partial K)}^{\ast} \times E_p(\partial K)}\\
&\ \ \ \ \ + \langle {y^p_h}^{\ast} \mid u^p_h \rangle_{{E_p(\partial K)}^{\ast} \times E_p(\partial K)} + \langle   {u^p_h}^{\ast} \mid  y^p_h \rangle _{{E_p(\partial K)}^{\ast} \times E_p(\partial K)}
\end{aligned}
\end{equation}
Using the relations ${y^p_h}^{\ast}= -Q(\mathrm{tr}_q(e^q_h)), y^p_h =-(-1)^p\mathrm{tr}_pe^p_h$ from \eqref{eq:disdirac} in \eqref{eq:discretepower2} we obtain
\begin{align*}
 & - \langle {u^p_h}^{\ast} \mid y^p_h \rangle_{{E_p(\partial K)}^{\ast} \times E_p(\partial K)} - \langle {y^p_h}^{\ast} \mid u^p_h\rangle_{{E_p(\partial K)}^{\ast} \times E_p(\partial K)}\\
 &\ \ \ \ \ + \langle {y^p_h}^{\ast} \mid u^p_h \rangle_{{E_p(\partial K)}^{\ast} \times E_p(\partial K)}+ \langle   {u^p_h}^{\ast} \mid  y^p_h \rangle _{{E_p(\partial K)}^{\ast} \times E_p(\partial K)} =0.
\end{align*}
Furthermore, because of the duality $\mathrm{\dim}(E'_{p,q})=\mathrm{\dim}(F'_{p,q})=\mathrm{\dim}(D_h)$. Using Lemma \ref{lem:Dirac}, it follows that \eqref{eq:disdirac} is a Dirac structure.
\end{proof}
Thus, for the Stokes-Dirac structure \eqref{eq:disdirac}, we obtain
\begin{equation}\label{eq:energybal}
\begin{aligned}
\langle e^p_h \mid f^p_h \rangle_K + \langle e^q_h \mid f^q_h \rangle_K &+ \langle {y^p_h}^{\ast} \mid u^p_h \rangle_{{E_p(\partial K)}^{\ast} \times E_p(\partial K)} \\
& + \langle   {u^p_h}^{\ast} \mid  y^p_h \rangle _{{E_p(\partial K)}^{\ast} \times E_p(\partial K)} =0.
\end{aligned}
\end{equation}
Hence power is conserved.
\par
Along with the elementwise Stokes-Dirac structure of Definition \ref{def:DGDS}, we need a power preserving interconnection structure at the boundaries of the elements so that we can connect all elements in the discretized manifold. We use the power preserving property of a Dirac structure and define a Dirac structure at the faces of the elements. We call such a Dirac structure an interconnection Dirac Structure. 
\subsubsection{The interconnection Dirac structure}
Let $\{K_1,K_2\} \in \mathcal{T}_h$ be two arbitrary elements sharing a common boundary $\partial K_1 \cap \partial K_2$. Using \eqref{eq:bounddef}, we have 
\begin{equation}\label{eq:bounddef1}
\begin{aligned}
E_p(\partial K_1) =\Big\{&(\mathrm{tr}_{f_1,\partial f_1} e^p_{h,f_1}, \mathrm{tr}_{f_2,\partial f_2} e^p_{h,f_2} , \cdots, \mathrm{tr}_{f_m},{\partial f}_m e^p_{h,f_m})\mid f_i \in \Delta(K_1),\\
&\mathrm{\dim}(f_i) \in[p,r+p-1], e^p_{h,f_i}\in E_p(f)\Big\}\\
E_p(\partial K_2) =\Big\{&(\mathrm{tr}_{f_1,\partial f_1} e^p_{h,f_1}, \mathrm{tr}_{f_2,\partial f_2} e^p_{h,f_2} , \cdots, \mathrm{tr}_{f_m},{\partial f}_m e^p_{h,f_m})\mid f_i \in \Delta(K_2),\\
&\mathrm{\dim}(f_i) \in[p,r+p-1], e^p_{h,f_i}\in E_p(f)\Big\},
\end{aligned}
\end{equation}
with $m$ being the number of subsimplices in $\partial \Delta(K)$.\\
As only the common boundaries between two elements will be connected through the interconnection Dirac Structure we split the function spaces $E_p(\partial K_1)$ and $E_p(\partial K_2)$ as 
\begin{equation}
\begin{aligned}
E_p(\partial K_1) &= E_p(\Delta_{1\cap2}) \oplus E_p(\Delta_{1 \neg 2}),\\
E_p(\partial K_2) &= E_p(\Delta_{1\cap2}) \oplus E_p(\Delta_{2 \neg 1}), 
\end{aligned}
\end{equation}
where $\Delta_{1\cap2}:= \partial K_1 \cap \partial K_2$, $\Delta_{1\neg2}= \partial K_1 - (\partial K_1 \cap \partial K_2)$ and $\Delta_{2\neg1}= \partial K_2$\\
$- (\partial K_1 \cap \partial K_2)$. We can then define 
\begin{equation}\label{eq:bounddef2}
\begin{aligned}
Z := E_p(\Delta_{1\cap2})=\Big((w^p_{h,L}\in E_p(\partial K_1)\vert_f, w^p_{h,R}\in E_p(\partial K_2)\vert_f)  \mid f \in \Delta_{1\cap2}\Big),
\end{aligned}
\end{equation}
and its dual function space as 
\begin{equation}\label{eq:bounddef4}
\begin{aligned}
Z^{\ast} := {E_p(\Delta_{1\cap2})}^{\ast}=\Big(&({w^{p\ast}_{h,L}}\in {E_p(\partial K_1)}^{\ast}\vert_f, {w^{p\ast}_{h,R}}\in {E_p(\partial K_2)}^{\ast}\vert_f)\mid f \in \Delta_{1\cap2}\Big).
\end{aligned}
\end{equation}
We choose the polynomial differential form space for the {\em interconnection} Dirac structure to be 
\begin{equation}
\begin{aligned}
Z_T:= Z \times {Z}^{\ast}.
\end{aligned} 
\end{equation}
The dual of $Z_T$, $Z^{\ast}_T =  {Z}^{\ast}\times Z$. For, $((w^p_{h,L}, w^p_{h,R}),({w^{p\ast}_{h,L}}, {w^{p\ast}_{h,R}})) \in Z_T$ and $((y^{p\ast}_{h,L}, y^{p\ast}_{h,R}),$\\
$(y^p_{h,L}, y^p_{h,R})) \in Z^{\ast}_T$, the duality product on the space $Z_T \times Z^{\ast}_T$ becomes 
\begin{equation}\label{eq:power2}
\begin{aligned}
&\langle {y^{p\ast}_{h,L}} \mid w^p_{h,L}\rangle_{{Z}^{\ast}\times Z} +\langle{w^{p\ast}_{h,L}} \mid y^p_{h,L}\rangle_{{Z}^{\ast}\times Z}+ \langle {y^{p\ast}_{h,R}} \mid w^p_{h,R}  \rangle_{{Z}^{\ast}\times Z} + \langle   {w^{p\ast}_{h,R}} \mid y^p_{h,R} \rangle_{{Z}^{\ast}\times Z}.
\end{aligned}
\end{equation}
Based on the duality product \eqref{eq:power2} we can now state the interconnection Dirac structure. 
\begin{definition}[Interconnection Dirac Structure]\label{def:intdirac}
For the boundary port variables, $((w^p_{h,L}, w^p_{h,R}),({w^{p\ast}_{h,L}}, {w^{p\ast}_{h,R}})) \in Z_T$ and $((y^{p\ast}_{h,L}, y^{p\ast}_{h,R}),(y^p_{h,L}, y^p_{h,R})) \in Z^{\ast}_T$,
 \begin{equation}
 \begin{aligned}
 D_c = \Bigg\{&(w^p_{h,L}, w^p_{h,R}, {w^{p\ast}_{h,L}}, {w^{p\ast}_{h,R}}, {y^{p\ast}_{h,L}},{y^{p\ast}_{h,R}}, y^p_{h,L}, y^p_{h,R}) \in Z_T \times Z^{\ast}_T) \mid \\
\label{eq:inter}
   &\left[ \begin{array}{c}w^p_{h,L}\\ {w^{p\ast}_{h,L}}\\ w^p_{h,R}\\ {w^{p\ast}_{h,R}} \end{array}\right] =   \left[ \begin{array}{cccc} 0 & -\frac{1}{2} + \theta & 0  & -\theta  \\ \frac{1}{2} -\theta &0 & \theta -1  & 0 \\ 0& 1-\theta & 0 & \theta-\frac{1}{2} \\ 
   \theta & 0 & \frac{1}{2} -\theta &0  \end{array}\right] \left[ \begin{array}{c}{y^{p\ast}_{h,L}}\\y^p_{h,L}\\ {y^{p\ast}_{h,R}}\\ y^p_{h,R} \end{array}\right],\Bigg\}
   \end{aligned}
\end{equation}
where $\theta \in [0,1]$.
\end{definition}
\begin{theorem}\label{the:boustu}
The interconnection structure \eqref{eq:inter} is a Dirac structure . 
\end{theorem}
\begin{proof}
Using Lemma \ref{lem:Dirac} we will be able to prove that \eqref{eq:inter} is a Dirac structure if the power of \eqref{eq:inter} is zero and the dimensions of the spaces are equal. The power of the Dirac structure, \eqref{eq:inter} is given by,
\begin{equation}\label{eq:interpower}
\begin{aligned}
&\langle {y^{p\ast}_{h,L}} \mid w^p_{h,L}\rangle_{{Z}^{\ast}\times Z}+\langle  {w^{p\ast}_{h,L}}  \mid  y^p_{h,L} \rangle_{{Z}^{\ast}\times Z} + \langle {y^{p\ast}_{h,R}} \mid w^p_{h,R}  \rangle_{{Z}^{\ast}\times Z}\\
&+ \langle  {w^{p\ast}_{h,R}}  \mid  y^p_{h,R}  \rangle_{{Z}^{\ast}\times Z}=(-\frac{1}{2}+\theta)\langle {y^{p\ast}_{h,L}}\mid y^p_{h,L}\rangle_{{Z}^{\ast}\times Z} - \theta \langle {y^{p\ast}_{h,L}}\mid y^p_{h,R}\rangle_{{Z}^{\ast}\times Z}\\
&\ \ \ \ \ + (\frac{1}{2}-\theta)\langle {y^{p\ast}_{h,L}}\mid y^p_{h,L}\rangle_{{Z}^{\ast}\times Z}+ (\theta-1) \langle {y^{p\ast}_{h,R}}\mid y^p_{h,L}\rangle_{{Z}^{\ast}\times Z}\\
&\ \ \ \ \ + (1-\theta)\langle {y^{p\ast}_{h,R}}\mid y^p_{h,L}\rangle_{{Z}^{\ast}\times Z} + (\theta-\frac{1}{2}) \langle {y^{p\ast}_{h,R}}\mid y^p_{h,R}\rangle_{{Z}^{\ast}\times Z}\\
&\ \ \ \ \ + \theta\langle {y^{p\ast}_{h,L}}\mid y^p_{h,R}\rangle_{{Z}^{\ast}\times Z} + (\frac{1}{2}-\theta) \langle {y^{p\ast}_{h,R}}\mid y^p_{h,R}\rangle_{{Z}^{\ast}\times Z}\\
&=0.
\end{aligned}
\end{equation}
Moreover, the dual of the function space $Z_T$, ${Z_T}^{\ast}$ is the same function space $Z_T$ so $ \mathcal{\dim}(Z_T)=\mathcal{\dim}({Z_T}^{\ast})=\mathcal{\dim}(D_c)$. Hence, \eqref{eq:inter} is a Stokes-Dirac structure. 
\end{proof}
\section{Choice of Interconnection variable}\label{sec:choiceinter}
In this section, we prove that upon identification of the correct interconnection port variables, the coupling of discrete Dirac structures results again in a discrete Dirac structure. The interconnection is diagrammatically shown in Figure \ref{fig:Fig 1}.
\begin{figure}[H]
\begin{adjustbox}{center}
\begin{tikzpicture}
[
every label/.append style={font= \scriptsize},
]
\draw (1,0) ellipse (1 cm and 2 cm) node {$D_1$}; 
\draw[fill=black] (0.15,1) node[label=left:{${u^{p\ast}_1} \vert_{\Delta_{1 \neg 2}}$}] {} circle(0.1cm);
\draw[fill=black] (0.35,1.5) node [label=left:{$u^p_1 \vert_{\Delta_{1 \neg 2}}$}]{}circle(0.1cm);
\draw[fill=black] (0.35,-1.5) node[label=left:{${y^{p\ast}_1} \vert_{\Delta_{1 \neg 2}}$}] {} circle(0.1cm);
\draw[fill=black] (0.15,-1) node [label=left:{$y^p_1 \vert_{\Delta_{1 \neg 2}}$}]{}circle(0.1cm);
\draw[fill=black] (1.85,-1) node[label=right:{${u^{p\ast}_1}\vert_{\Delta_{1 \cap2}}$}] {} circle(0.1cm);
\draw[fill=black] (1.65,1.5) node [label=right:{$u^p_1\vert_{\Delta_{1 \cap 2}}$}]{}circle(0.1cm);
\draw[fill=black] (1.65,-1.5) node[label=right:{${y^{p\ast}_1}\vert_{\Delta_{1 \cap 2}} $}] {} circle(0.1cm);
\draw[fill=black] (1.85,1) node [label=right:{$y^p_1\vert_{\Delta_{1 \cap 2}}$}]{}circle(0.1cm);
\draw (9,0) ellipse (1 cm and 2 cm)node{$D_2$};
\draw[fill=black] (9.85,-1) node[label=right:{${u^{p\ast}_2} \vert_{\Delta_{2 \neg 1}}$}] {} circle(0.1cm);
\draw[fill=black] (9.65,1.5) node [label=right:{$u^p_2 \vert_{\Delta_{2 \neg 1}}$}]{}circle(0.1cm);
\draw[fill=black] (9.65,-1.5) node[label=right:{${y^{p\ast}_2} \vert_{\Delta_{2 \neg 1}}$}] {} circle(0.1cm);
\draw[fill=black] (9.85,1) node [label=right:{$y^p_2 \vert_{\Delta_{2 \neg 1}}$}]{}circle(0.1cm);
\draw[fill=black] (8.15,-1) node[label=left:{${u^{p\ast}_2}\vert_{\Delta_{1 \cap 2}} $}] {} circle(0.1cm);
\draw[fill=black] (8.35,1.5) node [label=left:{$u^p_2\vert_{\Delta_{1 \cap 2}}$}]{}circle(0.1cm);
\draw[fill=black] (8.35,-1.5) node[label=left:{${y^{p\ast}_2}\vert_{\Delta_{1 \cap 2}}$}] {} circle(0.1cm);
\draw[fill=black] (8.15,1) node [label=left:{$y^p_2\vert_{\Delta_{1 \cap 2}}$}]{}circle(0.1cm);
 \draw (5,0) ellipse (1 cm and 2 cm)node{$D_c$};
\draw[fill=black] (5.85,-1) node[label=right:{${w^{p\ast}_{h,R}}$}] {} circle(0.1cm);
\draw[fill=black] (5.65,1.5) node [label=right:{$w^p_{h,R}$}]{}circle(0.1cm);
\draw[fill=black] (5.65,-1.5) node[label=right:{${y^{p\ast}_{h,R}}$}] {} circle(0.1cm);
\draw[fill=black] (5.85,1) node [label=right:{$y^p_{h,R} $}]{}circle(0.1cm);
\draw[fill=black] (4.15,-1) node[label=left:{${w^{p\ast}_{h,L}}$}] {} circle(0.1cm);
\draw[fill=black] (4.35,1.5) node [label=left:{$w^p_{h,L}$}]{}circle(0.1cm);
\draw[fill=black] (4.35,-1.5) node[label=left:{${y^{p\ast}_{h,L}}$}] {} circle(0.1cm);
\draw[fill=black] (4.15,1) node [label=left:{$y^p_{h,L}$}]{}circle(0.1cm);
 \end{tikzpicture}
  \end{adjustbox}
  \caption{Coupling of two Dirac structures via the interconnection Dirac structure $D_c$}
  \label{fig:Fig 1}
 \end{figure}
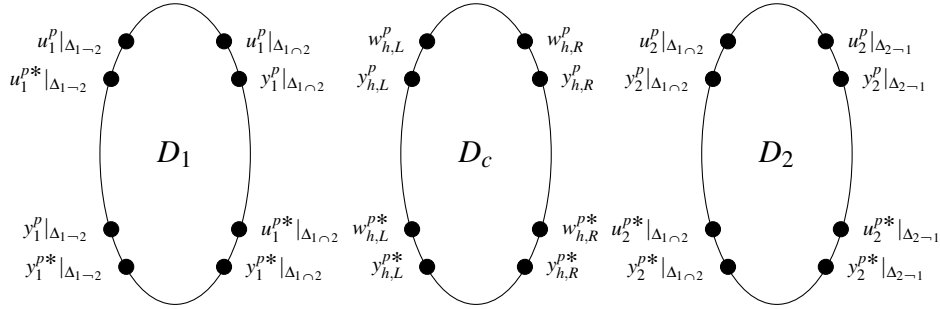
 Let us consider two simplices $K_1, K_2 \in \mathcal{T}_h$ with a common boundary. The conservation laws on $K_1$ are described by the port variables $(f^p_1,f^q_1,y^p_1,{y^p_1}^{\ast},e^p_1,e^q_1,{u^p_1}^{\ast},u^p_1) \in D_1  \subset F'_{p,q}(K_1)\times E'_{p,q}(K_1)$ and the conservation laws on $K_2$ by the port variables $(f^p_2,f^q_2,y^p_2,{y^p_2}^{\ast},
 e^p_2,$\\
 $e^q_2,{u^p_2}^{\ast},u^p_2) \in D_2 \subset F'_{p,q}(K_2)\times E'_{p,q}(K_2)$, respectively. We connect these two simplices through their common boundary $\Delta_{1\cap2}$.\\
For simplicity of notation we use 
\begin{equation}\label{eq:boundarynotation}
\begin{aligned}
Z:=E_p(\Delta_{1\cap2}),\qquad Z_1:=E_p(\Delta_{1\neg2}), \qquad Z_2:=E_p(\Delta_{2\neg1}).
 \end{aligned}
 \end{equation}
 For $F'_{p,q}(K_1) := F_p(K_1) \times F_q(K_1) \times E_p(\partial K_1) \times {E_p(\partial K_1)}^{\ast}$ and 
$E'_{p,q}(K_1) := E_p(K_1) \times E_q(K_1) \times {E_p(\partial K_1)}^{\ast} \times E_p(\partial K_1)$ the generalized Stokes-Dirac structure $D_1$ is defined as in Definition \ref{def:DGDS} and the interconnection Dirac structure $D_c$ is defined as in Definition \ref{def:intdirac}. Interconnecting $D_1$ and $D_c$, we define the structure $D_1 \circ D_c$ on $F'_{p,q}(K_1)\times E'_{p,q}(K_1)$, 
 \begin{equation}\label{eq:firstintersection}
\begin{aligned}
&D_1 \circ D_c = \Big\{(f^p_1,f^q_1,(y^p_1\vert_{\Delta_{1\neg2}},\ y^p_{h,R})^T,({y^{p\ast}_1}\vert_{\Delta_{1\neg2}},\ {y^{p\ast}_{h,R}})^T,e^p_1,e^q_1, \\ 
&\ \ \ \ \ \ \ \ ({u^{p\ast}_1}\vert_{\Delta_{1\neg2}},\ -{w^{p\ast}_{h,R}})^T,(u^p_1\vert_{\Delta_{1\neg2}},\ -w^p_{h,R})^T) \in F'_{p,q}(K_1)\times E'_{p,q}(K_1) \mid \\
&\ \ \ \ \ \ \ \ \exists \ u^p_1\vert_{\Delta_{1\cap2}},u^{p\ast}_1\vert_{\Delta_{1\cap2}},y^p_1\vert_{\Delta_{1\cap2}},y^{p\ast}_1\vert_{\Delta_{1\cap2}} \ \mathrm{s.t.\ with} \ \\
&\ \ \ \ \ \ \ \ y^{p\ast}_1 = ({y^{p\ast}_1}\vert_{\Delta_{1\neg2}},\ {y^{p\ast}_1}\vert_{\Delta_{1\cap2}})^T,  y^p_1 = ({y^p_1}\vert_{\Delta_{1\neg2}},\ {y^p_1}\vert_{\Delta_{1\cap2}})^T,\\
&\ \ \ \ \ \ \ \ u^{p\ast}_1 = ({u^{p\ast}_1}\vert_{\Delta_{1\neg2}},\ {u^{p\ast}_1}\vert_{\Delta_{1\cap2}})^T,  u^p_1 = ({u^p_1}\vert_{\Delta_{1\neg2}},\ {u^p_1}\vert_{\Delta_{1\cap2}})^T,\\
&\ \ \ \ \ \ \ \ \mathrm{holds} \ (f^p_1,f^q_1,y^p_1,y^{p\ast}_1,e^p_1,e^q_1,u^{p\ast}_1, u^p_1) \in D_1 \ \mathrm{and}\ (u^p_1\vert_{\Delta_{1\cap2}}, w^p_{h,R},\\
&\ \ \ \ \ \ \ \ u^{p\ast}_1\vert_{\Delta_{1\cap2}}, {w^{p\ast}_{h,R}}, y^{p\ast}_1\vert_{\Delta_{1\cap2}},{y^{p\ast}_{h,R}}, y^p_1\vert_{\Delta_{1\cap2}}, y^p_{h,R}) \in D_c.\Big\}
 \end{aligned}
\end{equation}
So when connecting $D_1$ and $D_c$, we basically take in Figure \ref{fig:Fig 1} $u^{p\ast}_1\vert_{\Delta_{1\cap2}} = w^{p\ast}_{h,L}$, $u^p_1\vert_{\Delta_{1\cap2}} = w^p_{h,L}$, $y^{p\ast}_1\vert_{\Delta_{1\cap2}} = y^{p\ast}_{h,L}$ and $y^p_1\vert_{\Delta_{1\cap2}} = y^p_{h,L}$.\\
The power of $D_1 \circ D_c$ is given by 
\begin{equation}\label{eq:interenergybal}
\begin{aligned}
&\langle e^p_1 \mid f^p_1 \rangle_{K_1} + \langle e^q_1 \mid f^q_1 \rangle_{K_1} + \langle {y^p_1}^{\ast} \mid u^p_1 \rangle_{Z_1^{\ast} \times Z_1} + \langle   {u^p_1}^{\ast} \mid  y^p_1 \rangle_{Z_1^{\ast} \times Z_1}\\
&- \langle {y^{p\ast}_{h,R}} \mid w^p_{h,R} \rangle_{Z^{\ast} \times Z} - \langle   {w^{p\ast}_{h,R}} \mid  y^p_{h,R}\rangle_{Z^{\ast} \times Z}.
\end{aligned}
\end{equation}
\begin{lemma}\label{lem:interconnectionlemma}
 The structure $D_1 \circ D_c$, \eqref{eq:firstintersection} with power given by \eqref{eq:interenergybal} is a Dirac structure.
 \end{lemma}
 \begin{proof}
Using Theorem \ref{the:elestu} we can state 
\begin{equation}\label{eq:brokenenergybal1}
\begin{aligned}
&\langle e^p_1 \mid f^p_1 \rangle_{K_1} + \langle e^q_1 \mid f^q_1 \rangle_{K_1} + \langle {y^p_1}^{\ast} \mid u^p_1 \rangle_{Z_1^{\ast} \times Z_1} + \langle   {u^p_1}^{\ast} \mid  y^p_1 \rangle_{Z_1^{\ast} \times Z_1}\\
&+ \langle {y^p_1}^{\ast} \mid u^p_1 \rangle_{Z^{\ast} \times Z} + \langle   {u^p_1}^{\ast} \mid  y^p_1\rangle_{Z^{\ast} \times Z}=0.
\end{aligned}
\end{equation}
We can rewrite \eqref{eq:brokenenergybal1} as
\begin{equation}\label{eq:brokenenergybal2}
\begin{aligned}
&\langle e^p_1 \mid f^p_1 \rangle_{K_1} + \langle e^q_1 \mid f^q_1 \rangle_{K_1} + \langle {y^p_1}^{\ast} \mid u^p_1 \rangle_{Z_1^{\ast} \times Z_1} + \langle   {u^p_1}^{\ast} \mid  y^p_1 \rangle_{Z_1^{\ast} \times Z_1}\\
&+ \langle {y^{p\ast}_{h,L}} \mid w^p_{h,L} \rangle_{Z^{\ast} \times Z} + \langle   {w^{p\ast}_{h,L}} \mid  y^p_{h,L}\rangle_{Z^{\ast} \times Z}=0.
\end{aligned}
\end{equation}
Using the power preserving property of interconnection Dirac structure $D_c$, Theorem \ref{the:boustu}, we have 
\begin{equation}\nonumber
\begin{aligned}
&\langle {y^{p\ast}_{h,L}} \mid w^p_{h,L}\rangle_{{Z}^{\ast}\times Z}+\langle  {w^{p\ast}_{h,L}}  \mid  y^p_{h,L} \rangle_{{Z}^{\ast}\times Z} = - \langle {y^{p\ast}_{h,R}} \mid w^p_{h,R}  \rangle_{{Z}^{\ast}\times Z} - \langle  {w^{p\ast}_{h,R}}  \mid  y^p_{h,R}  \rangle_{{Z}^{\ast}\times Z}.
\end{aligned}
\end{equation}
Thus, \eqref{eq:brokenenergybal2} finally becomes 
\begin{equation}
\begin{aligned}
&\langle e^p_1 \mid f^p_1 \rangle_{K_1} + \langle e^q_1 \mid f^q_1 \rangle_{K_1} + \langle {y^p_1}^{\ast} \mid u^p_1 \rangle_{Z_1^{\ast} \times Z_1} + \langle   {u^p_1}^{\ast} \mid  y^p_1 \rangle_{Z_1^{\ast} \times Z_1}\\
&- \langle {y^{p\ast}_{h,R}} \mid  w^p_{h,R}  \rangle_{{Z}^{\ast}\times Z}- \langle  {w^{p\ast}_{h,R}}  \mid  y^p_{h,R}  \rangle_{{Z}^{\ast}\times Z}=0.
\end{aligned}
\end{equation}
Hence, $D_1 \circ D_c$ is a Dirac structure.
 \end{proof}
Similarly, the generalized Stokes-Dirac structure $D_2$ can be defined as in Definition \ref{def:DGDS} on $F'_{p,q}(K_2)\times E'_{p,q}(K_2)$. Interconnecting $D_1\circ D_c$ and $D_2$, we define the structure $D_1 \circ D_c \circ D_2$ on $F'_{p,q}(K_1\cup K_2)\times E'_{p,q}(K_1 \cup K_2)$
\begin{equation}\label{eq:brokendirac11}
\begin{aligned}
D_1 \circ D_c \circ D_2 = \Big\{&((f^p_1,f^p_2)^T,(f^q_1,f^q_2)^T,(y^p_1\vert_{\Delta_{1\neg2}},\ y^p_2\vert_{\Delta_{2\neg1}})^T,({y^p_1}^{\ast}\vert_{\Delta_{1\neg2}},\ {y^p_2}^{\ast}\vert_{\Delta_{2\neg1}})^T,\\
& (e^p_1,e^p_2)^T,(e^q_1,e^q_2)^T,({u^p_1}^{\ast}\vert_{\Delta_{1\neg2}},{u^p_2}^{\ast}\vert_{\Delta_{2\neg1}})^T,\\
&(u^p_1\vert_{\Delta_{1\neg2}},\ u^p_2\vert_{\Delta_{2\neg1}})^T) \in F'_{p,q}(K_1\cup K_2)\times E'_{p,q}(K_1 \cup K_2) \mid \\
&\exists \ u^p_2\vert_{\Delta_{1\cap2}},u^{p\ast}_2\vert_{\Delta_{1\cap2}},y^p_2\vert_{\Delta_{1\cap2}},y^{p\ast}_2\vert_{\Delta_{1\cap2}} \ \mathrm{s.t.\ with} \ \\
&y^{p\ast}_2 = ({y^{p\ast}_2}\vert_{\Delta_{2\neg1}},\ {y^{p\ast}_2}\vert_{\Delta_{1\cap2}})^T,  y^p_2 = ({y^p_2}\vert_{\Delta_{2\neg1}},\ {y^p_2}\vert_{\Delta_{1\cap2}})^T,\\
&u^{p\ast}_2 = ({u^{p\ast}_2}\vert_{\Delta_{2\neg1}},\ {u^{p\ast}_2}\vert_{\Delta_{1\cap2}})^T,  u^p_2 = ({u^p_2}\vert_{\Delta_{2\neg1}},\ {u^p_2}\vert_{\Delta_{1\cap2}})^T,\\
&\mathrm{holds}\ (f^p_1,f^q_1,(y^p_1\vert_{\Delta_{1\neg2}},\ y^p_2\vert_{\Delta_{1\cap2}})^T,({y^{p\ast}_1}\vert_{\Delta_{1\neg2}},\ y^{p\ast}_2\vert_{\Delta_{1\cap2}})^T,e^p_1,e^q_1, \\ 
&({u^{p\ast}_1}\vert_{\Delta_{1\neg2}},\ u^{p\ast}_2\vert_{\Delta_{1\cap2}})^T,(u^p_1\vert_{\Delta_{1\neg2}},\ u^p_2\vert_{\Delta_{1\cap2}})^T) \in D_1 \circ D_c\  \mathrm{and}\ 
(f^p_2,f^q_2,\\
&y^p_2,y^{p\ast}_2,e^p_2,e^q_2,u^{p\ast}_2, u^p_2) \in D_2.\Big\}
 \end{aligned}
\end{equation}
So when connecting $D_1 \circ D_c$ and $D_2$, we basically take in Figure \ref{fig:Fig 1} $u^{p\ast}_2\vert_{\Delta_{1\cap2}} = w^{p\ast}_{h,R}$, $u^p_2\vert_{\Delta_{1\cap2}} = w^p_{h,R}$, $y^{p\ast}_2\vert_{\Delta_{1\cap2}} = y^{p\ast}_{h,R}$ and $y^p_2\vert_{\Delta_{1\cap2}} = y^p_{h,R}$.\\
The power of $D_1\circ D_c\circ D_2$ is given by 
\begin{equation}\label{eq:brokenenergybal}
\begin{aligned}
\langle e^p_1 \mid f^p_1 \rangle_{K_1} + \langle e^q_1 \mid f^q_1 \rangle_{K_1}+\langle e^p_2 \mid f^p_2 \rangle_{K_2} + \langle e^q_2 \mid f^q_2 \rangle_{K_2}+ \langle {y^p_1}^{\ast} \mid u^p_1 \rangle_{Z_1^{\ast} \times Z_1} \\
+ \langle   {u^p_1}^{\ast} \mid  y^p_1\rangle_{Z_1^{\ast} \times Z_1} + \langle {y^p_2}^{\ast} \mid u^p_2 \rangle_{Z_2^{\ast} \times Z_2} + \langle   {u^p_2}^{\ast} \mid  y^p_2 \rangle_{Z_2^{\ast} \times Z_2}.
\end{aligned}
\end{equation}
\begin{lemma}\label{lem:interconnectionlemma1}
 The structure $D_1 \circ D_c \circ D_2$, \eqref{eq:brokendirac11} with power given by \eqref{eq:brokenenergybal} is a Dirac structure.
 \end{lemma}
 \begin{proof}
 The proof is similar to the proof of Lemma \ref{lem:interconnectionlemma}.
 \end{proof}

\section{Galerkin formulation of the discontinuous Finite element port-Hamiltonian system} \label{sec:DGPHFEM}
After defining the Stokes-Dirac structure for individual elements \eqref{eq:disdirac} and connecting them via the interconnection Dirac structure \eqref{eq:inter} the next step is to formulate the discontinuous Galerkin formulation for the Stokes-Dirac structure on a single element $K \in \mathcal{T}_h$ and then to generalize it to the whole discretized manifold $\Omega$ by connections through the interconnection Dirac structure.
\subsection{DG formulation in port variables}
For element $K \in \mathcal{T}_h$, port variables $e^p_h \in E_p(K),\ e^q_h \in E_q(K) ,\ f^p_h \in F_p(K),\ f^q_h \in F_q(K)$, and external input boundary port variables $ u^p_h \in E_p(\partial K), {u^p_h}^{\ast} \in {E_p(\partial K)}^{\ast}$, the conservation laws stated through the Stokes-Dirac structure \eqref{eq:disdirac} are 
\begin{subequations}
\begin{align}
f^p_h&= (-1)^{r_1}d_{pq}e^q_h +(-1)^{r_1+q}\frac{1}{2} \mathrm{tr}^{\ast}_pQ(\mathrm{tr}_q e^q_h) + (-1)^{r_1+q}\mathrm{tr}^{*}_p {u^p_h}^{\ast} ,\label{eq:diseqa} \\ 
f^q_h &= d_{qp}e^p_h +(-1)^p\frac{1}{2} (Q\mathrm{tr}_q)^{\ast}(\mathrm{tr}_pe^p_h) +(-1)^p(Q\mathrm{tr})^{\ast} u^p_h\label{eq:diseqb}.
\end{align}
\end{subequations}
Note that $p + q = n+1$ and $r_1 = pq +1$.\\ 
Consider two elements, $K_1, K_2 \in \mathcal{T}_h$ connected through the common boundary $\partial K_1 \cap \partial K_2$. Using Lemma \ref{lem:intbyparts}, \eqref{eq:boundarynotation} and the linearity of duality products, for $\nu^p_h \in E_p(K_1), \nu^q_h \in E_q(K_1)$, we get
\begin{equation}\label{eq:weakform3}
\begin{aligned}
\langle \nu^p_h \mid f^p_h \rangle_{K_1} &=- \langle e^q_h \mid d_{qp}\nu^p_h \rangle_{K_1} +(-1)^{p+1} \langle \frac{1}{2}Q(\mathrm{tr}_q e^q_h) - {u^p_h}^{\ast}\mid \mathrm{tr}_p \nu^p_h \rangle_{Z^{\ast}_1 \times Z_1}\\
&\ \ \ \ +(-1)^{p+1}\langle \frac{1}{2}Q(\mathrm{tr}_q e^q_h) - {u^p_h}^{\ast}\mid \mathrm{tr}_p \nu^p_h \rangle_{{Z}^{\ast} \times Z},\\
\langle \nu^q_h\mid f^q_h\rangle_{K_1} &= (-1)^{p}\langle  d_{pq} \nu^q_h \mid e^p_h \rangle_{K_1}+(-1)^{p+1} \langle Q(\mathrm{tr}_q \nu^q_h)\mid \frac{1}{2}\mathrm{tr}_pe^p_h-u^p_h \rangle_{Z^{\ast}_1 \times Z_1}\\
&\ \ \ \ +(-1)^{p+1} \langle Q(\mathrm{tr}_q \nu^q_h)\mid \frac{1}{2}\mathrm{tr}_pe^p_h-u^p_h \rangle_{{Z}^{\ast} \times Z},
\end{aligned}
\end{equation}
and
\begin{equation}\label{eq:weakform4}
\begin{aligned}
\langle \tilde{\nu}^p_h \mid \tilde{f}^p_h \rangle_{K_2} &=- \langle \tilde{e}^q_h \mid d_{qp}\tilde{\nu}^p_h \rangle_{K_2} +(-1)^{p+1} \langle \frac{1}{2}Q(\mathrm{tr}_q \tilde{e}^q_h )- {\tilde{u}_h}^{p\ast}\mid \mathrm{tr}_p \tilde{\nu}^p_h \rangle_{Z^{\ast}_2 \times Z_2}\\
&\ \ \ \ +(-1)^{p+1}\langle \frac{1}{2}Q(\mathrm{tr}_q \tilde{e}^q_h) -{\tilde{u}_h}^{p\ast}\mid \mathrm{tr}_p \tilde{\nu}^p_h \rangle_{{Z}^{\ast} \times Z},\\
\langle \tilde{\nu}^q_h\mid \tilde{f}^q_h\rangle_{K_2} &= (-1)^{p}\langle  d_{pq} \tilde{\nu}^q_h  \mid \tilde{e}^p_h\rangle_{K_2} + (-1)^{p+1} \langle Q(\mathrm{tr}_q \tilde{\nu}^q_h)\mid \frac{1}{2}\mathrm{tr}_p\tilde{e}^p_h-\tilde{u}^p_h\rangle_{Z^{\ast}_2 \times Z_2}\\
&\ \ \ \ +(-1)^{p+1} \langle Q(\mathrm{tr}_q \tilde{\nu}^q_h)\mid \frac{1}{2}\mathrm{tr}_p\tilde{e}^p_h-\tilde{u}^p_h \rangle_{{Z}^{\ast} \times Z}.
\end{aligned}
\end{equation}
With the same argument given in Section \ref{sec:choiceinter} and using Figure \ref{fig:Fig 1} as a reference we choose the port variables defined on the interconnection to be, 
\begin{equation}\label{eq:choice1}
\begin{aligned}
w^p_{h,L} &= u^{p}_h\vert_{\Delta_{1\cap2}}, \ {w^{p\ast}_{h,L}} ={u^{p}_h}^{\ast}\vert_{\Delta_{1\cap2}}, \ w^p_{h,R} =- \tilde{u}^{p}_h\vert_{\Delta_{1\cap2}},\ {w^{p\ast}_{h,R}}=-{\tilde{u}_h}^{p^\ast}\vert_{\Delta_{1\cap2}},\\
y^p_{h,L} &=  \mathrm{tr}_pe^p_h\vert_{\Delta_{1\cap2}}, \ {y^{p\ast}_{h,L}} = Q(\mathrm{tr}_qe^q_h)\vert_{\Delta_{1\cap2}}, \ y^p_{h,R} =  \mathrm{tr}_p\tilde{e}^p_h\vert_{\Delta_{1\cap2}},\\
\ {y^{p\ast}_{h,R}} &= Q(\mathrm{tr}_q\tilde{e}^q_h)\vert_{\Delta_{1\cap2}}.\\
\end{aligned}
\end{equation}
Note the minus sign in front of $w^p_{h,R}$ and ${w^{p\ast}_{h,R}}$ has been introduced by considering the correct direction of energy flow.\\
Using \eqref{eq:inter} we then obtain
\begin{equation}\label{eq:extport}
\begin{aligned}
u^{p}_h\vert_{\Delta_{1\cap2}}&= (\theta - \frac{1}{2})\mathrm{tr}_pe^p_h\vert_{\Delta_{1\cap2}} - \theta   \mathrm{tr}_p\tilde{e}^p_h\vert_{\Delta_{1\cap2}},\\
{u^{p}_h}^{\ast}\vert_{\Delta_{1\cap2}}
&=( \frac{1}{2} - \theta )Q(\mathrm{tr}_qe^q_h\vert_{\Delta_{1\cap2}})+(\theta-1) Q(\mathrm{tr}_q\tilde{e}^q_h\vert_{\Delta_{1\cap2}}),\\
\tilde{u}^{p}_h\vert_{\Delta_{1\cap2}} &=-(1-\theta) \mathrm{tr}_pe^p_h\vert_{\Delta_{1\cap2}} - (\theta -\frac{1}{2})  \mathrm{tr}_p\tilde{e}^p_h\vert_{\Delta_{1\cap2}},\\
{\tilde{u}_h}^{p^\ast}\vert_{\Delta_{1\cap2}}
&=-\theta Q(\mathrm{tr}_qe^q_h\vert_{\Delta_{1\cap2}})- (\frac{1}{2}-\theta) Q(\mathrm{tr}_q\tilde{e}^q_h\vert_{\Delta_{1\cap2}}).\\
\end{aligned}
\end{equation}
Substituting the values for the input ports as given in \eqref{eq:extport}, we can simplify \eqref{eq:weakform3} and \eqref{eq:weakform4} as
\begin{equation}\label{eq:weak1}
\begin{aligned}
\langle \nu^p_h \mid f^p_h \rangle_{K_1} &=-\langle e^q_h \mid d_{qp}\nu^p_h \rangle_{K_1} +(-1)^{p+1} \langle \frac{1}{2}Q(\mathrm{tr}_qe^q_h) - {u^p_h}^{\ast}\mid \mathrm{tr}_p\nu^p_h \rangle_{{Z^{\ast}_1} \times Z_1}\\
&\ \ \ \ +(-1)^{p+1} \langle\theta Q(\mathrm{tr}_qe^q_h)+ (1-\theta)Q(\mathrm{tr}_q\tilde{e}^q_h)\mid \mathrm{tr}_p \nu^p_h \rangle_{{Z}^{\ast} \times Z},\\
\langle \nu^q_h\mid f^q_h\rangle_{K_1} &=(-1)^{p} \langle   d_{pq} \nu^q_h \mid e^p_h \rangle_{K_1} +(-1)^{p+1}  \langle Q(\mathrm{tr}_q \nu^q_h)\mid  \frac{1}{2}\mathrm{tr}_pe^p_h-u^p_h \rangle_{{Z^{\ast}_1} \times Z_1}\\
&\ \ \ \ +(-1)^{p+1}  \langle Q(\mathrm{tr}_q \nu^q_h) \mid (1-\theta)\mathrm{tr}_pe^p_h + \theta\mathrm{tr}_p\tilde{e}^p_h\rangle_{{Z}^{\ast} \times Z},
\end{aligned}
\end{equation}
and
\begin{equation}\label{eq:weak2}
\begin{aligned}
\langle \tilde{\nu}^p_h \mid \tilde{f}^p_h \rangle_{K_2} &=-\langle \tilde{e}^q_h \mid d_{qp}\tilde{\nu}^p_h \rangle_{K_2} +(-1)^{p+1}\langle \frac{1}{2}Q(\mathrm{tr}_q\tilde{e}^q_h) - {\tilde{u}_h}^{p\ast} \mid \mathrm{tr}_p\tilde{\nu}^p_h \rangle_{{Z^{\ast}_2} \times Z_2}\\
&\ \ \ \ +(-1)^{p+1} \langle\theta Q(\mathrm{tr}_qe^q_h)+ (1-\theta)Q(\mathrm{tr}_q\tilde{e}^q_h)\mid \mathrm{tr}_p \tilde{\nu}^p_h \rangle_{{Z}^{\ast} \times Z},\\
\langle \tilde{\nu}^q_h\mid \tilde{f}^q_h\rangle_{K_2} &=(-1)^{p} \langle d_{pq} \tilde{\nu}^q_h \mid  \tilde{e}^p_h  \rangle_{K_2} +(-1)^{p+1}  \langle Q(\mathrm{tr}_q \tilde{\nu}^q_h)\mid \frac{1}{2}\mathrm{tr}_p\tilde{e}^p_h-\tilde{u}^p_h\rangle_{{Z^{\ast}_2} \times Z_2}\\
&\ \ \ \ +(-1)^{p+1}  \langle Q(\mathrm{tr}_q \tilde{\nu}^q_h) \mid (1-\theta)\mathrm{tr}_pe^p_h + \theta\mathrm{tr}_p\tilde{e}^p_h\rangle_{{Z}^{\ast} \times Z}.
\end{aligned}
\end{equation}
As the operator $Q : E_q(\partial K) \rightarrow {E_p(\partial K)}^{\ast}$ is surjective, using \eqref{eq:operatordual}, \eqref{eq:operatordual1}, there exists a $u^q_h \in E_q(\partial K)$ such that ${u^p_h}^{\ast}=Q(u^q_h)$. But, we have not defined the exact definition of the operator $Q$ that can be used for numerical computation so we need to remove $Q$ from our DG formulation. Using $\langle {u^p_h}^{\ast} \mid u^p_h\rangle_{{E_p(\partial K)}^{\ast} \times E_p(\partial K)}= \langle u^q_h \mid u^p_h\rangle_{\partial K}$ in \eqref{eq:weak1} and \eqref{eq:weak2}  and then adding results for the two element system gives
\begin{equation}\label{eq:weak3}
\begin{aligned}
&\underset{K \in \{K_1,K_2\}}{\sum}\langle  \nu^p_h \mid f^p_h \rangle_K  =\underset{K \in \{K_1,K_2\}}{\sum}-\langle e^q_h \mid d_{qp}  \nu^p_h \rangle_K\\
&\ \ \ \ +(-1)^{p+1} \langle \frac{1}{2}\mathrm{tr}_qe^q_h- u^q_h \mid \mathrm{tr}_p\nu^p_h \rangle_{\Delta_{1\neg2}} +(-1)^{p+1}  \langle \frac{1}{2}\mathrm{tr}_q\tilde{e}^q_h - \tilde{u}^q_h \mid \mathrm{tr}_p\tilde{\nu}^p_h \rangle_{\Delta_{2\neg1}}\\
&\ \  \ \ +(-1)^{p+1} \langle\theta \mathrm{tr}_qe^q_h+ (1-\theta) \mathrm{tr}_q\tilde{e}^q_h\mid \mathrm{tr}_p \nu^p_h+\mathrm{tr}_p \tilde{\nu}^p_h\rangle_{\Delta_{1\cap2}},\\
&\underset{K \in \{K_1,K_2\}}\sum\langle \nu^q_h\mid f^q_h\rangle_{K}  = \underset{K \in \{K_1,K_2\}}\sum(-1)^{p}\langle d_{pq}  \nu^q_h\mid  e^p_h\rangle_{K}\\
&\ \ \ \ +(-1)^{p+1}  \langle \mathrm{tr}_q \nu^q_h\mid \frac{1}{2}\mathrm{tr}_pe^p_h-u^p_h \rangle_{\Delta_{1\neg2}} +(-1)^{p+1}\langle \mathrm{tr}_q \tilde{\nu}^q_h\mid \frac{1}{2}\mathrm{tr}_p\tilde{e}^p_h-\tilde{u}^p_h\rangle_{\Delta_{2\neg1}}\\
&\ \ \ \  +(-1)^{p+1} \langle \mathrm{tr}_q\nu^q_h+ \mathrm{tr}_q \tilde{\nu}^q_h\mid (1-\theta)\mathrm{tr}_pe^p_h + \theta\mathrm{tr}_p\tilde{e}^p_h\rangle_{\Delta_{1\cap2}}.
\end{aligned}
\end{equation}
Equation \eqref{eq:weak3} can be generalized to any two elements $K_i$ and $K_j$, connected through a common internal face, $\partial K_i \cap \partial K_j$. For the tessellation $\mathcal{T}_h$ of the manifold $\Omega$, we denote the variables $\mathrm{tr}_q\nu^q_h(K_1)\vert_{\partial K_1\cap \partial K_2}$ as $\mathrm{tr}_q\nu^q_h\vert_L$ (contribution from the left element) and $\mathrm{tr}_q\nu^q_h(K_2)\vert_{\partial K_1\cap \partial K_2}$ as $\mathrm{tr}_q\nu^q_h\vert_R$ (contribution from the right element). \\
We can then add all elements $K$ in the discretised manifold $\mathcal{T}_h$. Let $\mathcal{F}_o$ denoting the set of all external boundaries and $\mathcal{F}_i$ denote the set of all internal boundaries. We obtain the following discontinuous Galerkin formulation on the entire discretized manifold 
 \begin{equation}\label{eq:weak7}
\begin{aligned}
\underset{K \in \mathcal{T}_h}{\sum}\langle  \nu^p_h \mid f^p_h \rangle_K  &=\underset{K \in \mathcal{T}_h}{\sum}-\langle e^q_h \mid d_{qp}  \nu^p_h \rangle_K +(-1)^{p+1} \underset{F \in \mathcal{F}_o}\sum \langle\frac{1}{2}\mathrm{tr}_qe^q_h - u^q_h   \mid \mathrm{tr}_p\nu^p_h \rangle_{F} \\
&\ \ \ \ +(-1)^{p+1} \underset{F \in \mathcal{F}_i}\sum \langle\theta\mathrm{tr}_qe^q_h\vert_L+ (1-\theta) \mathrm{tr}_qe^q_h\vert_R\mid \mathrm{tr}_p \nu^p_h\vert_L+\mathrm{tr}_p \nu^p_h\vert_R\rangle_{F} \\
\underset{K \in \mathcal{T}_h}\sum\langle \nu^q_h\mid f^q_h\rangle_{K}  &= \underset{K \in \mathcal{T}_h}\sum(-1)^{p}\langle  d_{pq}  \nu^q_h  \mid e^p_h \rangle_{K} +(-1)^{p+1}  \underset{F \in \mathcal{F}_o}\sum\langle \mathrm{tr}_q \nu^q_h\mid\frac{1}{2}\mathrm{tr}_pe^p_h-u^p_h\rangle_{F}\\
&\ \ \ \ +(-1)^{p+1} \underset{F \in \mathcal{F}_i}\sum\langle \mathrm{tr}_q\nu^q_h\vert_L + \mathrm{tr}_q \nu^q_h\vert_R\mid (1-\theta)\mathrm{tr}_pe^p_h\vert_L + \theta\mathrm{tr}_pe^p_h\vert_R\rangle_{F}.
\end{aligned}
\end{equation}
We can rewrite \eqref{eq:weak7} in the following manner 
 \begin{equation}\label{eq:weak8}
\begin{aligned}
\underset{K \in \mathcal{T}_h}{\sum}\langle  \nu^p_h \mid f^p_h \rangle_K  &=\underset{K \in \mathcal{T}_h}\sum-\langle e^q_h \mid d_{qp}  \nu^p_h \rangle_K +(-1)^{p+1}\underset{F \in \mathcal{F}_o} \sum \langle \hat{\beta}^q_h \mid \mathrm{tr}_p\nu^p_h\rangle_{F}\\
&\ \ \ \ +(-1)^{p+1} \underset{F \in \mathcal{F}_i}\sum\langle\hat{\beta}^q_h\mid \mathrm{tr}_p \nu^p_h\vert_L+\mathrm{tr}_p \nu^p_h\vert_R\rangle_F\\
\underset{K \in  \mathcal{T}_h}\sum\langle \nu^q_h\mid f^q_h\rangle_{K} &= \underset{K \in  \mathcal{T}_h}\sum (-1)^{p}\langle  d_{pq}  \nu^q_h  \mid e^p_h \rangle_{K} +(-1)^{p+1}  \underset{F \in \mathcal{F}_o} \sum \langle \mathrm{tr}_q \nu^q_h\mid \hat{\beta}^p_h \rangle_F \\
&\ \ \ \ +(-1)^{p+1} \underset{F \in \mathcal{F}_i}\sum\langle \mathrm{tr}_q\nu^q_h\vert_L + \mathrm{tr}_q \nu^q_h\vert_R\mid \hat{\beta}^p_h\rangle_F,
\end{aligned}
\end{equation}
where $\hat{\beta}^q_h$ and $\hat{\beta}^p_h$ can be considered as the numerical fluxes in the discontinuous Galerkin formulation. For $F \in \mathcal{F}_i$ the numerical fluxes are equal to
\begin{equation}\label{eq:flu}
\begin{aligned}
\hat{\beta}^p_h&= (1 - \theta) \mathrm{tr}_pe^{p}_h\vert_L +  \theta \mathrm{tr}_pe^{p}_h\vert_R,\\
\hat{\beta}^q_h&=  \theta \mathrm{tr}_qe^{q}_h\vert_L + ( 1- \theta) \mathrm{tr}_qe^{q}_h\vert_R,\\
\end{aligned}
\end{equation}
 with $\theta \in [0,1]$ and for $F \in \mathcal{F}_o$ the numerical fluxes are equal to 
\begin{equation}\label{eq:flub}
\begin{aligned}
\hat{\beta}^p_h &= \frac{1}{2}\mathrm{tr}_pe^p_h - u^p_h, &\qquad \hat{\beta}^q_h &=  \frac{1}{2}\mathrm{tr}_qe^q_h - u^q_h,
\end{aligned}
\end{equation}
where the input boundary ports $u^p_h$ and $u^q_h$ are to be chosen depending on the external boundary conditions.
\subsection{DG formulation in state space variables}
We start with a pH-system (PDE) represented by the Stokes-Dirac structure \eqref{eq:powbal}, in which the port variables $e^p, e^q, f^p$ and $f^q$ are related to the energy variables $\alpha^p, \alpha^q$ and co-energy variables $\beta^p, \beta^q$ through the relations given by
\begin{equation}
\begin{aligned}
f^p &= -\frac{d \alpha^p}{dt}, &\qquad f^q &= -\frac{d \alpha^q}{dt},\\
e^p &= \beta^p = \delta_p\mathit{H}, &\qquad  e^q &= \beta^q = \delta_q\mathit{H},
\end{aligned}
\end{equation}
where $\mathit{H}$ is the Hamiltonian of the system.\\
Upon discretization let the discretized variables corresponding to the energy and co-energy variables be $\alpha^p_h, \alpha^q_h$ and $\beta^p_h, \beta^q_h$, respectively. The generalized Stokes-Dirac structure \eqref{eq:disdirac} represents the discretized Stokes-Dirac structure over each element, thus, the differential equation expressed in terms of discretized energy and co-energy variables is given by 
\begin{equation}\label{eq:energydirac}
\begin{aligned}
\left[ \begin{array}{c} -\frac{d \alpha^p_h}{dt} \\ -\frac{d \alpha^q_h}{dt} \end{array}\right]&=
 \left[ \begin{array}{cc} 0& (-1)^{r_1}d_{pq} + (-1)^{r_1+q} \frac{1}{2}  \mathrm{tr}^{\ast}_p Q(\mathrm{tr}_q)   \\ d_{qp} + (-1)^p\frac{1}{2}(Q\mathrm{tr}_q)^{\ast} \mathrm{tr}_p & 0 \end{array}\right] \left[ \begin{array}{c} \beta^p_h \\ \beta^q_h\end{array}\right]\\[5pt]
  & \ \ \ \ + \left[ \begin{array}{cc} 0 & (-1)^{r_1+q}\mathrm{tr}^{\ast}_p  \\ (-1)^p(Q\mathrm{tr}_q)^{\ast} & 0 \end{array}\right] \left[ \begin{array}{c} u^{p}_h\\ {u^p_h}^{\ast} \end{array}\right],\\[5pt]
 \left[ \begin{array}{c}  {y^p_h}^{\ast} \\ y^p_h \end{array}\right]  &= \left[ \begin{array}{cc} 0 & -Q(\mathrm{tr}_q) \\ -(-1)^p\mathrm{tr}_p & 0 \end{array}\right] \left[ \begin{array}{c} \beta^p_h  \\ \beta^q_h \end{array}\right],
 \end{aligned}
 \end{equation}
 where $p+q=n+1, r_1=pq+1$ and $Q$ defined as in \eqref{eq:operatordual} and \eqref{eq:operatordual1}.\\
 Here on each element we have used 
 \begin{equation}\label{eq:r1}
 \begin{aligned}
f^p_h &= -\frac{d \alpha^p_h}{dt}, &\qquad f^q_h&= -\frac{d \alpha^q_h}{dt},\\
e^p_h &= \beta^p_h, &\qquad   e^q_h&= \beta^q_h.
\end{aligned}
\end{equation}
Following Section 6.1, as the spaces of the energy and co-energy variables are the same as the corresponding port variables, we can rewrite the weak form \eqref{eq:weak8} in terms of the energy and co-energy variables 
\begin{equation}\label{eq:energyweak}
\begin{aligned}
\underset{K \in \mathcal{T}_h}{\sum}\langle \nu^p_h \mid \dot{\alpha}^p_h \rangle_K  &=\underset{K \in \mathcal{T}_h}\sum\langle \beta^q_h \mid d_{qp}  \nu^p_h \rangle_K +(-1)^{p} \underset{F \in \mathcal{F}_o} \sum \langle \hat{\beta}^q_h \mid \mathrm{tr}_p\nu^p_h\rangle_{F}\\
&\ \ \ \ +(-1)^{p} \underset{F \in \mathcal{F}_i}\sum\langle\hat{\beta}^q_h\mid \mathrm{tr}_p \nu^p_h\vert_L+\mathrm{tr}_p \nu^p_h\vert_R\rangle_F\\
\underset{K \in  \mathcal{T}_h}\sum\langle \nu^q_h\mid \dot{\alpha}^q_h\rangle_{K} &= (-1)^{p+1}\underset{K \in  \mathcal{T}_h}\sum\langle  d_{pq}  \nu^q_h\mid \beta^p_h \rangle_{K} +(-1)^{p}  \underset{F \in \mathcal{F}_o} \sum \langle \mathrm{tr}_q \nu^q_h\mid \hat{\beta}^p_h \rangle_F \\
&\ \ \ \ +(-1)^{p} \underset{F \in \mathcal{F}_i}\sum\langle \mathrm{tr}_q \nu^q_h\vert_L + \mathrm{tr}_q \nu^q_h\vert_R\mid \hat{\beta}^p_h\rangle_F,
\end{aligned}
\end{equation}
where $\hat{\beta^q_h}$ and $\hat{\beta^p_h}$ are the numerical fluxes \eqref{eq:flu}, \eqref{eq:flub} in the discontinuous Galerkin formulation. 
\section{Energy Conservation}\label{sec:Energy stability}
Consider an oriented polyhedral manifold $\Omega$ in $\mathbb{R}^n$ with Lipschitz continuous boundary $\partial \Omega$ discretized into shape regular finite elements $K \in \mathcal{T}_h$. Let $\mathcal{F}_h = \mathcal{F}^o_h \cup \mathcal{F}^i_h$ denote the set of all boundaries of the discretized manifold with $\mathcal{F}^o_h$ and $\mathcal{F}^i_h$ being the set of all external and internal faces, respectively.\\
At a given time $t$, let $\alpha^p(t) \in L^2\Lambda^p(\Omega), \alpha^q(t) \in L^2\Lambda^q(\Omega),$ $ \beta^p(t):=\delta_p \mathit{H} \in H^1\Lambda^{n-p}(\Omega)$ and $\beta^q(t) := \delta_q \mathit{H} \in H\Lambda^{n-q}(\Omega)$ be the solution satisfying the partial differential equations represented by the Dirac structure \eqref{eq:powbal}. The Hamiltonian of this system is given by 
\begin{equation}
\begin{aligned}
H &= \sum_{K \in \mathcal{T}_h}\Big(\langle \beta^p \mid \alpha^p \rangle_K + \langle \beta^q \mid \alpha^q \rangle_K\Big).
\end{aligned}
\end{equation}
We have the following constitutive relations for the port-Hamiltonian system
\begin{equation}\label{eq:constitutive}
\begin{aligned}
\alpha^p = C_p \ast \beta^p,\\
\alpha^q= C_q \ast \beta^q,
\end{aligned}
\end{equation}
where $C_p, C_q \in L^2\Lambda^0(\Omega)$ are coefficient functions that depend on the spatial domain and the physical problem under consideration.\\
Here, $\ast$ is the Hodge star operator defined in Definition \ref{def:Hodgestar}. Using the constitutive relations we can rewrite the Hamiltonian for the system as
\begin{equation}
\begin{aligned}
H &=\sum_{K \in \mathcal{T}_h}\Big(\langle \beta^p \mid C_p\ast \beta^p \rangle_K + \langle \beta^q \mid C_q\ast \beta^q \rangle_K\Big).
\end{aligned}
\end{equation}
Given the discontinuous finite element spaces 
 \begin{equation}\label{eq:discontinuousspace}
 \begin{aligned}
 F_p(\mathcal{T}_h) &:= \{\alpha^p_h \in L^2\Lambda^p(\Omega)\mid \alpha^p_h\vert_K \in F_p(K) \quad \forall K \in \mathcal{T}_h\},\\
 F_q(\mathcal{T}_h) &:= \{\alpha^q_h \in L^2\Lambda^q(\Omega) \mid \alpha^q_h\vert_K \in F_q(K) \quad \forall K \in \mathcal{T}_h\},\\
 E_p(\mathcal{T}_h) &:= \{\beta^p_h \in L^2\Lambda^{n-p}(\Omega) \mid \beta^p_h\vert_K \in E_p(K) \quad \forall K \in \mathcal{T}_h\},\\
 E_q(\mathcal{T}_h) &:= \{\beta^q_h \in L^2\Lambda^{n-q}(\Omega) \mid \beta^q_h\vert_K \in E_q(K) \quad \forall K \in \mathcal{T}_h\}.
 \end{aligned}
 \end{equation}
 Let $\alpha^p_h(t) \in F_p(\mathcal{T}_h), \alpha^q_h(t) \in F_q(\mathcal{T}_h), \beta^p_h(t) \in E_p(\mathcal{T}_h)$ and $\beta^q_h(t) \in E_q(\mathcal{T}_h)$ be the discrete solutions satisfying the discrete Stokes-Dirac structure \eqref{eq:energydirac}. The discrete energy of the system is then given by 
 \begin{equation}\label{eq:discreteenergy}
 \begin{aligned}
 E_h= \underset{K \in \mathcal{T}_h}\sum\langle \beta^p_h \mid \alpha^p_h \rangle_K +  \langle \beta^q_h \mid \alpha^q_h \rangle_K.
 \end{aligned}
 \end{equation}
 \subsection{Discrete Hodge star duality}\label{sec:conhodge}
We extend the $L^2$ inner product between $\beta^p(t), \nu^p \in L^2\Lambda^{n-p}(\Omega)$ and $\beta^q(t), \nu^q \in L^2\Lambda^{n-q}(\Omega)$, given by \eqref{eq:intro}, into a weighted inner product as
\begin{equation}\label{eq:L2innerpro}
\begin{aligned}
\langle \beta^p,\nu^p \rangle_{wL^2\Lambda^{n-p}(\Omega)} =  \int_{\Omega} C_p \ast \beta^p \wedge \nu^p = \int_{\Omega} \beta^p(\mu) \cdot \nu^p(\mu) \cdot C_p(\mu)d\mu,\\
\langle \beta^q,\nu^q \rangle_{wL^2\Lambda^{n-q}(\Omega)} =  \int_{\Omega} C_q \ast \beta^q \wedge \nu^q = \int_{\Omega} \beta^q(\mu) \cdot \nu^q(\mu) \cdot C_q(\mu)d\mu,
\end{aligned}
\end{equation}
where $d\mu$ is the Lebesgue measure on $\Omega$, and the functions $C_p,C_q,$ with $C_p, C_q \geq C >0$, stated in \eqref{eq:constitutive}.\\
Let $e^p_h, \tilde{e}^p_h \in E_p$ and $e^q_h, \tilde{e}^q_h \in E_q$. Analogously, we define a weighted inner product on $E_p$ and $E_q$,
\begin{equation}\label{eq:proinner}
\begin{aligned}
g_{wp}(e^p_h,\tilde{e}^p_h) &= \underset{f \in \Delta(K), \mathrm{\dim}(f)=[p,p+r-1]}\sum \int_f e^p_c(\mu) \cdot \tilde{e}^p_c(\mu) \cdot C_p(\mu) d\mu,\\
g_{wq}(e^q_h, \tilde{e}^q_h) &= \int_K e^q_c(\mu) \cdot \tilde{e}^q_c(\mu) \cdot C_q(\mu) d\mu,\\
\end{aligned}
\end{equation}
where $d\mu$ is the Lebesgue measure on element $K$. Here, $e^p_c, e^q_c$ denote the coefficients for the polynomial differential forms $e^p_h, e^q_h$, respectively. Analogously, $\tilde{e}^p_c, \tilde{e}^q_c$ denote the coefficients for the polynomial differential forms $\tilde{e}^p_h, \tilde{e}^q_h$, respectively. \\
For fixed $\tilde{e}^p_h$ and $\tilde{e}^q_h$, \eqref{eq:proinner} defines a linear functional in $e^p_h$ and $e^q_h$, respectively. As, $E_p$ and $F_p$ are dual to each other, there exists an $f^p_h \in F_p(K)$ such that $ \langle f^p_h \mid e^p_h \rangle_K = g_{wp}(e^p_h, \tilde{e}^p_h),\forall e^p_h \in E_p$. Similarly, there exists an $f^q_h \in F_q(K)$ such that $ \langle f^q_h \mid e^q_h \rangle_K = g_{wq}(e^q_h, \tilde{e}^q_h)$,
$\forall e^q_h \in E_q$. We denote $f^p_h = \star_p \tilde{e}^p_h$ and $f^q_h = \star_q\tilde{e}^q_h$, and call $\star_p$ and $\star_q$ discrete Hodge star operators. Thus
\begin{equation}\label{eq:relationinner}
\begin{aligned}
\langle e^p_h \mid  \star_p \tilde{e}^p_h \rangle_K &= g_{wp}(e^p_h,\tilde{e}^p_h), \quad\forall e^p_h \in E_p(K),\\
\langle e^q_h \mid \star_q \tilde{e}^q_h \rangle_K &= g_{wq}(e^q_h,\tilde{e}^q_h),\quad \forall e^q_h \in E_p(K).
\end{aligned}
\end{equation}
Hence in the discrete framework the constitutive relationship between the state space variables $\alpha^p_h, \beta^p_h$ and $\alpha^q_h, \beta^q_h$ is 
\begin{equation}\label{eq:hsrelation}
\begin{aligned}
\alpha^p_h &= \star_p \beta^p_h,\\
 \alpha^q_h &= \star_q \beta^q_h.
 \end{aligned}
 \end{equation}
 Using \eqref{eq:energyweak} and \eqref{eq:hsrelation} we can now state the port-Hamiltonian DG discretization of \eqref{eq:powbal} for $\beta^p_h(t)\in E_p(\mathcal{T}_h), \beta^q_h(t)\in E_q(\mathcal{T}_h)$ as 
 \begin{equation}\label{eq:finalweak}
\begin{aligned}
\underset{K \in \mathcal{T}_h}{\sum}\langle \nu^p_h \mid \star_{p} \dot{\beta}^p_h \rangle_K  &=\underset{K \in \mathcal{T}_h}\sum\langle \beta^q_h \mid d_{qp}  \nu^p_h \rangle_K +(-1)^{p}\underset{F \in \mathcal{F}_o} \sum \langle \hat{\beta}^q_h\mid \mathrm{tr}_p\nu^p_h\rangle_{F}\\
&\ \ \ \ +(-1)^{p} \underset{F \in \mathcal{F}_i}\sum\langle\hat{\beta}^q_h\mid \mathrm{tr}_p \nu^p_h\vert_L+\mathrm{tr}_p \nu^p_h\vert_R\rangle_F,\\
\underset{K \in  \mathcal{T}_h}\sum\langle \nu^q_h\mid \star_{q} \dot{\beta}^q_h\rangle_{K} &= (-1)^{p+1}\underset{K \in  \mathcal{T}_h}\sum\langle d_{pq}  \nu^q_h\mid  \beta^p_h\rangle_{K} +(-1)^{p}  \underset{F \in \mathcal{F}_o} \sum \langle \mathrm{tr}_q \nu^q_h\mid \hat{\beta}^p_h \rangle_F \\
&\ \ \ \ +(-1)^{p} \underset{F \in \mathcal{F}_i}\sum\langle \mathrm{tr}_q\nu^q_h\vert_L + \mathrm{tr}_q\nu^q_h\vert_R\mid \hat{\beta}^p_h\rangle_F,
\end{aligned}
\end{equation}
where $\hat{\beta}^q_h$ and $\hat{\beta}^p_h$ are given by \eqref{eq:flu} and \eqref{eq:flub} and $\nu^p_h \in E_p(K), \nu^q_h \in E_q(K)$.\\
Using \eqref{eq:hsrelation} along with \eqref{eq:discreteenergy}, the energy on each element of the system represented by the discrete Dirac structure \eqref{eq:energydirac} is given by 
\begin{equation}\label{eq:energyterm1}
\begin{aligned}
E_K = \langle \beta^p_h \mid \star_p\beta^p_h \mid  \rangle _K + \langle \beta^q_h \mid \star_q\beta^q_h \rangle _K.
\end{aligned}
\end{equation}
Choosing the port variables at each common face as stated in \eqref{eq:extport} and using Lemma \eqref{lem:interconnectionlemma1}, the interconnection between the elements $K \in \mathcal{T}_h$ given by \eqref{eq:inter}, is power preserving. Thus, the discrete energy of the total system is
 \begin{equation}\label{eq:energyterm2}
\begin{aligned}
E_h = \underset{K \in \mathcal{T}_h}\sum (\langle \beta^p_h\mid \star_p\beta^p_h \rangle _K + \langle \beta^q_h \mid \star_q\beta^q_h \rangle _K).
\end{aligned}
\end{equation}
\subsection{Energy Conservation}
\begin{theorem}[Energy Conservation]
Given an $n$-dimensional polyhedral oriented manifold $\Omega$ with Lipschitz continuous boundary $\partial \Omega$, discretized into a set of discontinuous elements $\mathcal{T}_h$. Let $\beta^p_h(t) \in E_p(\mathcal{T}_h), \ \beta^q_h(t) \in  E_q( \mathcal{T}_h)$ and
 $\hat{\beta}^p_h(t) \in E_p(\mathcal{F}_h),\ \hat{\beta}^q_h(t) \in E_q(\mathcal{F}_h)$ satisfy \eqref{eq:finalweak}. Let the elements be connected using the interconnection structure \eqref{eq:inter}, choosing the port variables and interconnection variables at each face as \eqref{eq:extport}. Let the total discrete energy of the discrete system $E_h$ be given by \eqref{eq:energyterm2}. Then the rate of change in discrete energy for the whole system is 
\begin{equation}\label{eq:stability}
\begin{aligned}
 \dot E_h &=2 \underset{F \in \mathcal{F}_o} {\sum} \Big(\langle y^q_h\mid u^p_h\rangle_F+  \langle u^q_h \mid y^p_h\rangle_F \Big).
 \end{aligned}
 \end{equation}
 \end{theorem}
 \begin{proof}
 The energy on each element $K \in \mathcal{T}_h$ of the discretised system represented by the Stokes-Dirac structure \eqref{eq:energydirac} is given by \eqref{eq:energyterm1}.
The change in discrete energy on each element $K \in \mathcal{T}_h$ is 
  \begin{equation}
 \begin{aligned}
 \dot{E}_K &= \langle \beta^p_h \mid \star_p \dot{\beta}^p_h\rangle_K + \langle \dot{\beta}^p_h \mid \star_p \beta^p_h \rangle_K + \langle \dot{\beta}^q_h \mid \star_q \beta^q_h\rangle_K + \langle \beta^q_h \mid \star_q \dot{\beta}^q_h\rangle_K,
 \end{aligned}
 \end{equation}
 which using the symmetry of $g_{wp}$ and $g_{wq}$, given by \eqref{eq:proinner}, results in  
 \begin{equation}\label{eq:powerstatespace}
 \begin{aligned}
 \dot{E}_K &= 2 \Big(\langle \beta^p_h \mid \star_p \dot{\beta}^p_h \rangle_K + \langle \beta^q_h \mid \star_q \dot{\beta}^q_h\rangle_K\Big).
 \end{aligned}
 \end{equation}
Using \eqref{eq:r1} and \eqref{eq:hsrelation}, the duality product \eqref{eq:power12} can be written in terms of state space variables as  
\begin{equation}\label{eq:dualitystatespace}
\begin{aligned}
&\langle e^p_h\mid f^p_h\rangle_K + \langle e^q_h \mid f^q_h \rangle_K + \langle {y^p_h}^{\ast} \mid u^p_h\rangle_{{E_p(\partial K)}^{\ast} \times E_p(\partial K)} + \langle {u^p_h}^{\ast} \mid y^p_h\rangle_{{E_p(\partial K)}^{\ast} \times E_p(\partial K)} \\
&=- \langle \beta^p_h \mid \dot{\alpha}^p_h \rangle_K - \langle \beta^q_h \mid \dot{\alpha}^q_h\rangle_K + \langle {y^p_h}^{\ast} \mid u^p_h\rangle_{{E_p(\partial K)}^{\ast} \times E_p(\partial K)}\\
&\ \ \ \ + \langle {u^p_h}^{\ast} \mid y^p_h\rangle_{{E_p(\partial K)}^{\ast} \times E_p(\partial K)}\\
&= -\langle \beta^p_h \mid \star_p \dot{\beta}^p_h \rangle_K - \langle \beta^q_h \mid \star_q \dot{\beta}^q_h\rangle_K + \langle {y^p_h}^{\ast} \mid u^p_h\rangle_{{E_p(\partial K)}^{\ast} \times E_p(\partial K)}\\
&\ \ \ \ + \langle {u^p_h}^{\ast} \mid y^p_h\rangle_{{E_p(\partial K)}^{\ast} \times E_p(\partial K)}.
\end{aligned}
\end{equation}
Next, using the fact that \eqref{eq:energydirac}  with the duality product \eqref{eq:dualitystatespace} is a Dirac structure, Theorem 4.9 then gives 
\begin{equation}\label{eq:dualitystatespace1}
\begin{aligned}
 &-\langle \beta^p_h \mid \star_p \dot{\beta}^p_h \rangle_K - \langle \beta^q_h \mid \star_q \dot{\beta}^q_h\rangle_K + \langle {y^p_h}^{\ast} \mid u^p_h\rangle_{{E_p(\partial K)}^{\ast} \times E_p(\partial K)}\\
 &+ \langle {u^p_h}^{\ast} \mid y^p_h\rangle_{{E_p(\partial K)}^{\ast} \times E_p(\partial K)} = 0.
\end{aligned}
\end{equation}
Using \eqref{eq:dualitystatespace1}, \eqref{eq:powerstatespace} becomes 
 \begin{equation}
 \begin{aligned}
   \dot{E}_K &= 2\Big(\langle {y^p_h}^{\ast} \mid u^p_h\rangle_{{E_p(\partial K)}^{\ast} \times E_p(\partial K)} + \langle {u^p_h}^{\ast} \mid y^p_h\rangle_{{E_p(\partial K)}^{\ast} \times E_p(\partial K)} \Big).
   \end{aligned}
   \end{equation}
   We have already proved that with the power preserving interconnection Dirac structure \eqref{eq:inter} and interconnection variables \eqref{eq:extport}, we can add all energy contributions from the elements and obtain, since the contributions from the internal faces $F \in \mathcal{F}_i$ cancel, see Section 5,
   \begin{equation}
   \begin{aligned}
     \dot E_h &= 2 \underset{F \in \mathcal{F}_o} {\sum} \Big(\langle {y^p_h}^{\ast} \mid u^p_h\rangle_{{E_p(F)}^{\ast} \times E_p(F)}+  \langle {u^p_h}^{\ast} \mid y^p_h\rangle_{{E_p(F)}^{\ast} \times E_p(F)} \Big).
     \end{aligned}
     \end{equation}
     As we have assumed that the operator $Q : E_q (\partial K) \rightarrow {E_p(\partial K)}^{\ast}$, stated in \eqref{eq:operatordual1}, is surjective there exist $y^q_h,u^q_h \in E_q(\partial K)$ such that ${y^p_h}^{\ast} = Q(y^q_h)$ and ${u^p_h}^{\ast} = Q(u^q_h)$, thus 
        \begin{equation}
   \begin{aligned}
     \dot E_h &=2 \underset{F \in \mathcal{F}_o} {\sum} \Big(\langle Q(y^q_h) \mid u^p_h\rangle_{{E_p(F)}^{\ast} \times E_p(F)}+  \langle Q(u^q_h) \mid y^p_h\rangle_{{E_p(F)}^{\ast} \times E_p(F)} \Big)\\
      &=2\underset{F \in \mathcal{F}_o} {\sum} \Big(\langle y^q_h \mid u^p_h\rangle_F+  \langle u^q_h \mid y^p_h\rangle_F \Big).
     \end{aligned}
     \end{equation}
 \end{proof}
 The change in discrete energy for the whole system thus only depends on the energy input-output through the domain boundaries. The choice of proper boundary conditions, or in other words the choice of suitable external input-output boundary port variables for the system, is therefore crucial for the stability of the system. 
 
\section{Error Analysis}\label{sec:error}
In this section we will state an a priori error estimate for the port-Hamiltonian discontinuous Galerkin discretization presented in this paper. For quick reference we summarize in Tables \ref{tab:tab4} and \ref{tab:tab5} the choices for the dual pairs of the spaces $F_q(K), E_q(K)$ and $F_p(K), E_p(K)$ for Cases 1 and 2, respectively. More details can be found in Sections \ref{sec:disspace}, \ref{sec:disdualityelement} and \ref{sec:disdualityboundary}.
\begin{table}[H]
{\footnotesize
\caption{Table 1: Overview of polynomial differential form spaces for Case 1. Note, $p+q = n+1$ and the spaces $E_q, F_q$ and $E_p, F_p$ are dual pairs of polynomial differential form spaces}
\label{tab:tab4}
\begin{center}
\begin{tabular}{|c|c|}

\hline
$E_q(K)$ &$\mathcal{P}_{r+1}\Lambda^{n-q}(K)$\\
\hline
 $F_q(K)$&$\underset{f \in \Delta(K), \mathrm{\dim}(f) \in [n-q,n-q+r]}\oplus\mathcal{P}^{-}_{r+1+n-q-\mathrm{\dim}(f)}\Lambda^{\mathrm{\dim}(f)-n+q}(f)$\\[1ex]
\hline
\hline
$E_p(K)$ & $\underset{f \in \Delta(K), \mathrm{\dim}(f)\in [p,p+r-1]}\oplus\mathcal{P}^{-}_{r+p-\mathrm{\dim}(f)}\Lambda^{\mathrm{\dim}(f)-p}(f)$ \\[1ex]
\hline
$F_p(K)$&$\mathcal{P}_r\Lambda^p(K)$\\ [1ex]
\hline
\end{tabular}
\end{center}
}
\end{table}
\begin{table}[H]
{\footnotesize
\caption{Table 2: Overview of polynomial differential form spaces for Case 2. Note, $p+q=n+1$ and the spaces $E_q, F_q$ and $E_p, F_p$ are dual pairs of polynomial differential form spaces}
\label{tab:tab5}
\begin{center}
\begin{tabular}{|c|c|}
\hline
$E_q(K)$ &$\mathcal{P}^{-}_{r}\Lambda^{n-q}(K)$\\
\hline
 $F_q(K)$&$\underset{f \in \Delta(K), \mathrm{\dim}(f)\in[n-q,n-q+r-1]}\oplus\mathcal{P}_{r+n-q-\mathrm{\dim}(f)-1}\Lambda^{\mathrm{\dim}(f)-n+q}(f)$\\[1ex]
\hline
\hline
$E_p(K)$ &$\underset{f \in \Delta(K), \mathrm{\dim}(f)\in[p,p+r-1]}\oplus\mathcal{P}_{r+p-\mathrm{\dim}(f)-1}\Lambda^{\mathrm{\dim}(f)-p}(f)$\\[1ex]
\hline
$F_p(K)$&$\mathcal{P}^{-}_r\Lambda^p(K)$\\ [1ex]
\hline
\end{tabular}
\end{center}
}
\end{table}
The related discontinuous Galerkin finite element spaces $F_p(\mathcal{T}_h), F_q(\mathcal{T}_h), E_p(\mathcal{T}_h)$ and $E_q(\mathcal{T}_h)$ are stated in \eqref{eq:discontinuousspace}.\\
 Using the trace operator $\mathrm{tr}$ we define the spaces $E_p(\mathcal{F}_h)$ and $E_q(\mathcal{F}_h)$ on the boundaries of each element in the discretized polyhedral manifold as 
 \begin{equation}
 \begin{aligned}
 E_p(\mathcal{F}_h) &:= \{\lambda^p_h \in L^2\Lambda^{n-p}(\Omega)\mid \lambda^p_h\vert_{\partial K} \in E_p(\partial K) \quad \forall K \in \mathcal{T}_h\},\\
 E_q(\mathcal{F}_h) &:= \{\lambda^q_h \in L^2\Lambda^{n-q}(\Omega) \mid \lambda^q_h\vert_{\partial K} \in E_q(\partial K) \quad \forall K \in \mathcal{T}_h\}.
 \end{aligned}
 \end{equation}
 Define the discontinuous Galerkin operators $\mathcal{D}_1$ and $\mathcal{D}_2$ as 
\begin{subequations}\label{eq:disoperator}
\begin{align}
&\mathcal{D}_1 : (H \Lambda^{n-q}(\Omega)+E_q(\mathcal{T}_h)) \times E_p(\mathcal{T}_h) \times (H^{-1/2}\Lambda^{n-q}(\mathcal{F}_h) +E_q(\mathcal{F}_h))\rightarrow \mathbb{R},\nonumber\\
&\mathcal{D}_1(u^q,\nu^p_h;\hat{u}^q)= -\sum_{K \in \mathcal{T}_h}  \langle u^q \mid d_{qp}\nu^p_h \rangle_K \label{eq:dop1}\\
& \ \ \ \ \ \ \ \ \ \ \ \ \ \ \ \ \ \ \ \ \ \ \ \ \ \ +(-1)^{p+1} \sum_{F \in \mathcal{F}_i} \langle \hat{u}^q \mid {\mathrm{tr}_p\nu^p_h}\vert_L + {\mathrm{tr}_p\nu^p_h}\vert_R  \rangle_{F}\nonumber \\
& \ \ \ \ \ \ \ \ \ \ \ \ \ \ \ \ \ \ \ \ \ \ \ \ \ \ +(-1)^{p+1}\sum_{F \in \mathcal{F}_o} \langle \hat{u}^q \mid {\mathrm{tr}_p\nu^p_h}\rangle_{F},\nonumber\\
&\mathcal{D}_2 : (H^1 \Lambda^{n-p}(\Omega)+E_p(\mathcal{T}_h)) \times E_q(\mathcal{T}_h) \times (H^{1/2}\Lambda^{n-p}(\mathcal{F}_h) + E_p(\mathcal{F}_h))\rightarrow \mathbb{R},\nonumber\\
&\mathcal{D}_2(v^p,\nu^q_h;\hat{v}^p)=  (-1)^{p}\sum_{K \in \mathcal{T}_h}\langle  d_{pq} \nu^q_h \mid v^p\rangle_K\label{eq:dop2}\\
& \ \ \ \ \ \ \ \ \ \ \ \ \ \ \ \ \ \ \ \ \ \ \ \ \ \ +(-1)^{p+1} \sum_{F \in \mathcal{F}_i} \langle{\mathrm{tr}_q\nu^q_h}\vert_L + {\mathrm{tr}_q\nu^q_h}\vert_R \mid \hat{v}^p \rangle_{F}\nonumber \\
& \ \ \ \ \ \ \ \ \ \ \ \ \ \ \ \ \ \ \ \ \ \ \ \ \ \ +(-1)^{p+1}\sum_{F \in \mathcal{F}_o} \langle{\mathrm{tr}_q\nu^q_h} \mid \hat{v}^p \rangle_{F}\nonumber.
\end{align}
\end{subequations}
\begin{lemma}[Energy Conservation]\label{lem:stab}
The discontinuous Galerkin (DG) operators $\mathcal{D}_1$ and $\mathcal{D}_2$, as defined in \eqref{eq:disoperator}, for $\beta^p_h \in E_p(\mathcal{T}_h)$, $\beta^q_h  \in E_q(\mathcal{T}_h)$ and external input $u^p_h, y^p_h \in  E_p(\mathcal{F}_h)$, $u^q_h, y^q_h \in E_q(\mathcal{F}_h)$ satisfy the following relation 
\begin{equation}
\begin{aligned}
&\mathcal{D}_1(\beta^q_h,\beta^p_h; \hat{\beta}^q_h) + \mathcal{D}_2(\beta^p_h,\beta^q_h; \hat{\beta}^p_h) = - \underset{F \in \mathcal{F}_o} {\sum} \Big(\langle y^q_h \mid u^p_h\rangle_F+  \langle u^q_h \mid y^p_h\rangle_F \Big),
\end{aligned}
\end{equation}
where $p + q = n + 1$.
\end{lemma}
\begin{proof}
The proof follows from \eqref{eq:energyweak} with numerical fluxes \eqref{eq:flu}-\eqref{eq:flub} and uses the same steps as in the proof of Theorem 7.1.\\
\end{proof}
\begin{definition}[Canonical Projection Operator $\pi^l_h$]\label{def:pro}%
Let $K \in \mathcal{T}_h$ be a simplex and let $\Delta(K)$ denote the set of all subsimplices $f$ of $K$ with $\mathrm{\dim}(f)\ge l$. For $r \in \mathbb{N}$, $r \ge 1$, the projection operator $\pi^l_h$ is the mapping, (Section 4.9 , \cite{arnold200finite}),
\begin{equation}\label{eq:ccp1}
\begin{aligned}
\pi^l_{h} : C^0\Lambda^l(K) &\rightarrow \mathcal{P}_r\Lambda^l(K),\\
\end{aligned}
\end{equation}
satisfying, for all $\omega^l\in C^0\Lambda^l(K)$
\begin{equation}\label{eq:ccp2}%
\begin{aligned}%
\int_{f} \mathrm{tr}_{K,f}(\omega^l - \pi^l_{h}\omega^l) \wedge \eta &= 0, \quad \forall \eta \in \mathcal{P}^{-}_{r+l-\dim(f)}\Lambda^{\dim(f)-l}(f), \quad \forall f \in \Delta(K).
\end{aligned}
\end{equation}
\end{definition}%
\begin{definition}[Canonical Projection Operator $\pi^l_{h-}$]\label{def:pro1}%
 Let $K \in \mathcal{T}_h$ be a simplex and let $\Delta(K)$ denote the set of all subsimplices $f$ of $K$ with $\mathrm{\dim}(f)\ge l$. For $r \in \mathbb{N}$, $r \ge 1$ , the projection operator $\pi^l_{h-}$ is the mapping, Section 4.9, \cite{arnold200finite}, 
\begin{equation}\label{eq:ccp3}
\begin{aligned}
\pi^l_{h-} : C^0\Lambda^l(K) &\rightarrow \mathcal{P}^{-}_r\Lambda^l(K),
\end{aligned}
\end{equation}
satisfying, for all $\omega^l \in C^0\Lambda^l(K)$
\begin{equation}\label{eq:ccp4}%
\begin{aligned}%
&\int_{f} \mathrm{tr}_{K,f}(\omega^l - \pi^l_{h-}\omega^l) \wedge \eta &= 0, \quad \forall \eta \in \mathcal{P}_{r+l-\mathrm{\dim}(f)-1}\Lambda^{\mathrm{\dim}(f)-l}(f), \quad \forall f \in \Delta(K).
\end{aligned}
\end{equation}
\end{definition}
\subsection{Interpolation Error Bounds}
The canonical projections stated in Definitions \ref{def:pro} and \ref{def:pro1} depend on traces on the subsimplices. This requires function spaces with more regularity than is available in the spaces $\mathcal{F}_{p,q}$ and $\mathcal{E}_{p,q}$, \eqref{eq:spacea} and \eqref{eq:spacec}, which are used in the definition of the Stokes-Dirac structure stated in Section 3.2. This can be remediated by modifying the canonical projections into smoothed projections, which we also denote as $\pi^l_h$ and $\pi^l_{h,-}$. The details of the construction of these smoothed projections can be found in Section 5.4, \cite{arnold200finite} and Section 5.5, \cite{arnold2006finite}. Based on Theorem 5.9, \cite{arnold2006finite}, we can state the following lemmas
\begin{lemma}[Interpolation Error Bounds]\label{lem:error}
 Let $\pi^l_{h} : L^2\Lambda^l(\Omega) \rightarrow \mathcal{P}_r\Lambda^l(\mathcal{T}_h), r \ge 1$ and $\pi^l_{h,-} : L^2\Lambda^l(\Omega) \rightarrow \mathcal{P}^{-}_{r+1}\Lambda^l(\mathcal{T}_h), r \ge 0$ be, respectively, the smoothed projection operators based on the canonical projections stated by Definitions \ref{def:pro} and \ref{def:pro1}. Then $\pi^l_{h}$ is a projection onto $\mathcal{P}_r\Lambda^l(\mathcal{T}_h)$ for $r \ge 1$ and satisfies
\begin{equation}\label{eq:bound11}
\begin{aligned}
|| \omega - \pi^l_{h} \omega||_{L^2\Lambda^l(\Omega)} \leq ch^s||\omega||_{H^s\Lambda^l(\Omega)} , \quad \omega \in H^s\Lambda^l(\Omega),
\end{aligned}
\end{equation}
for $0 \leq s\leq r+1$. Moreover, for all $\omega \in L^2\Lambda^l(\Omega)$, $ \pi^l_{h} \omega \rightarrow \omega$ in $L^2\Lambda^l(\Omega)$ as $h\rightarrow0$ and $d\pi^{l-1}_h=\pi^{l}_hd$. Similarly, $\pi^l_{h,-}$ is a projection onto $\mathcal{P}_{r+1}\Lambda^l(\mathcal{T}_h)$ for $r \ge 0$ and also satisfies \eqref{eq:bound11} for $0 \leq s \leq r+1$. Moreover, for all $\omega \in L^2\Lambda^l(\Omega)$, $ \pi^l_{h,-} \omega \rightarrow \omega$ in $L^2\Lambda^l(\Omega)$ as $h\rightarrow0$ and $d\pi^{l-1}_{h,-}=\pi^{l}_{h,-}d$.
\end{lemma}
\begin{lemma}\label{lem:consistency}
The discontinuous Galerkin discretization operators $\mathcal{D}_1$ and $\mathcal{D}_2$, as defined in \eqref{eq:disoperator} for $ \beta^p \in H^1\Lambda^{n-p}(\Omega), \hat{\beta}^p \in H^{\frac{1}{2}}\Lambda^{n-p}(\mathcal{F}_h), \beta^q \in H\Lambda^{n-q}(\Omega), \hat{\beta}^q \in H^{-\frac{1}{2}}\Lambda^{n-q}(\mathcal{F}_h),\nu^p_h \in E_p(\mathcal{T}_h), \nu^q_h \in E_q(\mathcal{T}_h)$, along with the smoothed canonical projection operators defined in Definitions \ref{def:pro} and \ref{def:pro1}, satisfy
\begin{subequations}
\begin{equation}\label{eq:operatordef1}
\begin{aligned}
\mathcal{D}_1(\beta^q-\pi^{n-q}_h\beta^q, \nu^p;\beta^q-\pi^{n-q}_h\beta^q)&=0,\\
\end{aligned}
\end{equation}
\begin{equation}\label{eq:operatordef2}
\begin{aligned}
\mathcal{D}_2(\beta^p-\pi^{n-p}_{h-}\beta^p, \nu^q;\beta^p-\pi^{n-p}_{h-}\beta^p)&=0.
\end{aligned}
\end{equation}
\end{subequations}
\end{lemma}
\begin{proof}
Using Definition \ref{def:pro} for the projection operator $\pi^l_h$ together with \eqref{eq:dop1}, \eqref{eq:derivative2}, \eqref{eq:operatordef1} is immediate. Similarly, using Definition \ref{def:pro1} for the projection operator $\pi^l_{h-}$ together with \eqref{eq:dop2}, \eqref{eq:derivative1} gives \eqref{eq:operatordef2}.
\end{proof}
\subsection{Energy equation}
For the error analysis of the port-Hamiltonian discontinuous Galerkin finite element formulation \eqref{eq:energyweak} with numerical fluxes \eqref{eq:flu} - \eqref{eq:flub}, we define the following bilinear forms.\\
For $u^p \in (H^1\Lambda^{n-p}(\Omega)+E_p(\mathcal{T}_h)),v^q\in (H\Lambda^{n-q}(\Omega)+E_q(\mathcal{T}_h)), \nu^p \in E_p(\mathcal{T}_h), \nu^q \in E_q(\mathcal{T}_h)$ and corresponding boundary port variables $\hat{u}^p \in (H^{1/2}\Lambda^{n-p}(\mathcal{F}_h) + E_p(\mathcal{F}_h)),$\\$\hat{v}^q \in (H^{-1/2}\Lambda^{n-q}(\mathcal{F}_h)+E_q(\mathcal{F}_h))$ we define the bilinear forms
\begin{subequations}
\begin{align}
\mathcal{A}(u^p,v^q;\nu^p,\nu^q) &=\underset{K \in \mathcal{T}_h}\sum\Big(\langle \nu^p \mid C_p\ast u^p \rangle_K+ \langle  \nu^q \mid C_q\ast v^q\rangle_K\Big)\nonumber\\
&= \underset{K \in \mathcal{T}_h}\sum(\int_K u^p(\mu) \cdot \nu^p(\mu) \cdot C_p(\mu)d\mu \\
&\ \ \ \ \ \ \ \ + \int_K v^q(\mu) \cdot \nu^q(\mu) \cdot C_q(\mu)d\mu) \label{eq:biforms1a}\\
\mathcal{B}(u^p,v^q,\hat{u}^p,\hat{v}^q;\nu^p,\nu^q) &= \mathcal{D}_1(v^q,\nu^p; \hat{v}^q) + \mathcal{D}_2(u^p,\nu^q; \hat{u}^p), \label{eq:biforms1b}
\end{align}
\end{subequations}
with $C_p$ and $C_q$ the coefficients in the constitutive relation between the energy and co-energy variables, \eqref{eq:constitutive}.\\
In the discrete setting $\mathcal{A}$ satisfies the following relation
\begin{lemma}
For $u^p_h,\nu^p \in E_p(\mathcal{T}_h), v^q_h,\nu^q\in E_q(\mathcal{T}_h)$ and corresponding boundary port variables $\hat{u}^p_h \in E_p(\mathcal{F}_h), \hat{v}^q_h \in E_q(\mathcal{F}_h)$ the bilinear forms satisfy
\begin{equation}
\begin{aligned}
\mathcal{A}(u^p_h,v^q_h;\nu^p,\nu^q) &=\underset{K \in \mathcal{T}_h}\sum\Big(\langle \nu^p \mid C_p\ast u^p \rangle_K+ \langle  \nu^q \mid C_q\ast v^q\rangle_K\Big)\nonumber\\
&=\underset{K \in \mathcal{T}_h}\sum\ \underset{f \in \Delta(K), \mathrm{\dim}(f)\in[p,r+p-1]}\sum\int_f u^p_h(\mu) \cdot \nu^p(\mu) \cdot C_p(\mu)\vert_f d\mu\nonumber\\
&\ \ \ \ + \underset{K \in \mathcal{T}_h}\sum\int_K v^q_h(\mu) \cdot \nu^q(\mu) \cdot C_q(\mu)\vert_Kd\mu\label{eq:biforms2}
\end{aligned}
\end{equation}
with $C_p$ and $C_q$ the coefficients in the constitutive relation between the energy and co-energy variables, \eqref{eq:constitutive}.
\end{lemma}
\begin{proof}
This relation is immediate using \eqref{eq:funcsp2} for $E_p(K)$.
\end{proof}
\begin{lemma}\label{lem:biliniearbal}
The DG formulation \eqref{eq:finalweak} with $\beta^p(t) \in H^1\Lambda^{n-p}(\Omega)$ and $\beta^q(t) \in H \Lambda^{n-q}(\Omega)$ the exact solution of the partial differential equations represented by the Dirac structure \eqref{eq:powbal}, satisfies $\forall \nu^p \in E_p(\mathcal{T}_h), \forall \nu^q \in E_q(\mathcal{T}_h)$ 
\begin{equation}\label{eq:energy1}
\begin{aligned}
\mathcal{A}(\dot{\beta}^p,\dot{\beta}^q;\nu^p,\nu^q) + \mathcal{B}(\beta^p,\beta^q,\beta^p,\beta^q;\nu^p,\nu^q) = 0,
\end{aligned}
\end{equation}
and the orthogonality condition 
\begin{equation}\label{eq:errorbiform}
\begin{aligned}
\mathcal{A}(\dot{\beta}^p-\dot{\beta}^p_h,\dot{\beta}^q-\dot{\beta}^q_h;\nu^p,\nu^q) + \mathcal{B}(\beta^p - \beta^p_h,\beta^q- \beta^q_h,\beta^p - \hat{\beta}^p_h,\beta^q- \hat{\beta}^q_h;\nu^p,\nu^q) = 0.
\end{aligned}
\end{equation}
\end{lemma}

\subsection{A priori error estimate for port-Hamiltonian DG discretization}
In this section we will prove an a priori error estimate for the port-Hamiltonian discontinuous Galerkin discretization.
\begin{theorem}[Error Estimate]
 Given the port-Hamiltonian discontinuous Galerkin formulation \eqref{eq:finalweak} with \eqref{eq:flu}-\eqref{eq:flub}, which is based on the generalized Stokes-Dirac structure \eqref{eq:disdirac}, with numerical solutions $\beta^p_h(t) \in E_p(\mathcal{T}_h)$ and $\beta^q_h(t) \in  E_q(\mathcal{T}_h)$ for $t \in (0,T]$. Assume that the exact solutions $\beta^p(t) \in H^1\Lambda^{n-p}(\Omega) $ and $\beta^q(t) \in H\Lambda^{n-q}_h(\Omega)$ with $t \in (0,T] $ are sufficiently smooth, then we have the following a priori error estimate
 \begin{equation}\label{eq:finalest}
\begin{aligned}
 &|| \beta^p(t) - \beta^p_h(t) ||^2_{L^2\Lambda^{n-p}(\Omega)} + || \beta^q(t) - \beta^q_h(t) ||^2 _{L^2\Lambda^{n-q}(\Omega)} \\
&\leq C h^{2(r+1)}\Big(e^{2\epsilon T}(||\beta^p(0)||^2_{H^{r+1}\Lambda^{n-p}(\Omega)}+||\beta^q(0)||^2_{H^{r+1}\Lambda^{n-q}(\Omega)} )\\
&\ \ \ \ + \int^T_0 e^{2\epsilon(T -t)} (||\dot{\beta}^p(t)||^2_{H^{r+1}\Lambda^{n-p}(\Omega)} +||\dot{\beta}^q(t)||^2_{H^{r+1}\Lambda^{n-q}(\Omega)})dt\Big),
 \end{aligned}
\end{equation}
with $C$ and $\epsilon$ strictly positive constants independent of $\beta^p, \beta^q$ and the mesh size $h$.
\end{theorem}
\begin{proof}
For Case 1 in Section 4.2, the error contributions satisfy 
\begin{equation}\label{eq:error}
\begin{aligned}
er^p &= \beta^p - \beta^p_h = \beta^p - \pi^{n-p}_{h-}\beta^p + \pi^{n-p}_{h-}er^p,\\
er^q &= \beta^q - \beta^q_h = \beta^q - \pi^{n-q}_h\beta^q + \pi^{n-q}_her^q,
\end{aligned}
\end{equation}
 where we used here the properties $\pi^{n-p}_{h-}\beta^p_h = \beta^p_h$ and $\pi^{n-q}_h\beta^q_h = \beta^q_h$ of the projection operators $\pi^{n-p}_{h-}$ and $\pi^{n-q}_h$, which follow directly from Definitions \ref{def:pro} and \ref{def:pro1}. Note, for Case 2 in Section 4.2 the projection operator $\pi_{h,-}$ and $\pi^h$ need to be interchanged in the proof of Theorem 8.8.\\
Introducing the test forms as $ \nu^p  = \pi^{n-p}_{h-}er^p \in E_p(\mathcal{T}_h)$ and $ \nu^q = \pi^{n-q}_her^q \in E_q(\mathcal{T}_h)$, into \eqref{eq:errorbiform} gives, 
\begin{equation}\label{eq:estimateq}
\begin{aligned}
&\mathcal{A}(\dot{\beta}^p-\dot{\beta}^p_h,\dot{\beta}^q-\dot{\beta}^q_h; \pi^{n-p}_{h-}er^p,\pi^{n-q}_her^q)\\
&+ \mathcal{B}(\beta^p - \beta^p_h,\beta^q- \beta^q_h, \beta^p - \hat{\beta}^p_h,\beta^q- \hat{\beta}^q_h;\pi^{n-p}_{h-}er^p,\pi^{n-q}_her^q) = 0.
\end{aligned}
\end{equation}
Using the linearity of $\mathcal{A}$ and \eqref{eq:error}, we obtain 
\begin{equation}\label{eq:alinearity}
\begin{aligned}
&\mathcal{A}(\dot{\beta}^p-\dot{\beta}^p_h,\dot{\beta}^q-\dot{\beta}^q_h; \pi^{n-p}_{h-}er^p,\pi^{n-q}_her^q)\\
&= \mathcal{A}(\dot{\beta}^p-\pi^{n-p}_{h-}\dot{\beta}^p,\dot{\beta}^q-\pi^{n-q}_h\dot{\beta}^q;\pi^{n-p}_{h-}er^p,\pi^{n-q}_her^q)\\
&\ \ \ + \mathcal{A}(\pi^{n-p}_{h-}\dot{er}^p,\pi^{n-q}_h\dot{er}^q;\pi^{n-p}_{h-}er^p,\pi^{n-q}_her^q).
\end{aligned}
\end{equation}
Also, using the linearity of $\mathcal{D}_1, \mathcal{D}_2$ and $\mathcal{B}$ with $\pi^{n-q}_h\hat{\beta}^q, \pi^{n-q}_h\hat{er}^q$ and $\pi^{n-p}_{h-}\hat{\beta}^p,$\\
$\pi^{n-p}_{h-}\hat{er}^p$ the projections on the faces $f \in \partial\Delta(K)$, we obtain
\begin{equation}\label{eq:break3}
\begin{aligned}
&\mathcal{B}(\beta^p - \beta^p_h,\beta^q- \beta^q_h,\beta^p-\hat{\beta}^p_h,\beta^q-\hat{\beta}^q_h;\pi^{n-p}_{h-}er^p,\pi^{n-q}_her^q)\\
&=\mathcal{B}(\beta^p-\pi^{n-p}_{h-}\beta^p,\beta^q-\pi^{n-q}_h\beta^q,\beta^p-\pi^{n-p}_{h-}\hat{\beta}^p_h,\beta^q-\pi^{n-q}_h\hat{\beta}^q_h;\pi^{n-p}_{h-}er^p,\pi^{n-q}_her^q)\\
&\ \ \ \ + \mathcal{B}(\pi^{n-p}_{h-}er^p,\pi^{n-q}_her^q,\pi^{n-p}_{h-}\hat{er}^p,\pi^{n-q}_h\hat{er}^q;\pi^{n-p}_{h-}er^p,\pi^{n-q}_her^q).
\end{aligned}
\end{equation}
Substituting \eqref{eq:alinearity} and \eqref{eq:break3} into \eqref{eq:estimateq}, we obtain 
\begin{equation}\label{eq:norm}
\begin{aligned}
&\mathcal{A}(\dot{\beta}^p-\pi^{n-p}_{h-}\dot{\beta}^p,\dot{\beta}^q-\pi^{n-q}_h\dot{\beta}^q;\pi^{n-p}_{h-}er^p,\pi^{n-q}_her^q)\\
&\ \ \  + \mathcal{A}(\pi^{n-p}_{h-}\dot{er}^p,\pi^{n-q}_h\dot{er}^q;\pi^{n-p}_{h-}er^p,\pi^{n-q}_her^q)\\
&\ \ \ + \mathcal{B}(\beta^p-\pi^{n-p}_{h-}\beta^p,\beta^q-\pi^{n-q}_h\beta^q, \beta^p-\pi^{n-p}_{h-}\hat{\beta}^p_h,\beta^q-\pi^{n-q}_h\hat{\beta}^q_h;\pi^{n-p}_{h-}er^p,\pi^{n-q}_her^q) \\
&\ \ \  +\mathcal{B}(\pi^{n-p}_{h-}er^p,\pi^{n-q}_her^q,\pi^{n-p}_{h-}\hat{er}^p,\pi^{n-q}_h\hat{er}^q;\pi^{n-p}_{h-}er^p,\pi^{n-q}_her^q)\\
&\ \ \  = 0.
\end{aligned}
\end{equation}
Using Lemma 8.6 and \eqref{eq:L2innerpro} we obtain,  
\begin{equation}\label{eq:norm1}
\begin{aligned}
 &\mathcal{A}(\pi^{n-p}_{h-}\dot{er}^p,\pi^{n-q}_h\dot{er}^q;\pi^{n-p}_{h-}er^p,\pi^{n-q}_her^q) \\
 &=\sum_{K \in \mathcal{T}_h} \Big(\underset{f \in \Delta(K), \mathrm{\dim}(f)\in[p,r+p-1]}\sum\int_f \pi^{n-p}_{h-}\dot{er}^p(\mu) \cdot \pi^{n-p}_{h-}er^p(\mu) \cdot  C_p(\mu)\vert_f d\mu \\
 &\ \ \ \ \ \ \ \ \ + \int_K \pi^{n-q}_h\dot{er}^q(\mu) \cdot \pi^{n-q}_her^q(\mu) \cdot  C_q(\mu)d\mu\Big)\\
 &= \frac{1}{2}\frac{d}{dt}\Big(||\pi^{n-p}_{h-}er^p||^2_{wL^2\Lambda^{n-p}(\Omega)} + ||\pi^{n-q}_her^q||^2_{wL^2\Lambda^{n-q}(\Omega)}\Big).
 \end{aligned}
 \end{equation}
Using Lemma \ref{lem:consistency}, we obtain
\begin{equation}\label{eq:norm5}
\begin{aligned}
 &\mathcal{B}(\beta^p-\pi^{n-p}_{h-}\beta^p,\beta^q-\pi^{n-q}_h\beta^q, \beta^p-\pi^{n-p}_{h-}\hat{\beta}^p_h,\beta^q-\pi^{n-q}_h\hat{\beta}^q_h;\pi^{n-p}_{h-}er^p,\pi^{n-q}_her^q)\\
 & =\mathcal{D}_1(\beta^q- \pi^{n-q}_h\beta^q,\pi^{n-p}_{h-}er^p;\beta^q- \pi^{n-q}_h\hat{\beta}^q)\\
 &\ \ \ \ + \mathcal{D}_2(\beta^p- \pi^{n-p}_{h-}\beta^p,\pi^{n-q}_h er^ q;\beta^p- \pi^{n-p}_{h-}\hat{\beta}^p)\\
 &=0.
 \end{aligned}
 \end{equation}
Lemma \ref{lem:stab} gives 
 \begin{equation}\label{eq:norm3}
\begin{aligned}
\mathcal{B}(\pi^{n-p}_{h-}er^p,\pi^{n-q}_her^q,\pi^{n-p}_{h-}\hat{er}^p,\pi^{n-q}_h\hat{er}^q;\pi^{n-p}_{h-}er^p,\pi^{n-q}_her^q)\\
= -\underset{F \in \mathcal{F}_o}\sum\Big(\langle y^q_h\mid u^p_h\rangle_F + \langle u^q_h\mid y^p_h\rangle_F),
 \end{aligned}
 \end{equation}
 where $y^p_h, u^p_h \in E_p(\mathcal{F}_h)$ and $y^q_h, u^q_h \in E_q(\mathcal{F}_h)$. Using the definition of the external boundary port variables $y^p_h$ and $y^q_h$ as stated in the generalized Stokes-Dirac structure \eqref{eq:disdirac} with $y^{p\ast}_h=Q(y^q_h)$, we can rewrite \eqref{eq:norm3} as
  \begin{equation}\label{eq:norm03}
\begin{aligned}
&\mathcal{B}(\pi^{n-p}_{h-}er^p,\pi^{n-q}_her^q,\pi^{n-p}_{h-}\hat{er}^p,\pi^{n-q}_h\hat{er}^q;\pi^{n-p}_{h-}er^p,\pi^{n-q}_her^q) \\
&= \underset{F \in \mathcal{F}_o}\sum\Big(\langle \mathrm{tr}_q\pi^{n-q}_her^q\mid u^p_h\rangle_F + (-1)^{p} \langle u^q_h\mid \mathrm{tr}_p\pi^{n-p}_{h-}er^p \rangle_F).
 \end{aligned}
 \end{equation}
Using $\pi^{n-p}_{h-}er^p = \pi^{n-p}_{h-}\beta^p - \beta^p_h$ and $\pi^{n-q}_her^q = \pi^{n-q}_h\beta^q - \beta^q_h$, we obtain
 \begin{equation}\label{eq:boundaryerror}
 \begin{aligned}
 \mathrm{tr}_p \pi^{n-p}_{h-}er^p &= \mathrm{tr}_p \pi^{n-p}_{h-}\beta^p - \mathrm{tr}_p\beta^p_h,\\
 \mathrm{tr}_q \pi^{n-q}_her^q  &= \mathrm{tr}_q \pi^{n-q}_h\beta^q - \mathrm{tr}_q\beta^q_h.
 \end{aligned}
 \end{equation}
Assume that the boundary conditions are applied exactly, then \eqref{eq:boundaryerror}\\
gives \ $ \mathrm{tr}_p \pi^{n-p}_{h-}er^p  =0$ and $ \mathrm{tr}_q \pi^{n-q}_her^q =0$. Thus,
\begin{equation}\label{eq:Berror2}
\begin{aligned}
\mathcal{B}(\pi^{n-p}_{h-}er^p,\pi^{n-q}_her^q,\pi^{n-p}_{h-}\hat{er}^p,\pi^{n-q}_h\hat{er}^q;\pi^{n-p}_{h-}er^p,\pi^{n-q}_her^q)=0.
\end{aligned}
\end{equation}
Furthermore, we have
\begin{equation}\label{eq:norm6}
\begin{aligned}
&\mathcal{A}(\dot{\beta}^p-\pi^{n-p}_{h-}\dot{\beta}^p,\dot{\beta}^q-\pi^{n-q}_h\dot{\beta}^q;\pi^{n-p}_{h-}er^p,\pi^{n-q}_her^q) \\
&=( \dot{\beta}^p - \pi^{n-p}_{h-}\dot{\beta^p}  \mid  \pi^{n-p}_{h-}er^p)_{wL^2\Lambda^{n-p}(\Omega)}+ (  \dot{\beta}^q - \pi^{n-q}_h\dot{\beta^q}\mid \pi^{n-q}_her^q)_{wL^2\Lambda^{n-q}(\Omega)}.
\end{aligned}
\end{equation}
Using Cauchy's inequality with $\epsilon>0$, we can simplify \eqref{eq:norm6} as
\begin{equation}\label{eq:norm7}
\begin{aligned}
&\mathcal{A}(\dot{\beta}^p-\pi^{n-p}_{h-}\dot{\beta}^p,\dot{\beta}^q-\pi^{n-q}_h\dot{\beta}^q;\pi^{n-p}_{h-}er^p,\pi^{n-q}_her^q)\\
&\leq \frac{1}{4\epsilon}||\dot{\beta}^p-\pi^{n-p}_{h-}\dot{\beta}^p||^2_{wL^2\Lambda^{n-p}(\Omega)} + \epsilon ||\pi^{n-p}_{h-}er^p||^2_{wL^2\Lambda^{n-p}(\Omega)} \\
& \ \ \ \  + \frac{1}{4\epsilon}||\dot{\beta}^q-\pi^{n-q}_h\dot{\beta}^q||^2_{wL^2\Lambda^{n-q}(\Omega)} + \epsilon ||\pi^{n-q}_her^q||^2_{wL^2\Lambda^{n-q}(\Omega) }.\\
\end{aligned}
\end{equation}
Using Lemma \ref{lem:error}, the following bounds hold 
\begin{equation}
\begin{aligned}
||\dot{\beta}^p-\pi^{n-p}_{h-}\dot{\beta}^p||^2_{wL^2\Lambda^{n-p}(\Omega)} &\leq Ch^{2(r+1)}||\dot{\beta}^p(t)||^2_{H^{r+1}\Lambda^{n-p}(\Omega)},\\
||\dot{\beta}^q-\pi^{n-q}_h\dot{\beta}^q||^2_{wL^2\Lambda^{n-q}(\Omega)} &\leq Ch^{2(r+1)}||\dot{\beta}^q(t)||^2_{H^{r+1}\Lambda^{n-q}(\Omega)},\\
\end{aligned}
\end{equation}
where $r$ is the lowest order of the polynomial differential form spaces $E_p,E_q$ and\\
$h$ is the element size. Finally, \eqref{eq:norm7} simplifies to 
\begin{equation}\label{eq:norm8}
\begin{aligned}
&\mathcal{A}(\dot{\beta}^p-\pi^{n-p}_{h-}\dot{\beta}^p,\dot{\beta}^q-\pi^{n-q}_h\dot{\beta}^q;\pi^{n-p}_{h-}er^p,\pi^{n-q}_her^q)\\
&\leq Ch^{2(r+1)}\Big(||\dot{\beta}^p(t)||^2_{H^{r+1}\Lambda^{n-p}(\Omega)} +||\dot{\beta}^q(t)||^2_{H^{r+1}\Lambda^{n-q}(\Omega)}\\
&\ \ \ \ \ \ + \epsilon (||\pi^{n-p}_{h-}er^p||^2_{wL^2\Lambda^{n-p}(\Omega)} +||\pi^{n-q}_her^q||^2_{wL^2\Lambda^{n-q}(\Omega) })\Big).
\end{aligned}
\end{equation}
Using \eqref{eq:norm}, with \eqref{eq:norm1}, \eqref{eq:norm5}, \eqref{eq:Berror2} and \eqref{eq:norm8}, gives 
  \begin{equation}\label{eq:norm10}
  \begin{aligned}
 &\frac{1}{2}\frac{d}{dt}\Big(||\pi^{n-p}_{h-}er^p||^2_{wL^2\Lambda^{n-p}(\Omega)} + ||\pi^{n-q}_her^q||^2_{wL^2\Lambda^{n-q}(\Omega)}\Big)\\
&\leq Ch^{2(r+1)}\Big((||\dot{\beta}^p(t)||^2_{H^{r+1}\Lambda^{n-p}(\Omega)}  + ||\dot{\beta}^q(t)||^2_{H^{r+1}\Lambda^{n-q}(\Omega)})\\
  &\ \ \ \ +  \epsilon (||\pi^{n-p}_{h-}er^p||^2_{wL^2\Lambda^{n-p}(\Omega)} + ||\pi^{n-q}_her^q||^2_{wL^2\Lambda^{n-q}(\Omega)})\Big).
  \end{aligned}
\end{equation}
Using Gronwall's inequality in differential form we obtain
\begin{equation}
\begin{aligned}
&||\pi^{n-p}_{h-}er^p||^2_{wL^2\Lambda^{n-p}(\Omega)} + ||\pi^{n-q}_her^q||^2_{wL^2\Lambda^{n-q}(\Omega)} \leq e^{2\epsilon T}  (||\pi^{n-p}_{h-}er^p(0)||^2_{wL^2\Lambda^{n-p}(\Omega)}\\
&\ \ \ \ +||\pi^{n-q}_her^q(0)||^2_{wL^2\Lambda^{n-q}(\Omega)}) + Ch^{2(r+1)}\int_0^T e^{2\epsilon(T -t)} (||\dot{\beta}^p(t)||^2_{H^{r+1}\Lambda^{n-p}(\Omega)}\\
&\ \ \ \ +||\dot{\beta}^q(t)||^2_{H^{r+1}\Lambda^{n-q}(\Omega)})dt.
\end{aligned}
\end{equation}
Let $\beta^p_h(x,0) = \pi^{n-p}_{h-}\beta^p(x,0)$ and $\beta^q_h(x,0)=\pi^{n-q}_h \beta^q(x,0)$ for all $x \in \Omega$, be the projections of the initial conditions to the discontinuous Galerkin finite element spaces. Using the fact that $\pi^{n-p}_{h,-}$ and $\pi^{n-q}_h$ are projections and $||\pi_h u||_{L^2} = ||u||_{L^2}$ gives \eqref{eq:finalest}.
\end{proof}
\section{Results}\label{sec:result}
To support our theory we apply the formulation for the DG discretization of linear port-Hamiltonian systems, presented in Section 6, to the scalar wave equation.\\ 
For an $n$-dimensional oriented manifold $\Omega$ the scalar wave equation is given in vector notation by 
\begin{equation}\label{eq:pdewave}
\begin{aligned}
\frac{d^2u}{dt^2}= \frac{1}{\mu} \nabla \cdot (Y \nabla u) \quad \mathrm{in} \ \Omega,
\end{aligned}
\end{equation}
along with appropriate boundary conditions at the boundary $\partial \Omega$. Here $u \in \mathbb{R}^n$ is the displacement, $\mu \in \mathbb{R}$ the mass density and $Y \in \mathbb{R}^{n \times n}$ the modulus of rigidity, $\nabla$ the nabla operator and $t$ time. In a port-Hamiltonian formulation \eqref{eq:pdewave} is written as, see \cite{talasila2002wave, van2014port, van2002hamiltonian, golo2004hamiltonian},
 \begin{equation}\label{eq:conswave}
\begin{aligned}
\left[ \begin{array}{c} \dot{\rho} \\  \dot{\epsilon} \\ \end{array} \right] & =  \begin{bmatrix} 0  & (-1)^{n+2} d \\  -d & 0 \\ \end{bmatrix} \left[ \begin{array}{c}  V \\ \sigma  \\ \end{array} \right] \ ,\\
\left[ \begin{array}{c} V^b \\  \sigma^b \\ \end{array} \right] & =  \begin{bmatrix} -1 & 0 \\ 0& (-1)^{n-1} \\ \end{bmatrix} \left[ \begin{array}{c}  \mathrm{tr}(V) \\  \mathrm{tr}(\sigma)\\ \end{array} \right]. \ 
\end{aligned}
\end{equation}
Here, the kinetic momentum $\rho= \mu \frac{du}{dt} \in L^2\Lambda^n(\Omega)$ and elastic strain $\epsilon= du \in L^2\Lambda^1(\Omega)$ are the energy variables. The velocity $V= \frac{du}{dt} \in H^1\Lambda^0(\Omega)$ and elastic stress $\sigma= (Y \ast du)\in H\Lambda^{n-1}(\Omega)$ are the co-energy variables. Note, $d$ is the exterior derivative for differential forms and $\ast$ the Hodge star operator.\\
 The Hamiltonian energy for the wave equation is given by 
  \begin{equation}\label{eq:hamwave}
\begin{aligned}
 H = \frac{1}{2} \int_\Omega(\epsilon \wedge \sigma + \rho \wedge V ).
  \end{aligned}
  \end{equation}
 In $2D$ the energy variables $\rho$ is a $2$-form and $\epsilon$ a $1$-form, whereas the co-energy variable $V$ is a $0$-form and $\sigma$ a $1$-form. So, we take $p=2, q=1, n=2$ and consider the domain $\Omega  = [-1,1] \times [-1,1]$.\\
We choose the discrete energy variables as $V_h \in E_p(\mathcal{T}_h), \sigma_h \in E_q(\mathcal{T}_h), \rho_h \in F_p(\mathcal{T}_h)$ and $\epsilon \in F_q(\mathcal{T}_h)$. Following the procedure discussed in Section 6, see also \eqref{eq:finalweak}, we obtain the following port-Hamiltonian discontinuous Galerkin (pHDG) formulation of the wave equation: Find $V_h \in E_p(\mathcal{T}_h), \sigma_h \in E_q(\mathcal{T}_h)$ such that forall $ \nu^p \in E_p(\mathcal{T}_h), \nu^q \in E_q(\mathcal{T}_h)$,
\begin{equation}\label{eq:actualweak}
\begin{aligned}
\underset{K \in \mathcal{T}_h}{\sum}\langle \nu^p \mid \star_{p} \dot{V}_h \rangle_K  &=\underset{K \in \mathcal{T}_h}\sum\langle \sigma_h \mid d_{qp}  \nu^p_h \rangle_K +(-1)^p\underset{F \in \mathcal{F}_o} \sum \langle \hat{\sigma}_h\mid \mathrm{tr}_p\nu^p_h\rangle_{F}\\
&\ \ \ \ + (-1)^p\underset{F \in \mathcal{F}_i}\sum\langle\hat{\sigma}_h\mid (\mathrm{tr}_p \nu^p_h\vert_L+\mathrm{tr}_p \nu^p_h\vert_R)\rangle_F,\\
\underset{K \in  \mathcal{T}_h}\sum\langle \nu^q \mid \star_{q} \dot{\sigma}_h\rangle_{K} &= (-1)^{p+1}\underset{K \in  \mathcal{T}_h}\sum\langle d_{pq}  \nu^q_h\mid  V_h\rangle_{K} + (-1)^p \underset{F \in \mathcal{F}_o} \sum \langle \mathrm{tr}_q \nu^q_h\mid \hat{V}_h \rangle_F \\
&\ \ \ \ + (-1)^p \underset{F \in \mathcal{F}_i}\sum\langle (\mathrm{tr}_q\nu^q_h\vert_L + \mathrm{tr}_q\nu^q_h\vert_R)\mid \hat{V}_h\rangle_F,
\end{aligned}
\end{equation}
where for $F \in \mathcal{F}_i$
\begin{equation}\label{eq:flux}
\begin{aligned}
\hat{V}_h&= (1 - \theta) \mathrm{tr}_pV_h\vert_L +  \theta \mathrm{tr}_pV_h\vert_R,\\
\hat{\sigma}_h&=  \theta \mathrm{tr}_q\sigma_h\vert_L + ( 1- \theta) \mathrm{tr}_q\sigma_h\vert_R,\\
\end{aligned}
\end{equation}
with $\theta \in [0,1]$ and for $F \in \mathcal{F}_o$,
\begin{equation}\label{eq:flux1}
\begin{aligned}
\hat{V}_h &= \mathrm{tr}_pV_h, &\qquad \hat{\sigma}_h &= \mathrm{tr}_q\sigma_h.
\end{aligned}
\end{equation}
 As model problem we use the exact solution for the displacement $u$, velocity $V$ and elastic stress $\sigma$ given by  
\begin{equation}\label{eq:testingvalues}
\begin{aligned}
u &= \frac{1}{2\pi} \sin(2\pi t)(\sin(2\pi x) + \sin(2\pi y)),\\
V&=  \cos(2\pi t)(\sin(2\pi x) + \sin(2\pi y)),\\
\sigma &= \sin(2\pi t)(\cos(2\pi x)dx + \cos(2\pi y)dy),
\end{aligned}
\end{equation}
with Dirichlet's boundary conditions for $V$ and $\sigma$ obtained by restricting $(x,y)$ in \eqref{eq:testingvalues} to $\partial \Omega$. The mass density and modulus of rigidity are taken as 1.\\
Using $p=2, q=1$ in the expression for the spaces $E_p$ and $E_q$, shown in Table 1, the finite element space $V_h$ identifies with the space of Lagrange finite elements of order $r$ and $\sigma_h$ with the space of Brezzi-Douglas-Marini finite elements of order $r+1$. The time integration is done with a 4th order explicit Runge-Kutta method.\\
The $L^2$ and $L^{\infty}$ errors were obtained for 5 values of theta, i.e., $\theta \in \{ 0,\frac{1}{3},\frac{1}{2},\frac{2}{3},1\}$. The results are shown in Tables \ref{tab:table3}, \ref{tab:table4} and \ref{tab:table5}
\begin{table}[H]
{\footnotesize
\caption{Table 3: Order of accuracy of the pH-DGFEM for the scalar wave equation, $V_h \in \mathcal{P}^-_0\Lambda^0(K)$ and $\sigma_h \in \mathcal{P}_1\Lambda^1(K)$}
\label{tab:table3}
\centering
\begin{tabular}{|c|c|c|c|c|c|c|c|c|c|}
\hline
& & \multicolumn{4} {|c|} {$L^2$error}&\multicolumn{4} {|c|} {$L^{\infty}$error}\\
\hline
$\theta$& h & Velocity & Order & Stress & Order & Velocity& Order  & Stress & Order \\
\hline
\multirow{3}{1 em}{0}& 0.0625 & 0.184760 & - - & 0.001473 & - -  & 2.1484 e-06 & - - & 0.001930 & - - \\  [1ex]
&0.03125& 0.092513 &0.9979&  7.4556e-04 & 0.9826 &   1.0338e-06 &   1.055 &  9.7758e-04 &0.9717\\ [1ex]
&0.015625& 0.046273  & 0.9994 &3.7454e-04&0.9932 &  5.0462e-07  &1.034 & 4.9035e-04 &0.9954\\ [1ex]
\hline
\multirow{3}{1 em}{1/3}&  0.0625  & 0.184760 & - - & 9.7839e-04 & - -  & 9.8968e-07 & - - & 0.001256 & - - \\  [1ex]
&0.03125 &0.092513 &0.9979&  4.9335e-04 & 0.9877 &   4.7913e-07 &   1.029&6.3797e-04&0.9718\\ [1ex]
&0.015625 & 0.046273  & 0.9994 & 2.4735e-04 &0.9960 & 2.3464e-07  &1.046 & 3.2022e-04 &0.9929\\ [1ex]
\hline
\multirow{3}{1 em}{1/2}& 0.0625  & 0.184760 & - - & 0.036180 & - -  & 9.8968e-07 & - - & 0.001252 & - - \\  [1ex]
&0.03125 &0.092513 &0.9979&  4.9335e-04 & 0.9877 &   4.7913e-07 &   1.029 & 6.3741e-04 &0.9718\\ [1ex]
&0.015625 & 0.046273  & 0.9994 & 2.4735e-04 &0.9960 & 2.3464e-07  &1.046 & 3.2016e-04 &0.9929\\ [1ex]
\hline
\multirow{3}{1 em}{2/3}& 0.0625  & 0.184760 & - - & 9.7839e-04 & - -  & 9.8968e-07 & - - & 0.001249 & - - \\  [1ex]
&0.03125 &0.092513 &0.9979&  4.9335e-04 &  0.9877 &  4.7913e-04 &   1.029  & 6.3711e-04&0.9718\\ [1ex]
&0.015625& 0.046273  & 0.9994 & 2.4735e-04 & 0.9960 & 2.3464e-07  &1.046 & 3.2011e-04 &0.9929\\ [1ex]
\hline
\multirow{3}{1 em}{1}& 0.0625 & 0.184760 & - - &  0.001473 & - -  & 2.1484e-06 & - - & 0.001930 & - - \\  [1ex]
&0.03125 &0.092513 &0.9979&  7.4556e-04 & 0.9826 &   1.0338e-06 &   1.055& 9.8423e-04&0.9717\\ [1ex]
&0.015625 & 0.046273  & 0.9994 & 3.7454e-04 &0.9932 & 5.0462e-07  &1.034 & 4.9553e-04 &0.9900\\ [1ex]
\hline
\end{tabular}
}
\end{table}
\begin{table}[H]
{\footnotesize
\caption{Table 4: Order of accuracy of the pH-DGFEM for the scalar wave equation, $V_h \in \mathcal{P}^-_1\Lambda^0(K)$ and $\sigma_h \in \mathcal{P}_2\Lambda^1(K)$}
\label{tab:table4}
\centering
\begin{tabular}{|c|c|c|c|c|c|c|c|c|c|}
\hline
& & \multicolumn{4} {|c|} {$L^2$error}&\multicolumn{4} {|c|} {$L^{\infty}$error}\\
\hline
$\theta$& h & Velocity & Order & Stress & Order & Velocity& Order  & Stress & Order \\
\hline
\multirow{3}{1 em}{0}& 0.0625 & 0.008886 & - - & 1.021e-04 & - -  & 1.211e-06 & - - & 5.9490e-04 & - - \\  [1ex]
&0.03125& 0.002224 &1.9980&  2.5540e-05 & 1.9995 &   3.0701e-07 &   1.9802 &   1.5026e-04 &1.9851\\ [1ex]
&0.015625& 5.5639e-04  & 1.9995 &6.3882e-06&1.9992 &  7.7017e-08  &1.9950& 3.7662e-05 &1.9962\\ [1ex]
\hline
\multirow{3}{1 em}{1/3}&  0.0625  &  0.008886 & - - & 7.8863e-05 & - -  & 3.7039e-07 & - - & 4.8877e-04 & - - \\  [1ex]
&0.03125 & 0.002224 &1.9980&  1.9734e-05 & 1.9986 &   9.3191e-08 &   1.9908& 1.2328e-04&1.9871\\ [1ex]
&0.015625 & 5.5639e-04  & 1.9995 & 4.9349e-06 &1.9996 & 2.3333e-08  &1.9977 & 3.0890e-05 &1.9967\\ [1ex]
\hline
\multirow{3}{1 em}{1/2}& 0.0625  & 0.008886 & - - & 7.5472e-05& - -  & 1.8361e-07 & - - & 4.3570e-04 & - - \\  [1ex]
&0.03125 &0.002224 &1.9980&  1.8886e-05 & 1.9985 &   4.6103e-08 &   1.9937 & 1.0979e-04 &1.9884\\ [1ex]
&0.015625 & 5.5639e-04  & 1.9995 &  4.7222e-06 &1.9998 & 1.1555e-08  &1.9962 & 2.7504e-05 &1.9971\\ [1ex]
\hline
\multirow{3}{1 em}{2/3}& 0.0625  & 0.008886 & - - & 7.8915e-05 & - -  & 3.5899e-07 & - - & 4.7640e-04 & - - \\  [1ex]
&0.03125 &0.002224 &1.9980&  1.9741e-05 &  1.9990 &  9.2462e-08 &   1.9570  & 1.2250e-04&1.9593\\ [1ex]
&0.015625& 5.5639e-04  & 1.9995 & 4.9357e-06 & 1.9998 &  2.3288e-08 &1.9892 & 3.0841e-05 &1.9898\\ [1ex]
\hline
\multirow{3}{1 em}{1}& 0.0625 & 0.008886 & - - & 1.0224e-04 & - -  &1.1824e-06 & - - & 5.8582e-04 & - - \\  [1ex]
&0.03125 &0.002224 &1.9980&  2.5555e-05 & 2.0003 &   3.0517e-07 &   1.9540& 1.4969e-04&1.9684\\ [1ex]
&0.015625 & 5.5639e-04  & 1.9995 & 6.3901e-06 &1.9997 & 7.6901e-08 &1.9885 & 3.7626e-05 &1.9921\\ [1ex]
\hline
\end{tabular}
}
\end{table}
\begin{table}[H]
{\footnotesize
\caption{Table 5: Order of accuracy of the pH-DGFEM for the scalar wave equation, $V_h \in \mathcal{P}^-_2\Lambda^0(K)$ and $\sigma_h \in \mathcal{P}_3\Lambda^1(K)$}
\label{tab:table5}
\centering
\begin{tabular}{|c|c|c|c|c|c|c|c|c|c|}
\hline
& & \multicolumn{4} {|c|} {$L^2$error}&\multicolumn{4} {|c|} {$L^{\infty}$error}\\
\hline
$\theta$& h & Velocity & Order & Stress & Order & Velocity& Order  & Stress & Order \\
\hline
\multirow{3}{1 em}{0}&  0.0625   &   2.8790e-04 & - - &5.7062e-06 & - -  & 1.7594e-07 & - - &  6.7039e-04 & - - \\  [1ex]
& 0.03125 &  3.6031e-05 & 2.9982 &7.2018e-07 & 2.9861& 2.2527e-08 &  2.9653 &     8.4647e-05 & 2.9854 \\ [1ex]
&0.015625    &  4.5053e-06  & 2.9995  &9.0376e-08  & 2.9943 & 2.8328e-09 & 2.9913 & 1.0607e-05& 2.9964 \\ [1ex]
\hline
\multirow{3}{1 em}{1/3}&  0.0625   &   2.8790e-04 & - - &3.4595e-06 & - -  & 7.1348e-08 & - - &  4.1759e-04 & - - \\  [1ex]
& 0.03125 &  3.6031e-05 & 2.9982 &4.3152e-07 & 3.0030& 9.1496e-09 &  2.9630 &     5.2720e-05 & 2.9856 \\ [1ex]
&0.015625    &  4.5053e-06  & 2.9995  &5.3888e-08  & 3.0014 &  1.1510e-09 & 2.9907 & 6.6060e-06& 2.9964 \\ [1ex]
\hline
\multirow{3}{1 em}{1/2}&  0.0625   &   2.8790e-04 & - - &3.0651e-06 & - -  & 4.6243e-08 & - - & 3.7029e-04 & - - \\  [1ex]
& 0.03125 &  3.6031e-05 & 2.9982 &3.8035e-07 & 3.0105& 5.9055e-09 &  2.9691 &     4.6908e-05 & 2.9807\\ [1ex]
&0.015625    &  4.5053e-06  & 2.9995  &4.7392e-08  & 3.0046 & 7.4213e-10 & 2.9923 & 5.8831e-06& 2.9952\\ [1ex]
\hline
\multirow{3}{1 em}{2/3}&  0.0625   &   2.8790e-04 & - - & 3.4600e-06 & - -  & 7.2013e-08 & - - &  4.2727e-04 & - - \\  [1ex]
& 0.03125 &  3.6031e-05 & 2.9982 &4.3155e-07 & 3.0031& 9.1709e-09 &  2.9731 &      5.3029e-05 & 3.0103 \\ [1ex]
&0.015625    &  4.5053e-06  & 2.9995  &5.3890e-08  & 3.0014 & 1.1516e-09 & 2.9933 & 6.6157e-06& 3.0028 \\ [1ex]
\hline
\multirow{3}{1 em}{1}&  0.0625   &   2.8790e-04 & - - &5.7071e-06 & - -  & 1.7496e-07 & - - &  6.8605e-04 & - - \\  [1ex]
& 0.03125 &  3.6031e-05 & 2.9982 &7.2023e-07 & 2.9862& 2.2492e-08 &  2.9594 &     8.5146e-05 & 3.0103 \\ [1ex]
&0.015625    &  4.5053e-06  & 2.9995  &9.0380e-08  & 2.9944 & 2.8317e-09 & 2.9896 & 1.0622e-05& 3.0027 \\ [1ex]
\hline
\end{tabular}
}
\end{table}
Referring to Tables 3, 4 and 5, we observe that the $L^2$ and $L^{\infty}$ errors are restricted by the polynomial order $r$ of the space $V_h$, resulting in convergence order $r+1$. Note that although $\sigma_h$ belongs to an $r+1$-th order finite element space the convergence order is $r+1$ and not $r+2$. The reason being that in the port-Hamiltonian formulation the conservation laws are coupled, hence the convergence rate of both the co-energy variables is restricted by the convergence rate of the co-energy variable $V$, which is discretized using an $r$-th order accurate finite element space. We also observe that the numerical results are consistent with the a priori error bounds, given in \eqref{eq:finalest}.
\section{Acknowledgement}
We thank the Shell-NWO/FOM PhD-75 program, project 15CSER49,  for giving us the opportunity and funding to work on this project


\bibliographystyle{amsplain}
\bibliography{PHDGFEM_arxiv}
\end{document}